\documentclass{elsarticle}
\usepackage[utf8]{inputenc}

\usepackage{tikz-cd}
\usepackage{subfig}

\usepackage{graphicx}
\usepackage{caption}
\usepackage{xspace}
\usepackage{color}
\usepackage{subfig}
\usepackage{epstopdf}% To incorporate .eps illustrations using PDFLaTeX, etc.
\usepackage{comment}

\newcommand{\revOne}[1]{{\color{black}#1}}
\newcommand{\revTwo}[1]{{\color{black}#1}}

\usepackage{amsmath,amssymb}
\usepackage{dsfont}
\usepackage{bm}
\usepackage{mathrsfs}
\usepackage{amsthm}
\usepackage{a4wide}

\usepackage{multirow}
\usepackage{diffcoeff}
\usepackage{arydshln,mathtools}% http://ctan.org/pkg/{arydshln,leftidx,mathtools}
\usepackage[most]{tcolorbox}

\usepackage[numbers]{natbib}% Citation support using natbib.sty
\newtheorem{definition}{Definition}
\newtheorem{remark}{Remark}
\newtheorem{proposition}{Proposition}

\newtheorem{corollary}{Corollary}
\newtheorem{assumption}{Assumption}

\renewcommand\d{\ensuremath{\mathrm{d}}}

\newcommand{\bbR}{\mathbb{R}}

\DeclareMathOperator{\tr}{tr}

\newcommand*{\dual}[1]{\ensuremath{\widehat{#1}}}

\newcommand{\inpr}[3][]{\ensuremath{( #2, \, #3 )_{#1}}}

\newcommand{\dualpr}[3][]{\ensuremath{\langle #2 \, \vert #3 \rangle_{#1}}}

\DeclareMathOperator*{\grad}{grad}

\renewcommand{\div}{\operatorname{div}}

\DeclareMathOperator*{\curl}{curl}

\newcommand{\fenics}{\textsc{FEniCS}\xspace}
\newcommand{\firedrake}{\textsc{Firedrake}\xspace}

\journal{arXiv}

\graphicspath{{./images/}}

\begin{document}

\begin{frontmatter}	
		%% Title, authors and addresses
		
		%% use the tnoteref command within \title for footnotes;
		%% use the tnotetext command for theassociated footnote;
		%% use the fnref command within \author or \address for footnotes;
		%% use the fntext command for theassociated footnote;
		%% use the corref command within \author for corresponding author footnotes;
		%% use the cortext command for theassociated footnote;
		%% use the ead command for the email address,
		%% and the form \ead[url] for the home page:
		%% \title{Title\tnoteref{label1}}
		%% \tnotetext[label1]{}
		%% \author{Name\corref{cor1}\fnref{label2}}
		%% \ead{email address}
		%% \ead[url]{home page}
		%% \fntext[label2]{}
		%% \cortext[cor1]{}
		%% \address{Address\fnref{label3}}
		%% \fntext[label3]{}	
		\title{Dual field structure-preserving discretization of port-Hamiltonian systems using finite element exterior calculus}	
		%% use optional labels to link authors explicitly to addresses:
		%% \author[label1,label2]{}
		%% \address[label1]{}
		%% \address[label2]{}
		\author[UT]{Andrea Brugnoli\corref{cor1}}
		\ead{a.brugnoli@utwente.nl}
		
		\cortext[cor1]{Corresponding author}
	    \author[UT]{Ramy Rashad}
	    \ead{r.a.m.rashadhashem@utwente.nl}
        \author[UT]{Stefano Stramigioli}
        \ead{s.stramigioli@utwente.nl}

	    \address[UT]{Robotics and Mechatronics Department, University of Twente, The Netherlands}

		\begin{abstract}
	
	In this paper we propose a novel approach to discretize linear port-Hamiltonian systems while preserving the underlying structure. We present a finite element exterior calculus formulation that is able to mimetically represent conservation laws and cope with mixed open boundary conditions using a single computational mesh.
	The possibility of including open boundary conditions allows for modular composition of complex multi-physical systems whereas the exterior calculus formulation provides a coordinate-free treatment.
	Our approach relies on a dual-field representation of the physical system that is redundant at the continuous level but eliminates the need of mimicking the Hodge star operator at the discrete level.
	By considering the Stokes-Dirac structure representing the system together with its adjoint, which  embeds the metric information directly in the codifferential, the need for an explicit discrete Hodge star is avoided altogether.
	By imposing the boundary conditions in a strong manner, the power balance characterizing the Stokes-Dirac structure is then retrieved at the discrete level via symplectic Runge-Kutta integrators based on Gauss-Legendre collocation points. Numerical experiments validate the convergence of the method and the conservation properties in terms of energy balance both  for the wave and Maxwell equations in a three dimensional domain. For the latter example, the magnetic and electric fields preserve their divergence free nature at the discrete level.

		\end{abstract}
		
		\begin{keyword}
			%% keywords here, in the form: keyword \sep keyword
			%% PACS codes here, in the form: \PACS code \sep code
			%% MSC codes here, in the form: \MSC code \sep code
			%% or \MSC[2008] code \sep code (2000 is the default)
			 Port-Hamiltonian systems \sep Structure preserving discretization \sep Finite element exterior calculus \sep de Rham complex \sep Dual field representation 
		\end{keyword}

	\end{frontmatter}
	
\section{Introduction}

The field of computational engineering seeks to advance the reliability and accuracy of simulations by means of faster and/or higher fidelity algorithms. Complex engineering systems arise from the interconnection of many simple components. For this reason, it is important that computational models be also constructed via a modular approach, where each subsystem is modelled separately and then assembled to the rest (as in dynamic substructuring). Furthermore, even more important is the fact that such models can be reusable and are to a big extent independent of the super-model they will be part of. This idea is at the core of the port-Hamiltonian (pH) framework, since it describes the interaction between systems and the surrounding environment by means of ports (cf. \cite{maschke1993pH} for the first paper discussing pH finite dimensional systems and their connection with bond-graph modelling). In the infinite dimensional case the mathematical literature has been mostly focused on Hamiltonian systems  with no interaction with the environment (see \cite{olver1986applications} for a reference on the subject). However in the seminal work \cite{vanderSchaft2002}, distributed pH systems were first introduced to model physical systems of conservation laws (which happen to be described by hyperbolic PDEs due to the physical finite propagation time). In order to account for the non-trivial boundary interaction, the authors introduced the concept of Stokes-Dirac structure. This geometrical structure represents an extension of Dirac manifolds \cite{courant1990} to the infinite dimensional case. The Stokes theorem allows defining appropriate boundary variables to account for the power exchange through the boundary of the spatial domain. This geometrical structure is acausal, in the sense that it ignores the actual boundary conditions of the problem at hand, and simply captures all admissible boundary flows. \\

Distributed pH systems bring together different areas of mathematics, physics and engineering. They were first formulated using the language of differential geometry but their mathematical as well as system theoretic properties have also being studied in a functional analytic setting (cf.  \cite{legorrec2005,villegas2007port,skrepek2021,birgit2021riesz} for works on mathematical wellposedness and \cite{ramirez2014exponential,augner2014stability} for studies on stability and stabilization). Since they are closed under interconnection (this feature stems from the properties of the Dirac structure \cite{cervera2007int}), pH systems have the potential to tackle complex multiphysical engineering applications.  So far they were employed to model fluid-structure coupled phenomena \cite{cardoso2017}, reactive flows \cite{altmann2017reactive},  Euler and Navier-Stokes equations \cite{rashad2021part1,rashad2021part2,califano2021ns}, thin mechanical and thermomecanical structures \cite{brugnoli2020,brugnoli2021thermo}. The interested reader may consult \cite{rashad2020review} for a comprehensive review on distributed pH systems. \\

To exploit the full potential of this novel modelling framework, discretization schemes must retain the many properties of pH systems. In particular, numerical algorithms should capture the underlying conservation laws and provide a systematic way to handle boundary conditions, so that the resulting discrete pH system defines a (modular) Dirac structure. The first contribution in this sense dates back to \cite{golo2004hamiltonian}. Therein, the Stokes-Dirac structure is discrete in strong form exploiting mixed finite element differential forms and compatibility conditions need to be fulfilled\footnote{These conditions turn out to be very particular cases of the bounded commuting projectors of the Finite Element Exterior calculus framework \cite{arnold2006acta}.}.
A reduction method based on symplectic Gauss-Legendre collocation points is detailed in \cite{moulla2012pseudo}. This methodology suffers a major drawback, as it cannot handle higher dimensional spatial domains. A staggered finite difference strategy is proposed in \cite{trenchant2018} but the scheme is specialized for the wave equation only. The discrete exterior calculus framework \cite{hirani2003discrete}, based on algebraic topology, can also be used to formulate pH systems in a purely discrete manner, as shown in \cite{seslija2014simplicial}. However, this framework does not \revTwo{explicitly} rely on interpolating basis function but only on  mimetic discrete operators that reproduce the behaviour of their continuous counterparts. This makes necessary the employment of dual topological meshes, based on the Delaunay-Voronoi duality, to construct an isomorphic discrete Hodge star. This limits the applicability of the method, as dual meshes require ghost points that do not lie inside the physical domain, thus making it difficult to interconnect systems along a physical boundary. In order to avoid the need for dual meshes, a Galerkin formulation based on Whitney forms is detailed in \cite{kotyczka2018weak}. In order to construct a minimal bond space\footnote{\revTwo{The concept of bond space is fundamental in bond graph modelling. See \cite{vanderschaft2014port} for the relation between bond graph modelling and port-Hamiltonian systems}}, projector matrices parametrized by tunable parameters need to be introduced. The best choice of these parameters depends strongly on the application at hand and this complicates the employment of the method. When only one computational mesh is used, it is not possible to construct a discrete isomorphic Hodge operator \cite{hiptmair2001,bochev2006}.
\revTwo{Mixed finite elements do not rely on dual meshes to construct an isomorphic discrete Hodge, but rather on a weak formulation of the codifferential operator via the integration by parts formula. This was formalized by the mimetic discretization community \cite{bossavit2000,bochev2006,arnold2006acta} and, in the context of port-Hamiltonian systems, by the Partitioned Finite Element method \cite{cardoso2020pfem}. In particular, in \cite{haine2020numerical} convergence of mixed finite element discretization is guaranteed for a large class of finite elements, and not only restricted to the families that satisfy the de Rham cohomology}. \\

From the past literature, it may look like new tools and implementations are needed to be able to capture the particular features of pH systems. In this paper we argue the opposite. The pH modelling paradigm represent a unifying framework to describe physical problems. As such, it should be discretized via a unifying discretization method: the Finite Element Exterior calculus theory, \revTwo{initiated by the Japanese papers of Alain Bossavit and fully developed by Arnold,  Falk and Whinter}  \cite{arnold2006acta}, provides all the tools needed to accomplish this task. Using finite element differential forms, a discrete counterpart of the Stokes-Dirac structure can be easily obtained. However, the discretization of constitutive equations remains problematic as those require a discrete Hodge star. Recently, a dual field formulation has been introduced to discretize the Navier-Stokes equations \cite{zhang2021mass}. This methodology seeks variables in dual polynomial finite element spaces via a mixed finite element formulation and is capable of conserving mass, helicity and energy. A staggered time discretization, in which the vorticity information is exchanged between the two systems, allows treating the non linear rotational term in a linear fashion and makes this scheme computationally competitive with state of the art methods. 
This method has a clear connection with Hamiltonian dynamical systems, as discussed by the authors in the introduction of the paper. Furthermore, it relies on the Finite Element Exterior calculus framework as the finite elements used therein (the so-called mimetic quadrilateral/hexahedral finite elements \cite{palha2014mimetic}) form a discrete de Rham complex. \\

In this contribution, we show how the dual field formulation can be used for the systematic discretization of linear pH systems (but the framework is extendable to the state modulated case where one has a Poisson manifold, as already done in \cite{zhang2021mass}). To show the generality of the approach, we detail the calculations in exterior calculus. The dual field formulation is based on the construction of an adjoint Stokes-Dirac structure, where the flows and efforts are the Hodge dual of the initial Stokes-Dirac structure. The introduction of an adjoint system leads to a redundant representation of the dynamics, as the state variables are doubled. Nevertheless, this redundancy removes the need for a discrete Hodge star, as one can use dual variables to expresses the energy as a duality product of forms, rather than the classical inner product. The Stokes-Dirac structure can be immediately discretized in a mimetic fashion, whereas the adjoint Stokes-Dirac structure requires the integration by parts formula to obtain a discrete representation of the codifferential. The constitutive equations are embedded in the dynamics by expressing the state variables in terms of the effort ones. In other words, by using a redundant dual system, the state variables will be the original ones, plus the Hodge star of the former. In this way, it is possible to calculate the needed efforts corresponding to the differential of the energy function, without any need of the Hodge operation. The boundary conditions are here imposed strongly. This choice leads to two uncoupled mixed finite element discretization of the dynamical pH systems, each containing two unknowns (as in the Partitioned Finite Element method).  In this way, after these two systems are integrated on the same time grid via symplectic collocation Runge-Kutta methods (for simplicity the implicit midpoint scheme will be used here), an exact discrete power balance is retrieved. This implies that the discretization scheme gives rise to a Dirac structure and can be used for interconnecting different physical systems in a structured fashion. We detail the discrete operations (like the inner and duality products, the exterior derivative and the trace) using the trimmed polynomial spaces. The lowest order case coincide with the Whitney forms \cite{whitney1957}. The Whitney forms are particularly meaningful, since they make explicit a number of concepts  like metrical versus topological properties and operations. Numerical tests, assessing the conservation properties of the scheme and the convergence rate of the different variables, are performed on the wave and Maxwell equations in a three dimensional Euclidean domain\footnote{The implementation of the finite element method on a triangulated polyhedral manifold is immediate when the manifold is embedded in Euclidean space \cite{arnold2018finite}.}. \revTwo{The numerical tests show that each of two mixed discretizations converge to the exact solution under mesh refinement, without exhibiting synchronisation problems.} The well established open-source library \firedrake \cite{rathgeber2017firedrake} is employed for the implementation (but one can also use the \fenics \cite{logg2012} library). \\

The paper is organised as follows: in Sec. \ref{sec:prel} some preliminaries on needed concepts and notation will be presented. In Sec. \ref{sec:pradj_SDS} we introduce the concept and representation of the \revTwo{classical} and adjoint Stokes-Dirac structure and the associated pH systems which is instrumental in the presented solution. The weak formulation, upon which the discretization scheme is based, is presented in Sec. \ref{sec:weak_df}. The spatial discretization is illustrated in Sec. \ref{sec:space_discr} by means of the trimmed polynomial spaces. We detail the discrete version of the dual field Stokes-Dirac structure as well as the associated dynamical system. In Sec. \ref{sec:time_discr}, the time discretization is explained using the implicit midpoint method. Sec. \ref{sec:num_exp} collects the numerical tests and shows the numerical \revTwo{properties} of the proposed methodology. 

\section{Preliminaries}\label{sec:prel}
In this section we will introduce some preliminary concepts and their notation which are needed for the treatment ahead. The presented material is inspired by \cite{bochev2006,arnold2018finite}.

\subsection{Smooth differential forms}

Let $M\subset \bbR^n$ be a differentiable Riemannian manifold with boundary $\partial M$ and metric $g$. The space of smooth differential forms on $M$ (i.e. the space of smooth sections of the $k$th exterior power of the cotangent bundle $T^* M$) is denoted by $\Omega^k(M)$. The wedge product \revTwo{$\wedge : \Omega^k(M) \times \Omega^l(M) \rightarrow \Omega^{k+l}(M), \; k,l \ge 0, \; k+l \le n$} is the skew-symmetric exterior product of differential forms
\begin{equation}\label{eq:skew_wedge}
    \alpha^k \wedge \beta^l = (-1)^{kl} \beta^l \wedge \alpha^k,
    \qquad \alpha^k \in \Omega^k(M), \; \beta^l \in \Omega^l(M).
\end{equation}

One fundamental operator acting on smooth differential forms is the exterior derivative $\d:\Omega^k(M) \rightarrow \Omega^{k+1}(M), \; k\le n-1$ that satisfies the following axiomatic properties
\begin{itemize}
    \item $\d f$ for $f \in \Omega^0(M)$ is the differential of $f$;
    \item $\d\d f = 0$ for $f \in \Omega^0(M)$;
    \item Product (or Leibniz) rule
    \begin{equation}\label{eq:leibniz}
        \d (\alpha^k \wedge \beta^l) = \d\alpha^k \wedge \beta^l + (-1)^k \alpha^k \wedge \d\beta^l, \qquad \alpha^k \in \Omega^k(M), \quad \beta^l \in \Omega^l(M), \quad k+l < n.
    \end{equation}
\end{itemize}
Occasionally for the sake of additional clarity, the exterior derivative acting on a $k$-form will be denoted by $\d^k$. \revTwo{Consider a manifold $M$ of dimension $n$, a manifold $S$ of dimension $m$, and a smooth
mapping between them, $\Phi : S \rightarrow M$. The pullback operator, $\Phi^*$, is a mapping $\Phi^* : \Omega^k(M) \rightarrow \Omega^k(S)$, that
maps $k$-forms in $M$ to $k$-forms in $S$, with $k\le m$, and $k \le n$, naturally.} An important case of pull back is the trace operator.
\begin{definition}[Trace operator]
The trace operator is defined to be the pull back of the \revTwo{inclusion} map $\iota: \partial M \rightarrow M$
\begin{equation}\label{eq:trace}
    \tr \omega^k := \iota^*(\omega^k), \qquad \omega^k \in \Omega^k(M), \quad k\le n-1.  
\end{equation}
\end{definition}
One fundamental result in exterior calculus is the Stokes theorem, that relates derivation, integration and the trace operator
\begin{equation}\label{eq:Stokes}
    \int_M \d \omega^{n-1} = \int_{\partial M} \tr \omega^{n-1}, \qquad \omega^{n-1} \in \Omega^{n-1}(M),
\end{equation}
\revTwo{with $M$ an $n-$dimensional manifold.} Differential forms possess a natural duality product.
\begin{definition}[Duality product]
Given a smooth manifold $M$ of dimension $n$, the duality product is denoted by
\begin{equation}\label{eq:dual_pr}
    \dualpr[M]{\alpha^k}{\beta^{n-k}} := \int_M \alpha^k \wedge \beta^{n-k}, \qquad \alpha^{k} \in \Omega^{k}(M), \quad \beta^{n-k} \in \Omega^{n-k}(M), \quad k=0, \dots, n.
\end{equation}

The duality product is also defined on the boundary $\partial M$ \revTwo{(whose orientation is inherited from the one of the manifold)}
\begin{equation}
    \dualpr[\partial M]{\alpha^k}{\beta^{n-k-1}} := \int_{\partial M} \tr \alpha^k \wedge \tr \beta^{n-k-1}, \qquad \alpha^k \in \Omega^{k}(M), \quad  \beta^{n-k-1} \in \Omega^{n-k-1}(M), 
\end{equation}
\end{definition}
\revTwo{where $k=0, \dots, n-1$.}
Combing together the Leibniz rule and the Stokes theorem, one has the integration by parts formula 
\begin{equation}\label{eq:int_byparts_d}
    \dualpr[M]{\d\alpha}{\beta} + (-1)^k \dualpr[M]{\alpha}{\d\beta} = \dualpr[\partial M]{\alpha}{\beta}, \qquad \alpha \in \Omega^{k}(M), \quad \beta \in \Omega^{n-k-1}(M), \quad k=0, \dots, n-1.
\end{equation}

\subsection{$L^2$ theory of differential forms}
Given a coordinate chart $\xi_i:M \rightarrow \bbR \quad i=1, \dots, n$, we can represent locally a point $p$ in the manifold $M$ with the tuple $\xi(p) := (\xi_1(p), \dots, \xi_n(p))$.
Then, the local representation of a form $\alpha^k \in \Omega^k(M)$ reads
\begin{equation}\label{eq:local_form}
    \alpha^k(p) = \sum_I \alpha_I(\xi(p)) \d\xi^{i_1} \wedge \dots \wedge \d\xi^{i_k},
\end{equation}
where the multi-index $I= i_1, \dots, i_k, \; 1 \le i_1 \le \dots \le i_k \le n$ has been introduced. The space of $k$-forms can be equipped with a pointwise inner product, inherited form the metric structure of the Riemannian manifold (that establishes an inner product of vectors) and the duality between vectors and differential forms. 

\begin{definition}[Pointwise inner product of forms]\label{eq:local_inpr}
Let $\alpha^k, \beta^k \in \Omega^k(M)$ be two forms with local representation as in \eqref{eq:local_form}, the point-wise inner product is given by
\begin{equation}
    \inpr{\alpha^k}{\beta^k} = g^{i_1, j_1} \dots g^{i_k, j_k} \alpha_{i_1, \dots, i_k} \beta_{j_1, \dots, j_k},
\end{equation}
where $g^{k, l} = (\d\xi^k, \d\xi^l)$ are the components of the inverse metric tensor.
\end{definition}

Once an orientation is given to the Riemannian manifold, the Hodge star operator $\star$ can be properly defined. \revTwo{
The Hodge maps inner-oriented (or true) forms, measuring intensities, to outer-oriented (or pseudo) forms \cite{kreeft2011mimetic,frankel2011geometry}, measuring quantities, and vice versa. This distinction is of fundamental importance and allows defining quantities that are orientation independent (like mass, energy, etc.). In this paper outer-oriented forms are denoted by means of a hat, i.e. $\dual{\alpha}^k \in \dual{\Omega}^k(M)$ where $\dual{\Omega}^k(M)$ is the space of outer oriented (or pseudo) forms.
}
\begin{definition}[Hodge-$\star$ operator]
The Hodge-$\star$ operator, defined for an $n$-dimensional Riemannian manifold $M$, is the operator $\star : \Omega^k(M) \rightarrow \dual{\Omega}^{n-k}(M)$ such that
\begin{equation*}
    \alpha^k \wedge \star \beta^k = \inpr{\alpha^k}{\beta^k} \mathrm{vol}, \qquad \alpha^k, \beta^k \in \Omega^k(M)
\end{equation*}
where the inner product is defined in Def. \ref{eq:local_inpr}. The standard volume form in local coordinates is given by \cite[Page 362]{frankel2011geometry}
\begin{equation*}
    \mathrm{vol} = \sqrt{|\mathrm{det}(g_{ij})|} \epsilon_{1\dots n}\d\xi^1 \wedge \dots \wedge \d\xi^n,
\end{equation*}
where $g_{ij}$ are the components of the metric tensor in the chosen chart $\xi$ and $\epsilon_{1\dots n}$ is the Levi-Civita symbol. \revTwo{Notice the often forgotten importance of the Levi-Civita symbol, that indicates that the volume form is indeed a pseudo form $\mathrm{vol} = \star 1$, as the total volume of a portion of space cannot be negative.}
\end{definition}

We now introduce the $L^2$ inner product of forms.

\begin{definition}[$L^2$ inner product]
Given a smooth manifold $M$ of dimension $n$, the $L^2$ inner product is defined by
\begin{equation*}
    \inpr[M]{\alpha^k}{\beta^k} := \int_M \inpr{\alpha^k}{\beta^k} \mathrm{vol} = \int_M \alpha^k \wedge \star \beta^k, \qquad \alpha, \; \beta \in \Omega^{k}(M)
\end{equation*}
\end{definition}
As in vector calculus, the $L^2$ Hilbert space of differential forms is defined by completion. 
\begin{definition}[$L^2$ space of differential forms]
    The $L^2\Omega^k(M)$ space is the completion of the space of smooth forms $\Omega^k(M)$ in the norm induced by the $L^2$ inner product.
\end{definition}

\revTwo{The $L^2$ inner product is fundamental for the construction of weak formulations. Weak formulations allows enlarging the solution space from the space of smooth forms, to the less regular Sobolev spaces, that represent the domain of the exterior derivative.
}

\begin{definition}[Sobolev spaces of differential forms]
The space of square integrable forms whose derivative is square integrable is denoted by
\begin{equation*}
    H\Omega^k(M) := \{\omega^k \in L^2 \Omega^k(M) \vert \; \d{\omega^k} \in L^2 \Omega^{k+1}(M)\}, \qquad k=0, \dots, n-1.
\end{equation*}
For $k=n$, the space reduces to an $L^2$ space, $H\Omega^n(M) = L^2\Omega^n(M)$.
\end{definition}
These spaces, connected by the operator $\d$, form the de Rham domain complex, as illustrated in the following diagram.
\begin{figure}[h]
\centering
\begin{tikzcd}
H\Omega^0(M) \arrow[r, "\d"] & \dots \arrow[r, "\d"] & H\Omega^k(M) \arrow[r, "\d"] & \dots \arrow[r, "\d"] & H\Omega^{n}(M) 
\end{tikzcd} 
\caption{The domain de Rham complex.}
\label{fig:cd_deRham}
\end{figure}

\revTwo{The adjoint\footnote{In some mathematical textbooks, the term adjoint is used as synonymous for the dual operator. Here, the term adjoint always refers to the Hilbert adjoint based on the inner product structure.} operator to the exterior derivative with respect to the $L^2$ inner product is the codifferential.}

\begin{definition}[Codifferential]
    The co-differential map $\d^* : \Omega^k(M) \xrightarrow{} \Omega^{k-1}(M)$ is defined by
\begin{equation}\label{eq:codif}
        \d^* := (-1)^{nk + n + 1} {\star} \d {\star}.
\end{equation}
The definition is such that the codifferential corresponds to the formal  adjoint of the exterior derivative\footnote{Please note the difference in notation between the inner product $\inpr{\bullet}{\bullet}$ and the duality product $\dualpr{\bullet}{\bullet}$.}
\begin{equation}\label{eq:intbyparts_codif}
    \inpr[M]{\alpha^k}{\d^* \beta^{k+1}} = \inpr[M]{\mathrm{d} \alpha^{k}}{\beta^{k+1}} - \dualpr[\partial M]{\alpha^{k}}{\star \beta^{k+1}}, \qquad \alpha \in \Omega^k(M), \; \beta \in \Omega^{k+1}(M).
\end{equation}
\end{definition}

\revTwo{Sobolev spaces for the codifferential are then defined as the space of forms whose codifferential is square integrable. In this work the codifferential will be always treated in a weak manner, using the integration by parts in Eq. \eqref{eq:intbyparts_codif}. This choice leads to a weak formulation in which all variable live in the de-Rham complex of Fig. \ref{fig:cd_deRham}}.

\section{The \revTwo{canonical} and adjoint Stokes-Dirac structure}\label{sec:pradj_SDS}
In this section the strong formulation of port-Hamiltonian systems and the important role of the underlying geometrical structure are recalled. The corresponding adjoint system is constructed. \revTwo{For sake of concreteness, first the introductory example of the scalar wave equation with Dirichlet boundary control is presented. This problem is a simple example of port-Hamiltonian system.
\subsection{An introductory example: the scalar wave equation}\label{sec:ex_wave}
The propagation of acoustic waves in $M\subset \bbR^3$ is described by the following hyperbolic partial differential equation, that determines the time-dependent (pseudo) $3$-form field $\dual{\eta}^3(t): [0, T_{\mathrm{end}}] \rightarrow \dual{\Omega}^3(M)$
\begin{equation}\label{eq:irr_wave}
    \partial_{tt}^2 \dual{\eta}^3 + \d \d^* \dual{\eta}^3 = 0,
\end{equation}
together with time-varying Dirichlet boundary condition
\begin{equation}\label{eq:dirbc_wave}
    \tr \star \dual{\eta}^3|_{\partial M} = g^0.
\end{equation}
To highlight the Hamiltonian structure of the wave equation, consider the variables
\begin{equation}
    \dual{v}^3 := \partial_t \dual{\eta}^3, \qquad \sigma^1 = -\d \star \dual{\eta}^3.
\end{equation}
Equation \eqref{eq:irr_wave}, together with the boundary condition \eqref{eq:dirbc_wave}, can now be recast into a first order system 
\begin{equation}\label{eq:mix_wave}
    \begin{pmatrix}
        \partial_t \dual{v}^3 \\
        \partial_t \sigma^1
    \end{pmatrix} = -\underbrace{\begin{bmatrix}
        0 & \d \\
        \d & 0 \\
    \end{bmatrix}}_{J}
    \begin{pmatrix}
        \star \dual{v}^3 \\
        \star \sigma^1
    \end{pmatrix}, \qquad \tr \star \dual{v}^3|_{\partial M} = \partial_t g^0:= u^0.
\end{equation}
Here $u^0$ corresponds to a control input applied to the boundary. Indeed port-Hamiltonian systems are boundary controlled systems and the boundary conditions coincide with inputs that describe interactions with the external environment. \\

The choice of variables $\dual{v}^3, \; \sigma^1$ leads to a clear separation between topological and metrical properties. The topological properties are incorporated into the (skew-dual) operator $J$, whereas the metric properties are evident in the appearance of the Hodge in the right hand side of Eq. \eqref{eq:irr_wave}. The topological properties relate to conservation laws, where the metrical properties are related to the constitutive laws of the physical system under consideration.

The power that flows through the domain boundary is obtained by taking the duality product with $\star \dual{v}^3$ and $\star \sigma^1$ and using the Leibniz rule and  Stokes theorem
\begin{equation}
	\label{eq:wave_power_balance}
\begin{aligned}
    \dualpr[M]{\star \dual{v}^3}{\partial_t \dual{v}^3} + \dualpr[M]{\star \sigma^1}{\partial_t \sigma^1} &= \dualpr[M]{\star \dual{v}^3}{-\d \star \sigma^1} + \dualpr[M]{\star \dual{v}^3}{-\d \star v^3}, \\
    &= \dualpr[\partial M]{-\tr \star \sigma^1}{\tr \star \dual{v}^3}, \\
    &= \dualpr[\partial M]{\dual{y}^2}{u^0}
\end{aligned}
\end{equation}
where variable $\dual{y}^2$ corresponds to the power conjugated output to the input, i.e. the power collocated output
\begin{equation}
    \dual{y}^2 := -\tr \star \sigma^1.
\end{equation}
It can be noticed that this variable corresponds to the Neumann boundary condition \cite{arnold2018finite}. Furthermore, the output is dual to the input on the boundary $\partial M$.
%Notice that the duality product at the boundary $\dualpr[\partial M]{\dual{y}^2}{u^0}$ involve an inner and an outer oriented forms. \\
Notice that each duality product in (\ref{eq:wave_power_balance}) involves pairing an inner-oriented form with an outer-oriented one.\\

System \eqref{eq:mix_wave} is an example of a port-Hamiltonian system. The underlying geometrical structure is the Stokes-Dirac structure, an infinite dimensional generalization of Dirac manifolds introduced by Courant \cite{courant1990}. In what follows we introduce the general port-Hamiltonian formulation for a system of two conservation laws.
}

\subsection{Linear Stokes-Dirac structures}
Given a smooth manifold $M$ of dimension $n$, the \revOne{flow} variables $\dual{f}^p_1 \in \dual{\Omega}^p(M), \; {f}^q_2 \in \Omega^q(M)$ and the effort variables ${e}^{n-p}_1 \in \Omega^{n-p}(M), \; \dual{e}^{n-q}_2 \in \dual{\Omega}^{n-q}(M)$, with\footnote{This relation in dimensions is due to the fact that the fields are dual and will become clear from the construction hereafter.} $p+q=n+1$, consider the Stokes-Dirac structure

\begin{equation}\label{eq:StokesDirac}
    \begin{pmatrix}
        \dual{f}^p_1 \\
        {f}^q_2
    \end{pmatrix} = 
    \underbrace{\begin{bmatrix}
    0 & (-1)^r \d \\
    \d & 0 \\
    \end{bmatrix}}_{J}
    \begin{pmatrix}
        {e}^{n-p}_1 \\
        \dual{e}^{n-q}_2
    \end{pmatrix}, \qquad 
    \begin{pmatrix}
        {f}_\partial^{n-p} \\
        \dual{e}_\partial^{n-q}
    \end{pmatrix} = 
    \begin{bmatrix}
    \tr & 0 \\
    0 &  (-1)^p\tr
    \end{bmatrix}
    \begin{pmatrix}
        {e}^{n-p}_1 \\
        \dual{e}^{n-q}_2
    \end{pmatrix},
\end{equation}
where $r = pq +1$ for mathematical and physical reasons \cite{vanderSchaft2002}. The operator $J$, representing what is called in physical system theory the junction structure, is called the interconnection operator and is a formally \revTwo{skew-dual} operator. This property is at the core of the balance equation
\begin{equation}\label{eq:bal_eq}
    \dualpr[M]{e^{n-p}_1}{\dual{f}^p_1} + \dualpr[M]{\dual{e}^{n-q}_2}{f^q_2} + \dualpr[\partial M]{\dual{e}_\partial^{n-q}}{{f}_\partial^{n-p}} = 0,
\end{equation}
where $\partial M$ denotes the boundary of manifold $M$. This balance equation arises from the application of the Leibniz rule and the Stokes theorem and states the overall conservation of energy.
This is the  embodiment of Tellegen's theorem generalised to the  distributed parameters systems case \cite{iftime2014kernel}. 
\revTwo{\begin{remark}
The Stokes-Dirac structure is a geometrical concept that does not involve time. The exterior derivative is associated to the spatial manifold only. A flow variable may be associated to a dynamics (i.e. $f^k = -\partial_t \alpha^k$ as it will be shown in the next section) leading to a conservation law, or to a resistive relation that introduces dissipation. The hyperbolic case arises when both flows have an associated dynamics, whereas the parabolic case is obtained when one flow is associated to a dynamics and the other one is purely algebraic and associated to resistive relation. In this paper the focus will be on the hyperbolic case only.
\end{remark}
}

\subsection{Port-Hamiltonian systems}
Port-Hamiltonian systems encoding conservation laws are associated to the geometrical Stokes-Dirac structure. To establish this connection, consider the distributed state variables $\dual{\alpha}^p(t) : [0, T_{\text{end}}] \rightarrow \dual{\Omega}^p(M), \; \beta^{q}(t) : [0, T_{\text{end}}] \rightarrow \Omega^q(M)$ and the Hamiltonian functional 
\begin{equation}\label{eq:H}
H(\dual{\alpha}^p, \beta^{q}) = \int_M
\mathcal{H}(\dual{\alpha}^p, \beta^{q})
\end{equation}
with Hamiltonian density $n$-form $\mathcal{H}$. A fundamental notion is the variational derivative of the Hamiltonian functional \cite{olver1986applications, vanderSchaft2002}. 
\begin{definition}[Variational derivative]\label{def:var_der}
The variational derivatives of the Hamiltonian $\delta_{\dual{\alpha}} H^{n-p} \in \Omega^{n-p}(M), \; \delta_{\beta} H^{n-q} \in \dual{\Omega}^{n-q}(M)$ are defined implicitly by
\begin{equation*}
\begin{aligned}
    \left.\diff{}{\varepsilon}\right|_{\varepsilon = 0} H(\dual{\alpha}^p + \varepsilon \delta \dual{\alpha}^p, \beta^{q}) = \dualpr[M]{\delta_{\dual{\alpha}} H^{n-p}}{\delta \dual{\alpha}^p}, \\
    \left.\diff{}{\varepsilon}\right|_{\varepsilon = 0} H(\dual{\alpha}^p, \beta^{q} + \varepsilon \delta \beta^{q}) = \dualpr[M]{\delta_{\beta} H^{n-q}}{\delta \beta^{q}}.
\end{aligned}
\end{equation*}
\end{definition}

Considering then a trajectory $(\dual{\alpha}^p(t),\beta^{q}(t))$ parameterised by time in the manifold of state fields, the variational derivative is defined so that the rate of the Hamiltonian \eqref{eq:H} reads
\begin{equation}\label{eq:H_dot}
    \dot{H} = \dualpr[M]{\delta_{\dual{\alpha}} H^{n-p}}{\partial_t \dual{\alpha}^p} + \dualpr[M]{\delta_{\beta} H^{n-q}}{\partial_t \beta^{q}}.
\end{equation}

With an abuse of notation, we will write \revTwo{$\dual{\alpha}^p(\xi, t)= \dual{\alpha}^p(t), \; \beta^q(\xi, t) = \beta^q(t)$} to make explicit the value of the field at a certain point $\xi\in M$ rather than the field as a section. Consider then the following equations which can be seen to represent a system of two conservation laws with canonical inter-domain coupling \cite{vanderSchaft2002}
\begin{equation}\label{eq:pHsys}
    \begin{pmatrix}
        \partial_t \dual{\alpha}^p(\xi, t) \\
        \partial_t \beta^{q}(\xi, t) \\
    \end{pmatrix} = -
    \begin{bmatrix}
    0 & (-1)^r \d \\
    \d & 0 \\
    \end{bmatrix}
    \begin{pmatrix}
        \delta_{\dual{\alpha}} H^{n-p} \\
        \delta_{\beta} H^{n-q}
    \end{pmatrix}, 
\end{equation}
with initial conditions
\begin{equation}
    \begin{pmatrix}
        \dual{\alpha}^p(\xi, 0) \\
        \beta^{q}(\xi, 0)
    \end{pmatrix} = 
    \begin{pmatrix}
        \dual{\alpha}^p_0(\xi) \\
        \beta^{q}_0(\xi) \\
    \end{pmatrix}.
\end{equation}
\revTwo{
The flows and efforts of the associated Stokes-Dirac 
structure \eqref{eq:StokesDirac} are then defined by
\begin{equation}\label{eq:pHsys_flows_efforts}
\begin{pmatrix}
        \dual{f}^p_1 \\
        f^q_2
    \end{pmatrix} := -
    \begin{pmatrix}
        \partial_t \dual{\alpha}^p(\xi, t) \\
        \partial_t \beta^{q}(\xi, t) \\
    \end{pmatrix}, \qquad 
    \begin{pmatrix}
        e^{n-p}_1 \\
        \dual{e}^{n-q}_2
    \end{pmatrix} := 
    \begin{pmatrix}
        \delta_{\dual{\alpha}} H^{n-p} \\
        \delta_{\beta} H^{n-q}
    \end{pmatrix}.
\end{equation}
\paragraph{Boundary conditions}
Mixed boundary conditions will be considered in this work. We denote with $\Gamma_1$  and $\Gamma_2$ two open subsets of the boundary that   verify $\overline{\Gamma}_1 \cup \overline{\Gamma}_2 = \partial M$ and $\Gamma_1 \cap \Gamma_2=\emptyset$. Each boundary subpartition is associated with one boundary condition. In particular, the values of $\delta_{\dual{\alpha}} H^{n-p} = e^{n-p}_1$ (resp. $\delta_{\beta} H^{n-q}= \dual{e}^{n-q}_2$) are imposed on $\Gamma_1$ (resp. $\Gamma_2$). Since in this paper we consider boundary controlled systems, the boundary conditions are assigned by means of the inputs
 \begin{equation}\label{eq:u}
 \begin{aligned}
     \tr e^{n-p}_1 \vert_{\Gamma_1} &= u^{n-p}_1 \in \Omega^{n-p}(\Gamma_1), \\
    (-1)^p \tr \dual{e}^{n-q}_2 \vert_{\Gamma_2} &= \dual{u}^{n-q}_2 \in \dual{\Omega}^{n-q}(\Gamma_2).
 \end{aligned}
 \end{equation}
}
\revTwo{
The outputs are defined in such a way that they are power conjugated to the inputs
 \begin{equation}\label{eq:y}
 \begin{aligned}
      y^{n-p}_1 := \tr e^{n-p}_1 \vert_{\Gamma_2} &\in \Omega^{n-p}(\Gamma_2), \\
    \dual{y}^{n-q}_2 := (-1)^p \tr \dual{e}^{n-q}_2 \vert_{\Gamma_1} &\in \dual{\Omega}^{n-q}(\Gamma_1). 
 \end{aligned}
 \end{equation}
Given the balance equation \eqref{eq:bal_eq}, the rate of change of Hamiltonian (power flow) reads
\begin{equation}\label{eq:powbal}
\begin{aligned}
    \dot{H} &= \dualpr[\partial M]{\dual{e}_\partial^{n-q}}{f_\partial^{n-p}}, \\
    &= \dualpr[\Gamma_1]{\dual{y}^{n-q}_2}{u^{n-p}_1} + \dualpr[\Gamma_2]{\dual{u}^{n-q}_2}{y^{n-p}_1}
\end{aligned}
\end{equation}
where the last equality descends from the additive property of the integration and the definition of the boundary variables given in Eq. \eqref{eq:StokesDirac}.
Equation \eqref{eq:powbal} states that the change of energy within the spatial domain $M$ is equal to the supplied power through its boundary $\partial M$. Notice that since $p+q=n+1$, the boundary variables are dual forms on the boundary $(n-p) + (n-p) = 2n - (n+1) = n-1$.
}

\paragraph{Constitutive equations}
Since in this work the focus is on the case in which the closure equations are linear, we make the additional assumption of a quadratic Hamiltonian.
\begin{assumption}[Quadratic Hamiltonian]\label{ass:quad_Ham}
The Hamiltonian density in assumed to be quadratic in the state variables
\begin{equation}\label{eq:H_den}
\mathcal{H}(\dual{\alpha}^p, \beta^{q}) = 
    \frac{1}{2} \dual{A}^p \dual{\alpha}^p \wedge \star \dual{\alpha}^p + \frac{1}{2}  B^q \beta^{q} \wedge \star  \beta^{q},
\end{equation}
where $\dual{A}^p: L^2\dual{\Omega}^{p}(M) \rightarrow L^2\dual{\Omega}^{p}(M)$ and $B^q: L^2\Omega^{q}(M) \rightarrow L^2\Omega^{q}(M)$ are bounded from above and below, symmetric, positive definite operators with respect to the standard inner products in $L^2$. These tensors are related to physical properties of space like electric permittivity, mass density or Young modulus.
\end{assumption}
Because of this assumption, the variational derivatives of the Hamiltonian, computed according to Definition \ref{def:var_der}, can be seen to be
\begin{equation}
    \delta_{\dual{\alpha}} H^{n-p} = (-1)^{p(n-p)}\star \dual{A}^p\dual{\alpha}^p, \qquad
    \delta_{\beta} H^{n-q} = (-1)^{q(n-q)}\star B^q\beta^{q}.
\end{equation}

\subsection{Adjoint Stokes-Dirac structure \label{sec:asds}}
\revTwo{
The port-Hamiltonian system (\ref{eq:pHsys}) and its associated Stokes-Dirac structure (\ref{eq:StokesDirac}) were defined on the state space denoted by $X = \dual{\Omega}^p(M) \times \Omega^q(M)$. The key ingredient that will be used for our proposed dual-field discretization approach is the Adjoint Stokes-Dirac structure that we will introduce in this section. The adjoint structure is defined on the Hodge dual state space $\dual{X} = \Omega^{n-p}(M) \times \dual{\Omega}^{n-q}(M)$ related to $X$ by the diffeomorphism}
\begin{equation}\label{eq:Phi_Diffeo}
	\Phi:
	X \rightarrow \dual{X} ; \qquad
	\begin{pmatrix}
	\dual{\alpha}^p \\    
	\beta^{q}
	\end{pmatrix}
	\mapsto 
	\begin{pmatrix}
	{\alpha}^{n-p} \\    
	\dual{\beta}^{n-q}
	\end{pmatrix} :=
	\begin{pmatrix}
	\star^{-1}\dual{\alpha}^p \\    
	\star \beta^{q}
	\end{pmatrix}, \qquad \Phi = 
	\begin{bmatrix}
        \star^{-1} & 0 \\
        0 & \star 
    \end{bmatrix}.
\end{equation}
It is worth noting that the choice of $\Phi$ is not unique. While in this work we choose (\ref{eq:Phi_Diffeo}) for numerical reasons, in Sec. \ref{sec:additionalinsights} we will discuss an alternative \revTwo{in which the material tensors will be included in the Hodge.}
The pushforward associated to the diffeomorphism converts elements in the tangent space of the state space
\begin{equation}\label{eq:pushfor_phi}
\Phi_* : T X \cong {X} \rightarrow T \dual{X} \cong \dual{X}: \qquad  
	\begin{pmatrix}
	\dual{f}^p_1 \\
	f^q_2
	\end{pmatrix}
\mapsto  
    \begin{pmatrix}
    {f}^{n-p}_1\\
    \dual{f}^{n-q}_2
    \end{pmatrix} :=
    \begin{pmatrix}
    \star^{-1} \dual{f}^p_1\\
    \star f^q_2
    \end{pmatrix}, \qquad \Phi_* = \begin{bmatrix}
        \star^{-1} & 0 \\
        0 & \star 
    \end{bmatrix}.
\end{equation}
\begin{remark}
Since the space of differential forms is an infinite dimensional $\bbR$ vector space, the tangent space  is actually isomorphic with the space of forms: $TX = T\dual{\Omega}^p \times T\Omega^q \cong \dual{\Omega}^p \times \Omega^q$. 
\end{remark}
\revTwo{
\begin{proposition}\label{pr:pullback_phi}
The pullback $\Phi^*$ of the map $\Phi$ is given by:
\begin{equation}
\label{eq:pullback_phi}
\Phi^*: T^*_{\dual{x}} \dual{X} \cong \dual{X} \rightarrow T^*_x X \cong X: \qquad
\begin{pmatrix}
    \dual{e}_1^{p}\\
    {e}^{q}_2
\end{pmatrix} :=
 \begin{pmatrix}
     \star^{-1} {e}^{n-p}_1 \\
     \star \dual{e}^{n-q}_2
 \end{pmatrix}
 \mapsto 
 \begin{pmatrix}
     {e}_1^{n-p} \\
     \dual{e}_2^{n-q}
 \end{pmatrix},  \qquad 
 \Phi^* = \begin{bmatrix}
        \star & 0 \\
        0 & \star^{-1}
    \end{bmatrix}.
\end{equation}
\end{proposition}
\begin{proof}
The pullback is defined by duality
\begin{equation}
    \dualpr[X]{\Phi^* 
    \begin{pmatrix}
      \dual{e}^p_1 \\ 
      {e}^q_2
    \end{pmatrix}
    }{\begin{pmatrix}
        \dual{f}^p_1\\
        f^q_2
    \end{pmatrix}
    } = \dualpr[\dual{X}]{\begin{pmatrix}
      \dual{e}^p_1 \\ 
      {e}^q_2
    \end{pmatrix}}{\Phi_*\begin{pmatrix}
	\dual{f}^p_1 \\
	f^q_2
	\end{pmatrix}}.
\end{equation}
Using the definition of duality product \eqref{eq:dual_pr} one obtains
\begin{equation}
    \begin{aligned}
    \dualpr[\dual{X}]{\begin{pmatrix}
      \dual{e}^p_1 \\ 
      {e}^q_2
    \end{pmatrix}}{\Phi_*\begin{pmatrix}
	\dual{f}^p_1 \\
	f^q_2
	\end{pmatrix}} &=\int_M \dual{e}^{p}_1 \wedge \star^{-1} \dual{f}^p_1 + \int_M {e}_2^{q} \wedge \star f^q_2, \\
        &= \int_M (-1)^{p(n-p)}\dual{e}^p_1 \wedge \star \dual{f}^p_1 + \int_M (-1)^{q(n-q)}\star {e}^q_2 \wedge f^q_2, \\
        &= \int_M \star \dual{e}^p_1 \wedge \dual{f}^p_1 + \int_M (-1)^{q(n-q)}\star {e}^q_2 \wedge f^q_2, \\
        &=  \dualpr[X]{\begin{bmatrix}
        \star & 0 \\
        0 & \star^{-1}
    \end{bmatrix}\begin{pmatrix}
      \dual{e}^p_1 \\ 
      {e}^q_2
    \end{pmatrix}}{\begin{pmatrix}
	\dual{f}^p_1 \\
	f^q_2
	\end{pmatrix}}
    \end{aligned}.
\end{equation}
\end{proof}
\begin{remark}
The procedure here illustrated is general. Given a generic diffeomorphism that acts on the states, its push-forward and pullback provide the flow and effort variables in the transformed system \cite{rashad2021part1,vankerschaver2010}.
\end{remark}
}

By introducing the new flow and effort variables
\begin{equation*}
    \begin{pmatrix}
    f^{n-p}_1 \\
    \dual{f}^{n-q}_2
\end{pmatrix}
 := \Phi_*\begin{pmatrix}
     \dual{f}^p_1\\
     f^q_2
 \end{pmatrix}, \qquad \text{and} \qquad
 \begin{pmatrix}
     e^{n-p}_1\\
     \dual{e}^{n-q}_2
 \end{pmatrix}= \Phi^*
 \begin{pmatrix}
     \dual{e}^{p}_1 \\
     {e}^{q}_2
 \end{pmatrix},
\end{equation*}
\revTwo{the adjoint Dirac structure of (\ref{eq:StokesDirac}) is expressed by means of co-differential map $\d^*$ \eqref{eq:codif} as}
\begin{equation}\label{eq:AdjStokesDirac}
    \begin{pmatrix}
    {f}^{n-p}_1 \\
    \dual{f}^{n-q}_2
    \end{pmatrix} = 
    \begin{bmatrix}
        0 &  (-1)^{a_0}\d{}^* \\
        (-1)^{a_1}\d{}^* & 0 \\
    \end{bmatrix}
    \begin{pmatrix}
        \dual{e}^{p}_1\\
        {e}^{q}_2
    \end{pmatrix}, \qquad 
    \begin{pmatrix}
        {f}_\partial^{n-p} \\
        \dual{e}_\partial^{n-q}
    \end{pmatrix} = 
    \begin{bmatrix}
    \tr \star & 0 \\
    0 &  (-1)^{p+q(n-q)}\tr \star
    \end{bmatrix}
    \begin{pmatrix}
        \dual{e}^{p}_1\\
        {e}^{q}_2
    \end{pmatrix},
\end{equation}
where the following notation has been used
\begin{equation}\label{eq:a_coeff}
    a_0 = r + p(n-p) + q(n-q) + n(q + 1) +1, \qquad a_1 = n(p + 1) + 1.
\end{equation}
\revTwo{The following corollary is an immediate consequence of Prop. \ref{pr:pullback_phi}.}
\begin{corollary}\label{cor:pow_bal_Adj}
The adjoint Stokes-Dirac structure verifies the power balance
\begin{equation}\label{eq:pow_bal_Adj}
    \dualpr[M]{\dual{e}^p_1}{{f}^{n-p}_1} + \dualpr[M]{{e}^q_2}{\dual{f}^{n-q}_2} + \dualpr[\partial M]{\dual{e}^{n-q}_\partial}{f^{n-p}_\partial} = 0.
\end{equation}
\end{corollary}
\begin{proof}
    The property follows from the fact that the pullback \eqref{eq:pullback_phi} preserves duality products. It can be  verified by a direct computation that:
    \begin{equation}
    \begin{aligned}
    \dualpr[M]{e^{n-p}_1}{\dual{f}^p_1} &= \int_M \star \dual{e}^p_1 \wedge \star {f}_1^{n-p}, \qquad &\text{Symmetry of the inner product}, \\
    &= \int_M {f}_1^{n-p} \wedge \star (\star \dual{e}^p_1), \qquad &\text{Property of the the Hodge star}, \\
    &= (-1)^{p(n-p)} \int_M {f}_1^{n-p} \wedge  \dual{e}^p_1, \qquad &\text{Skew-symmetry of the wedge},  \\
    &= \dualpr[M]{\dual{e}^p_1}{{f}_1^{n-p}}.
    \end{aligned} 
\end{equation}
Analogous computations are performed for the second \revOne{flow} and effort variables $f^q_2$ and $\dual{e}^{n-q}_2$.
\end{proof}
To simplify the expressions of the adjoint Dirac structure and relate the new indices $a_0$ and $a_1$ to the $p,q,n$, the following proposition is of value.
\revTwo{
\begin{proposition}\label{pr:parity_a0_a1}
The coefficients $a_0, \;1$ and $a_1, \; 1+ r + p(n-p) + q(n-q)$ have the same parity which means
\begin{equation}
    a_0 \equiv 1 \mod{2}, \qquad a_1 \equiv 1+r+p(n-p) +q(n-q) \mod{2},
\end{equation}
where $\mod{2}$ denotes the modulo of the division by $2$.
\end{proposition}
\begin{proof}
Reported in \ref{app:proofs}.
\end{proof}
}
\noindent This results implies that:
\[
(-1)^{a_0}=-1
\qquad \text{and}
\qquad
(-1)^{a_1}=(-1)^{1+r+p(n-p)+q(n-q)}
\]
and this could be used to express \eqref{eq:AdjStokesDirac} with the coefficients $p,q,n$ and $r$, and will be of importance later. The adjoint Stokes-Dirac structure introduces the codifferential and its important role in discretization. 

\subsection{Associated adjoint port-Hamiltonian system}

Given the adjoint Stokes-Dirac structure \eqref{eq:AdjStokesDirac} and the pH system \eqref{eq:pHsys}, the adjoint pH system is obtained by considering the isomorphism
\[ 
\begin{pmatrix}
    \alpha^{n-p} \\
    \dual{\beta}^{n-q}
\end{pmatrix} := \Phi\begin{pmatrix}
    \dual{\alpha}^{p} \\
    \beta^{q}
\end{pmatrix} = 
\begin{pmatrix}
    \star^{-1}\dual{\alpha}^{p} \\
    \star {\beta}^{q}
\end{pmatrix},
\]
and considering the new Hamiltonian as a function of the state variables of the adjoint system
\begin{equation}
   \widetilde{H}(\alpha^{n-p}(\xi, t), \dual{\beta}^{n-q}(\xi, t)) = \int_M  \widetilde{\mathcal{H}}(\alpha^{n-p}(\xi, t), \dual{\beta}^{n-q}(\xi, t)).
\end{equation}
The energy density $\widetilde{\mathcal{H}}$ corresponds to $\mathcal{H}$ under the isomorphism $\Phi$. The variational derivative, being defined by duality product (cf. Def. \ref{def:var_der}), undergoes the same pullback as the efforts, so that
\begin{equation}\label{eq:dual_coenergies}
    \diffd{\widetilde{H}}{{\alpha}^{n-p}} = \dual{e}^p_1, \qquad  \diffd{\widetilde{H}}{\dual{\beta}^{n-q}} = {e}^q_2.
\end{equation}

The port-Hamiltonian adjoint system is obtained substituting into the adjoint Stokes-Dirac structure \eqref{eq:AdjStokesDirac} the dual effort variables \eqref{eq:dual_coenergies} and the dual \revOne{flow variables} 
\begin{equation}\label{eq:dual_flows}
{f}^{n-p}_1 = - \partial_t {\alpha}^{n-p}, \qquad \dual{f}^{n-q}_2 = - \partial_t \dual{\beta}^{n-q},   
\end{equation}
which yields
\begin{equation}\label{eq:AdjPHsys}
    \begin{pmatrix}
        \partial_t {\alpha}^{n-p} \\
        \partial_t \dual{\beta}^{n-q} \\
    \end{pmatrix} = -
    \begin{bmatrix}
        0 &  (-1)^{a_0}\d{}^* \\
        (-1)^{a_1}\d{}^* & 0 \\
    \end{bmatrix}
    \begin{pmatrix}
        \delta_{\alpha} \widetilde{H}^p \\
        \delta_{\dual{\beta}} \widetilde{H}^q
    \end{pmatrix}.
\end{equation}

Given Assumption \ref{ass:quad_Ham}, the adjoint energy is a quadratic function in the adjoint state variables
\begin{equation}\label{eq:adjoint_energy}
    \widetilde{\mathcal{H}}({\alpha}^{n-p}(\xi, t), \dual{\beta}^{n-q}(\xi, t)) = \frac{1}{2} {A}^{n-p} {\alpha}^{n-p} \wedge \star {\alpha}^{n-p} + \frac{1}{2} \dual{B}^{n-q} \dual{\beta}^{n-q} \wedge \star  \dual{\beta}^{n-q},
\end{equation}
where the operators ${A}^{n-p}$, respectively $\dual{B}^{n-q}$ are the dual of $\dual{A}^p$, respectively ${B}^{q}$ defined implicitly by
\begin{equation}
\begin{aligned}
    \dualpr[M]{\dual{A}^p\dual{\alpha}^{p}}{{\alpha}^{n-p}} &= \dualpr[M]{\dual{\alpha}^p}{{A}^{n-p} {\alpha}^{n-p}}, \\
    \dualpr[M]{B^q\beta^{q}}{\dual{\beta}^{n-q}} &= \dualpr[M]{\beta^{q}}{\dual{B}^{n-q} \dual{\beta}^{n-q}},
\end{aligned}
\end{equation}
are bounded from above and below, symmetric and positive definite with respect to the $L^2$ inner product. Furthermore they verify the commutativity property $\star \dual{A}^p = {A}^{n-p} \star, \; \star B^q = \dual{B}^{n-p} \star$, thus are explicitly expressed by
\begin{equation}
    {A}^{n-p} = (-1)^{p(n-p)} \star \dual{A}^{p} \star, \qquad \dual{B}^{n-q} = (-1)^{q(n-q)} \star B^q \star.
\end{equation}

The variational derivative of the adjoint energy then reads
\begin{equation}
    \delta_{\alpha} \widetilde{H}^p = (-1)^{p(n-p)} \star  {A}^{n-p}{\alpha}^{n-p}, \qquad
    \delta_{\dual{\beta}} \widetilde{H}^q = (-1)^{q(n-q)} \star \dual{B}^{n-q} \dual{\beta}^{n-q}.
\end{equation} 
For what concerns the boundary conditions, it is important to notice that the adjoint system carries the same boundary conditions \eqref{eq:u} as the \revTwo{canonical} one, but expressed by means of the adjoint efforts
 \begin{equation}
    \tr\star \dual{e}^p_1 \vert_{\Gamma_1} = u^{n-p}_1, \qquad
    (-1)^{p+q(n-q)} \tr \star {e}^q_2 \vert_{\Gamma_2} = \dual{u}^{n-q}_2.
 \end{equation}
and analogously for the outputs.

\begin{remark}
	By comparing the adjoint energy density \eqref{eq:adjoint_energy} to the original one in \eqref{eq:H_den} it can be seen that the adjoint port-Hamiltonian system \eqref{eq:AdjPHsys} describes the dynamics governing the \textit{true-form} representation of the $\alpha$ state and the \textit{pseudo-form} representation of the $\beta$ state. On the other hand, the canonical port-Hamiltonian system \eqref{eq:pHsys} describes the dynamics of $\alpha$ as a \textit{pseudo-form} and $\beta$ as a \textit{true-form}. Combining the two representations is the key idea behind the dual-field formulation which we discuss next.
\end{remark}

\section{Weak dual field formulation}\label{sec:weak_df}

We now present the weak dual field formulation of the port-Hamiltonian model, inspired by the work of \cite{zhang2021mass}, but now including general mixed boundary conditions. In this formulation, we combine together the canonical and adjoint port-Hamiltonian systems described in the previous section. This combination will eliminate the necessity of using an explicit Hodge star operator, which will be accounted for using the integration by parts formula in the weak adjoint system \eqref{eq:intbyparts_codif}. \revTwo{As a consequence, two decoupled port-Hamiltonian systems are obtained: one based on outer-oriented forms, and a second, based on inner-oriented forms. In the following we will refer to the former as the primal system and at the latter as the dual system.}

\subsection{Motivation of the approach}

The novelty of this approach is that the Hamiltonian can now be rewritten as a function of the \revTwo{original}  state variables \revTwo{and their Hodge duals}:
 \begin{equation}
H_T(\dual{\alpha}^p, \beta^{q}, {\alpha}^{n-p}, \dual{\beta}^{n-q}) = H(\dual{\alpha}^p, \beta^{q}) + \widetilde{H}({\alpha}^{n-p}, \dual{\beta}^{n-q})= \int_M \mathcal{H}_T,
\end{equation}
with associated density
\revTwo{
\begin{equation*}
\begin{aligned}
    \mathcal{H}_T
(\dual{\alpha}^p, \beta^{q} 
,{\alpha}^{n-p}, \dual{\beta}^{n-q})
&=\frac{1}{2} \dual{A}^p \dual{\alpha}^p \wedge \star \dual{\alpha}^p + \frac{1}{2}  B^q \beta^{q} \wedge \star  \beta^{q} \\
&+ \frac{1}{2} {A}^{n-p} {\alpha}^{n-p} \wedge \star {\alpha}^{n-p} + \frac{1}{2} \dual{B}^{n-q} \dual{\beta}^{n-q} \wedge \star  \dual{\beta}^{n-q} \\
&= {\alpha}^{n-p} \wedge \dual{A}^p \dual{\alpha}^p
+ B^q \beta^{q} \wedge \dual{\beta}^{n-q}. \\
\end{aligned}
\end{equation*}
}
 \revTwo{The two expressions for the Hamiltonian density are equivalent at the continuous level but not when one seeks a discrete representation of the variational derivative. In the second expression, the Hodge star operators have been removed and effectively replaced by a dual-field representation of the variables. This is important because the discretization can be performed without relying on a discrete Hodge star. The effort variables are now computed using the latter expression of the total Hamiltonian}
\begin{equation}
\begin{aligned}
    e^{n-p}_1 &= \diffd{H_T}{\dual{\alpha}^p} = {A}^{n-p} {\alpha}^{n-p}, \\
    \dual{e}^{n-q}_2 &= \diffd{H_T}{\beta^{q}} = (-1)^{q(n-q)} \dual{B}^{n-q} \dual{\beta}^{n-q}, \\
\end{aligned} \qquad \qquad
\begin{aligned}
    \dual{e}^p_1 &= \diffd{H_T}{{\alpha}^{n-p}} = (-1)^{p(n-p)}\dual{A}^p \dual{\alpha}^p, \\
    {e}^q_2 &= \diffd{H_T}{\dual{\beta}^{n-q}} = B^q \beta^{q}. \\
\end{aligned}
\end{equation}
The constitutive equations are then incorporated directly in the dynamics by reducing the state variables in terms of the efforts \revTwo{(in the literature this is called co-energy formulation \cite{brugnoli2020})}. This reduces the number of equations to be solved and formulate the problem in terms of the most important variables, as the efforts are subjected to the boundary conditions. \revTwo{To express the dynamics in term of the effort variables the material operators $\dual{A}^p, B^q, A^{n-p}, \dual{B}^{n-q}$ are inverted (this is possible since these operator are positive definite and bounded from above and below for physical reasons)
\begin{equation}\label{eq:effort_CE}
     \begin{aligned}
     \dual{\alpha}^p &= (-1)^{p(n-p)}\dual{C}^p \dual{e}^p_1, \\
      \beta^{q} &= E^q {e}^q_2, \\
     \end{aligned} \qquad 
     \begin{aligned}
      \alpha^{n-p} &= {C}^{n-p} e_1^{n-p}, \\
      \dual{\beta}^{n-q} &= (-1)^{q(n-q)}  \dual{E}^{n-q} \dual{e}^{n-q}_2, 
     \end{aligned}
\end{equation}
where the tensors $\dual{C}^p:= (\dual{A}^p)^{-1}, \; E^q := (B^q)^{-1}, \; {C}^{n-p}:=({A}^{n-p})^{-1} , \; \dual{E}^q := (\dual{B}^{n-q})^{-1}$ have been introduced. These tensors are commonly introduced in mixed finite element (for example in linear elasticity they represent the density and the compliance \cite{arnold2014weak}, whereas in electromagnetism they represent the electric permittivity and the magnetic permeability \cite{asad2019maxwell}). 
}

\subsection{Dual field Stokes-Dirac structure}
Now we will present the dual field Stokes-Dirac structure which will be constructed by combining the Stokes-Dirac structure \eqref{eq:StokesDirac} and its adjoint \eqref{eq:AdjStokesDirac}. In order to obtain the weak form, we take the inner product of \eqref{eq:StokesDirac} and \eqref{eq:AdjStokesDirac} with test differential forms of appropriate degrees and use the integration by parts formula \eqref{eq:intbyparts_codif} for the adjoint system. \revTwo{This gives rise to two uncoupled systems. One formulation contains outer-oriented forms only, except for the boundary term arising from the integration by parts. We will refer to this formulation as the primal formulation. The second formulation contains inner-oriented forms only, except once again for the boundary term. We will refer to this formulation as the dual formulation.} 
\revTwo{
\paragraph{Primal Stokes-Dirac structure}
The primal Stokes-Dirac structure contains outer oriented forms. Considering that $p+q = n+1$, it holds $n-q=p-1$ and $n-p=q-1$. This shows that the primal formulation corresponds to a first mixed discretization in which the variables are related by $\d^{p-1}$

\begin{equation}\label{eq:primal_SD}
    \begin{aligned}
    \inpr[M]{\dual{v}^p}{\dual{f}^p_1} &= (-1)^r\inpr[M]{\dual{v}^p}{\d \dual{e}^{p-1}_2}, \\
      \inpr[M]{\dual{v}^{p-1}}{\dual{f}^{p-1}_2} &= (-1)^{a_1} \inpr[M]{\d\dual{v}^{p-1}}{\dual{e}^p_1} - (-1)^{a_1} \dualpr[\partial M]{\dual{v}^{p-1}}{f_\partial^{q-1}}, \\
      \dual{e}_\partial^{p-1} &= (-1)^p \tr \dual{e}_2^{p-1} \in H^{-1/2}\dual{\Omega}^{p-1}(\partial M),
    \end{aligned} \qquad
    \begin{aligned}
    &\forall \dual{v}^p \in H\dual{\Omega}^p(M), \\
    &\forall \dual{v}^{p-1} \in H\dual{\Omega}^{p-1}(M),\\
    &
    \end{aligned}
\end{equation}
where in the integration by parts, the fact that $H^{-1/2}{\Omega}^{q-1}(\partial M) \ni f_\partial^{q-1} = \tr e_1^{q-1} = \tr \star \dual{e}_1^p$ has been used. The fractional Sobolev spaces $H^{-1/2}{\Omega}^{q-1}(\partial M), \; H^{-1/2}\dual{\Omega}^{p-1}(\partial M)$ arise from the fact that the trace operator on $H^1\Omega^k(M)$ (the space  of $k$-forms whose coefficients are in $H^1(M)$) extends to bounded linear operator $\tr: H\Omega^k(M) \rightarrow H^{-1/2}\Omega^{k}(\partial M)$ \cite[Theorem 6.3]{arnold2018finite}.

\paragraph{The dual Stokes-Dirac structure}
The dual Stokes-Dirac structure contains inner oriented forms. This formulation corresponds to a second mixed discretization in which the variables are related by $\d^{q-1}$ 

\begin{equation}\label{eq:dual_SD}
    \begin{aligned}
    \inpr[M]{{v}^{q-1}}{{f}^{q-1}_1} &= -\inpr[M]{\d{v}^{q-1}}{{e}_2^q} - (-1)^{(p-1)(q-1)} \dualpr[\partial M]{{v}^{q-1}}{\dual{e}_\partial^{p-1}}, \\
    \inpr[M]{v^q}{f^q_2} &= \inpr[M]{v^q}{\d e^{q-1}_1}, \\
    f_\partial^{q-1} &= \tr e_1^{q-1} \in H^{-1/2}\Omega^{q-1}(\partial\Omega),
    \end{aligned} \qquad
    \begin{aligned}
    &\forall {v}^{q-1} \in H\Omega^{q-1}(M), \\
    &\forall v^q \in H\Omega^q(M). \\
    &
    \end{aligned}
\end{equation}
where it has been used $a_0 \equiv 1, \mod 2$ and $\tr \star e_2^q= (-1)^{q(n-q)} \tr e_2^{p-1} = (-1)^{p+q(n-q)} e_\partial^{p-1}$.
}

\subsection{Boundary conditions enforcement}\label{subsection:weak_bcs}
\revTwo{The primal and dual Stokes-Dirac structures have an associated primal and dual port-Hamiltonian system respectively. The Stokes-Dirac structures do not carry boundary conditions information but the associated port-Hamiltonian systems do. A number of different techniques can be used to enforce boundary conditions (cf. \cite{benner2015timebc} and references therein). We opt for a strong imposition of the essential boundary condition (i.e. the one that does not appear in the integration by parts) by direct assignment of the boundary degrees of freedom in each resulting port-Hamiltonian system. In this way two uncoupled port-Hamiltonian systems are obtained. Furthermore, this strategy will ensure that the power flow is correctly represented (up to the approximation error introduced by the polynomial interpolation) as it will be proven later.}

\revTwo{
\subsection{Strong primal and dual port-Hamiltonian system} 
In this section, the primal-dual port-Hamiltonian system will be constructed. For sake of clarity the systems are first presented in strong form and then the associated weak form is presented.

\paragraph{Strong primal port-Hamiltonian system}
Using the first line of \eqref{eq:pHsys} and the second line of \eqref{eq:AdjPHsys}, together with the expression of the state variables in terms of the effort in Eq. \eqref{eq:effort_CE} and Pr. \ref{pr:parity_a0_a1}, the outer oriented primal system is obtained: 
 \begin{equation}\label{eq:strong_primalPH}
    \begin{bmatrix}
        \dual{C}^p & 0 \\
        0 & \dual{E}^{p-1}
    \end{bmatrix}
    \diffp{}{t}\begin{pmatrix}
        \dual{e}^p_1 \\
        \dual{e}^{p-1}_2 
    \end{pmatrix} = (-1)^{p}
    \begin{bmatrix}
        0 & \d \\
        -\d^* & 0\\
    \end{bmatrix}
    \begin{pmatrix}
        \dual{e}^p_1 \\
        \dual{e}^{p-1}_2 
    \end{pmatrix}, \qquad
    \begin{aligned}
        \tr \star \dual{e}_1^p|_{\Gamma_1} = u_1^{q-1}, \\
        (-1)^p\tr \dual{e}_2^{p-1}|_{\Gamma_2} = \dual{u}_2^{p-1}, \\
    \end{aligned}
 \end{equation}
 where $\dual{e}_1^p \in \dual{\Omega}^p(M)$ and $\dual{e}_2^{p-1} \in \dual{\Omega}^{p-1}(M)$ and Proposition \ref{pr:parity_a0_a1} has been used, together with the fact that $(-1)^{r+1+p(n-p)}= (-1)^p$.

\paragraph{Strong dual port-Hamiltonian system}
Using the first line of \eqref{eq:AdjPHsys} and the second line of \eqref{eq:pHsys}, together with the expression of the state variables in terms of the effort in Eq. \eqref{eq:effort_CE}, the inner oriented dual system is obtained: 
\begin{equation}\label{eq:strong_dualPH}
    \begin{bmatrix}
        {C}^{q-1} & 0 \\
        0 & {E}^{q}
    \end{bmatrix}
    \diffp{}{t}\begin{pmatrix}
        {e}^{q-1}_1 \\
        {e}^{q}_2 
    \end{pmatrix} =
    \begin{bmatrix}
        0 & \d^* \\
        -\d & 0\\
    \end{bmatrix}
    \begin{pmatrix}
        {e}^{q-1}_1 \\
        {e}^{q}_2 
    \end{pmatrix}, \qquad
    \begin{aligned}
        \tr {e}_1^{q-1}|_{\Gamma_1} = u_1^{q-1}, \\
        (-1)^{p+q(n-q)}\tr \star {e}_2^{q}|_{\Gamma_2} = \dual{u}_2^{p-1}, \\
    \end{aligned}
 \end{equation}
 where ${e}_1^{q-1} \in {\Omega}^{q-1}(M)$ and ${e}^{q}_2  \in {\Omega}^q(M)$ and Proposition \ref{pr:parity_a0_a1} has been used.

\subsection{Weak primal and dual port-Hamiltonian system} 

The weak formulation for the primal and dual system is here presented. The material tensors are supposed to be positive and bounded above and below operators from $L^2\Omega^k(M)$ to $L^2\Omega^k(M)$ (analogously to what is done in \cite[Sec. 7.3]{arnold2006acta}) and are therefore allowed to be discontinuous. For this reason, a weak formulation of the dynamics is necessary. The codifferential is here interpreted weakly. This will allow incorporating the natural boundary condition as a system input.

\paragraph{The weak primal port-Hamiltonian system}
The weak formulation for the primal system reads: find $\dual{e}^p_1 \in H\dual{\Omega}^{p}(M), \; \dual{e}^{p-1}_2 \in H\Omega^{q}(M)$ such that $(-1)^p \tr \dual{e}^{p-1}_2 \vert_{\Gamma_2} = \dual{u}^{p-1}_2 \in H^{-1/2}\Omega^{p-1}(\Gamma_2)$ and
 \begin{equation}\label{eq:weak_primalPH}
    \begin{aligned}
    \inpr[M]{\dual{v}^p}{\dual{C}^p \partial_t \dual{e}^p_1} &= (-1)^{p}\inpr[M]{\dual{v}^p}{\d \dual{e}^{p-1}_2}, \\
     \inpr[M]{\dual{v}^{p-1}}{\dual{E}^{p-1} \partial_t \dual{e}^{p-1}_2} &= (-1)^{p} \{- \inpr[M]{\d\dual{v}^{p-1}}{\dual{e}^p_1} + \dualpr[\Gamma_1]{\dual{v}^{p-1}}{u^{q-1}_1}\},
    \end{aligned}  \qquad 
    \begin{aligned}
    &\forall \dual{v}^p \in H\dual{\Omega}^p(M), \\
    &\forall \dual{v}^{p-1} \in H\dual{\Omega}^{p-1}(M, \Gamma_2), \\
    \end{aligned}
 \end{equation}
where $u_1^{q-1} \in H^{-1/2}\Omega^{q-1}(\Gamma_1)$ and $H\dual{\Omega}^{p-1}(M, \Gamma_2)$ is a Sobolev spaces with boundary conditions 
\begin{equation}
        H\dual{\Omega}^{p-1}(M, \Gamma_2) := \{\dual{w}^{p-1} \in H\dual{\Omega}^{p-1}(M) \vert \; \tr \dual{w}^{p-1}\vert_{\Gamma_2} = 0 \}. 
\end{equation}

\paragraph{The weak dual port-Hamiltonian system}
For the dual system, the following weak formulation is obtained: find ${e}^{q-1}_1 \in H{\Omega}^{p}(M), \; {e}^{q}_2 \in H\Omega^{q}(M)$ such that $\tr e^{q-1}_1 \vert_{\Gamma_1} = u_1^{q-1} \in H^{-1/2}\Omega^{q-1}(\Gamma_1), $ and
 \begin{equation}\label{eq:weak_dualPH}
    \begin{aligned}
    \inpr[M]{{v}^{q-1}}{{C}^{q-1} \partial_t {e}^{q-1}_1} &= \inpr[M]{\d {v}^{q-1}}{ {e}^{q}_2} +(-1)^{(p-1)(q-1)}\dualpr[\Gamma_2]{v^{q-1}}{\dual{u}_2^{p-1}}, \\
     \inpr[M]{{v}^{q}}{{E}^{q} \partial_t {e}^{q}_2} &= -\inpr[M]{{v}^{q}}{\d{e}^{q-1}_1},
    \end{aligned}  \qquad 
    \begin{aligned}
    &\forall {v}^{q-1} \in H{\Omega}^{q-1}(M, \Gamma_1), \\
    &\forall {v}^{q} \in H{\Omega}^{q}(M), \\
    \end{aligned}
 \end{equation}
where $\dual{u}^{p-1}_2 \in H^{-1/2}\dual{\Omega}^{p-1}(\Gamma_2)$ and $H{\Omega}^{p-1}(M, \Gamma_1)$  is defined by
\begin{equation}
        H{\Omega}^{q-1}(M, \Gamma_1) := \{{w}^{q-1} \in H{\Omega}^{q-1}(M) \vert \; \tr {w}^{q-1}\vert_{\Gamma_1} = 0 \}. 
\end{equation}
}

\section{Space discretization}\label{sec:space_discr}

In this section, the space discretization of the dual-field port-Hamiltonian model is performed using the trimmed polynomial family (whose construction is discussed in \revTwo{\cite{arnold2006acta,arnold2018finite}}).
First, we start by introducing the discrete operators on the finite-element space which will mimic their continuous counterparts.
The crucial point is that the employed finite element forms form a discrete de Rham complex, as stated in \cite[Lemma 3.8]{arnold2006acta}. \\

For notational simplicity the trimmed polynomial space will be denoted by $\mathcal{V}_{s, h}^k = \mathcal{P}^-_s\Omega^k(\mathcal{T}_h)$, where $k$ is the degree of the differential form, $s$ the polynomial degree associated to the interpolating bases functions and $h$ the mesh size.
\revTwo{Furthermore, the notation $[\mathbf{A}]_i^j$ will be used to denoted the element of matrix $\mathbf{A}$ corresponding to the $i$-th row and the $j$-th column.}

\subsection{Mimetic operations}
For a generic $k$-discrete form $\mu_{h}^k \in \mathcal{V}_{s, h}^k$, one has 
\begin{equation}\label{eq:trimmed_basis}
    \mu_{h}^k = \sum_{i=1}^{N_{k, s}} \phi_{s, i}^{k}(\xi) \mu_i^k, 
\end{equation}
where $N_{k, s}$ is the number of degrees of freedom for $\mathcal{V}_{s, h}^k$, the degree of freedom $\mu_i^k \in \bbR$ is given \revTwo{in \cite[Eq. 5.2]{arnold2006acta},} and $\phi_{s, i}^{k} : M \rightarrow  \mathcal{V}_{s, h}^k \subset H\Omega^k(M)$ is a section of $\mathcal{V}_{s, h}^k$, corresponding to a finite element basis function. \revTwo{For an outer oriented form $\dual{\mu}_{h}^k \in \dual{\mathcal{V}}_{s, h}^k$ the following notation is used for the finite element expansion
\begin{equation}\label{eq:trimmed_basis_hat}
    \dual{\mu}_{h}^k = \sum_{i=1}^{N_{k, s}} \dual{\phi}_{s, i}^{k}(\xi) \dual{\mu}_i^k, 
\end{equation}
}

\paragraph{Inner product}
Given two discrete forms $(\nu_h^k, \mu_h^k) \in \mathcal{V}_{s, h}^k$, the inner product reads
\begin{equation}\label{eq:alg_inner}
    \inpr[M]{\nu_h^k}{\mu_h^k} = (\bm{\nu}^k)^\top \mathbf{M}^k_s \bm{\mu}^k,
\end{equation}
where $\bm{\nu}^k, \; \bm{\mu}^k \in \bbR^{N_{k, s}}$ are the vectors collecting the degrees of freedom $\mu_i^k, \; \mu_i^k$ respectively and the mass matrix $\mathbf{M}^k_s \in \mathbb{R}^{N_{k, s}\times N_{k, s}}$ of order $k$ (symmetric and positive definite) is computed as  
$$[\mathbf{M}^k]_{i}^j = \inpr[M]{\phi_{s, i}^{k}}{\phi_{s, j}^{k}}.$$

\paragraph{Duality product over the domain}
Given two discrete forms variables $\lambda_h^{n-k} \in \mathcal{V}_{s, h}^{n-k}, \; \mu_h^k \in \mathcal{V}_{s, h}^{k}$, their duality product over the \revOne{domain} is expressed as
\begin{equation}\label{eq:alg_dualpr}
    \dualpr[M]{\lambda_h^{n-k}}{\mu_h^k} = (\bm{\lambda}^{n-k})^\top \mathbf{L}^k_s \bm{\mu}^k,
\end{equation}
where the matrix $\mathbf{L}^k_s \in \mathbb{R}^{N_{n-k, s}\times N_{k, s}}$ is computed via $[\mathbf{L}^k_s]_{i}^j = \dualpr[M]{\phi_{s, i}^{n-k}}{\phi_{s, j}^k}$. Since the wedge product, used in the definition of dual product introduced in Eq.(\ref{eq:dual_pr}), is \revTwo{alternating}, 
it  could be seen that the following holds
\begin{equation}
    \mathbf{L}^k_s = (-1)^{k(n-k)} (\mathbf{L}^{n-k}_s)^\top,
\end{equation}
\revTwo{where the matrix $\mathbf{L}^{n-k}_s \in \mathbb{R}^{N_{k, s}\times N_{n-k, s}}$ is computed via $[\mathbf{L}^{n-k}_s]_{i}^j = \dualpr[M]{\phi_{s, i}^k}{\phi_{s, j}^{n-k}}$.}

\paragraph{Exterior derivative}
The expression of the exterior derivative is here specialized for the inner and duality product. Given a form $\nu_h^{k+1} \in \mathcal{V}_{s, h}^{k+1}$ and $\mu_h^k \in \mathcal{V}_{s, h}^{k}$ ($k\le n-1$), the inner product of $\nu_h^{k+1}$ and $\d \mu_h^k$ is expressed by
\begin{equation}\label{eq:alg_d_inner}
    \inpr[M]{\nu_h^{k+1}}{\d\mu_h^k} = (\bm{\nu}^{k+1})^\top \mathbf{D}^{k}_s \bm{\mu}^{k}
\end{equation}
where $\mathbf{D}^{k}_s \in \mathbb{R}^{N_{k+1, s} \times N_{k, s}}$ is computed as  $[\mathbf{D}^{k}_s]_{i}^j = \inpr[M]{\phi_{s, i}^{k+1}}{\d\phi_{s, j}^{k}}.$ Given $\lambda_h^{n-k-1} \in \mathcal{V}_{s, h}^{n-k-1}$ the duality product between $\lambda_h^{n-k-1}$ and $\mu_h^k$ reads
\begin{equation}\label{eq:alg_d_dualpr}
    \dualpr[\partial M]{\lambda_h^{n-k-1}}{\d \mu_h^k} = (\bm{\lambda}^{n-k-1})^\top \mathbf{G}^k_s \bm{\mu}^k, 
\end{equation}
where $\mathbf{G}^{k}_s \in \mathbb{R}^{N_{n-k-1, s} \times N_{k, s}}$ is computed as  $[\mathbf{G}^{k}_s]_{i}^j = \dualpr[M]{\phi_{s, i}^{n-k-1}}{\d\phi_{s, j}^{k}}.$ 

\revTwo{
\begin{remark}
The duality product over the domain is not used by the numerical scheme. However, the scheme is such to preserve the integration by parts formula \eqref{eq:int_byparts_d} that contains it.
\end{remark}}

\paragraph{Trace}
For a discrete differential form of order $k \le n-1$ the trace operator is also defined. Since the basis forms are defined locally for each element of the mesh, when the trace of a form is considered the simplices that do not lie on the boundary can be discarded in the expansion. Denoting the boundary of the simplicial complex by
$\Delta_j(\partial \mathcal{T}_h) \subset \Delta_j(\mathcal{T}_h)$, the trace of a discrete $k$-form $\mu_h \in \mathcal{V}_{s, h}^k$ reads
\begin{equation}
    \tr \mu_h^k = \sum_{i=1}^{N_{k, s}} \tr(\phi_{s, i}^{k}(\xi)) \mu_i^k = \sum_{l=1}^{N_{k, s}^\partial} \psi_{s, l}^{k}(\xi) \mu_{l, \partial}^k.
\end{equation}
where $N_{k, s}^\partial = \sum_{j=k}^{n-1}\# \Delta_j(\partial \mathcal{T}_h)$ is the number of all $j$-simplices (with $k\le j\le n-1$) along the boundary for a polynomial differential form of order $k$. The degrees of freedom along the boundary $\mu_{l, \partial}^k$ are associated to $j-$simplices $\sigma_{j, \partial}^l, \; k\le j\le n-1$ lying on the boundary. The trace matrix simply collects them
\begin{equation}\label{eq:alg_trace}
    \bm{\mu}_\partial^k = \mathbf{T}^k_s \bm{\mu}^k, \qquad 
    [\mathbf{T}^k_s]_{l}^i = 
\begin{cases}
    1, \quad \text{if } \sigma_{j, \partial}^l \equiv \sigma_{j}^i \implies \psi_l^{k}(\xi) \equiv \tr(\phi_{s, j}^{k}(\xi))\\
    0, \quad \text{otherwise},
\end{cases} \qquad 
\begin{aligned}
    \forall l = 1, \dots, N_{k, s}^\partial, \\
    \forall i = 1, \dots, N_{k, s}.
\end{aligned}  
\end{equation}

\paragraph{Duality product over the boundary}
Given two discrete forms variables $\lambda_h \in \mathcal{V}_{s, h}^{n-k-1}, \; \mu_h^k \in \mathcal{V}_{s, h}^{k}$, their boundary duality product is expressed as
\begin{equation}\label{eq:alg_dualpr_bd}
    \dualpr[\partial M]{\lambda_h^{n-k-1}}{\mu_h^k} = (\bm{\lambda}^{n-k-1})^\top \mathbf{L}^k_{s, \partial} \bm{\mu}^k,
\end{equation}
where the matrix $\mathbf{L}^k_{s, \partial} \in \mathbb{N}^{N_{n-k-1, s}\times N_{k, s}}$ is computed via $[\mathbf{L}^k_{s, \partial}]_{i}^j = \dualpr[\partial M]{\phi_{s, i}^{n-k-1}}{\phi_{s, j}^k}$ and satisfies
\begin{equation}
    \mathbf{L}^k_{s, \partial} = (-1)^{k(n-k-1)} (\mathbf{L}^{n-k-1}_{s, \partial})^\top.
\end{equation}
Using the trace operator relation \eqref{eq:alg_trace}, the duality product can be expressed using the basis functions that lie on the boundary, i.e. $\psi_{s, l}^k$
\begin{equation}
\label{eq:alg_wedge_bd}
    \dualpr[\partial M]{\lambda_h^{n-k-1}}{\mu_h^k} = (\bm{\lambda}_\partial^{n-k-1})^\top \mathbf{\Psi}^k_{s, \partial} \bm{\mu}^k_\partial,
\end{equation}
where the boundary matrix $\mathbf{\Psi}^k_{s, \partial} \in \mathbb{R}^{N_{n-k-1, s}^\partial \times N_{k, s}^\partial}$ is given by $[\mathbf{\Psi}^k_{s, \partial}]_{l}^m = \dualpr[\partial M]{\psi^{n-k-1}_l}{\psi^k_m}$ and satisfies $\mathbf{\Psi}^k_{s, \partial} = (-1)^{k(n-k-1)}(\mathbf{\Psi}^{n-k-1}_{s, \partial})^\top$. Using Eq. \eqref{eq:alg_trace}, the following matrix relation is obtained
\begin{equation}
    \mathbf{L}^k_{s, \partial} = (\mathbf{T}^{n-k-1}_s)^\top \mathbf{\Psi}^k_{s, \partial} \mathbf{T}^k_s.
\end{equation}

\revTwo{
\paragraph{Discrete integration by parts}
The integration by parts formula \eqref{eq:int_byparts_d} is also valid when the differential forms and the domain are not smooth (cf. \cite[Theorem 6.3]{arnold2018finite})
\begin{equation}\label{eq:int_byparts_d_H}
    \dualpr[M]{\d{\mu}}{\lambda} + (-1)^k \dualpr[M]{\mu}{\d{\lambda}} = \dualpr[\partial M]{\mu}{\lambda}, \qquad \mu \in H^1\Omega^{k}(M), \quad \lambda \in H \Omega^{n-k-1}(M),
\end{equation}
where $H^1\Omega^{k}(M)$ is the space of $k$-forms with coefficients in $H^1(M)$. Conforming finite element $\mathcal{V}_{s,h}^k \subset H\Omega^k(M)$ do not possess the $H^1\Omega^{k}(M)$ regularity (except for the case $k=0$). However, the integration by parts holds at the discrete level when conforming finite element spaces are used
\begin{equation}\label{eq:int_byparts_d_disc}
    \dualpr[M]{\d\mu_h}{\lambda_h} + (-1)^k\dualpr[M]{\mu_h}{\d\lambda_h} = \dualpr[\partial M ]{\mu_h}{\lambda_h}, \qquad  \forall \mu_h \in \mathcal{V}_{s,h}^k, \; \forall\lambda_h \in \mathcal{V}_{s,h}^{n-k-1}.
\end{equation}
}
\revOne{
This is proven considering that the Stokes theorem holds for the each element of the mesh and it extends to the whole mesh as all terms arising on inter-cell boundaries will cancel due to the special continuity properties of discrete differential forms. A proof of this statement based on vector calculus is reported in \ref{app:proofs} for the wave and Maxwell equations in 3D. A formal proof of the integration by parts formula in an exterior calculus setting is out of the scope of this paper. Since Eq. \eqref{eq:int_byparts_d_disc} is valid $\forall \mu_h, \forall \lambda_h$, the algebraic form of the Stokes theorem is obtained
\begin{equation}\label{eq:alg_StokesTh}
(-1)^{(k+1)(n-k-1)}(\mathbf{G}^{k})^\top + (-1)^k \mathbf{G}^{n-k-1} = (\mathbf{T}^{k})^\top \bm{\Psi}_{s, \partial}^{n-k-1} \mathbf{T}^{n-k-1},
\end{equation}
where the alternating property of the duality product has been used for the first term.
}
\paragraph{Hodge star operator}
As already discussed, in this paper a discrete Hodge star operator will never be employed thanks to the dual-field formulation of the system dynamics. The following discussion will highlight the benefits of not using a discrete Hodge star which is in fact a non trivial task \cite{hiptmair2001}.

Essentially, an isomorphic discrete Hodge star operator requires dual meshes.
The explanation for the statement resides in the fact that finite element differential forms of dual type do not possess in general the same number of degrees of freedom. For example, in the Whitney forms case each $k$-order finite element has a dimension equal to the number of $k$-simplices. In general, the number of $k$-simplices will be different than the number of $n-k$ simplices. Hence it is not possible to find a bijection relating a finite element $k$-form and its $n-k$ Hodge dual. To retrieve a discrete isomorphism, dual meshes based on the Voronoi-Delaunay duality are normally used \cite{hirani2003discrete,hiptmair2001}. However,  \revTwo{in some discrete formulations} this approach result in mesh entities located outside the physical domain, making it cumbersome to implement boundary interconnections of different systems, which is a fundamental feature of port-Hamiltonian systems that we aim to preserve at the discrete level.

If one mesh is used, a non-isomorphic discrete Hodge star can be constructed relying on a weak formulation. Given $\mu_h^k \in \mathcal{V}_{s, h}^{k}$ and its dual variable $\star \mu_h^k = \dual{\mu}_h^{n-k} \in \dual{\mathcal{V}}_{s, h}^{n-k}$ a discrete Hodge can be constructed using a weak formulation of the equation 

\begin{equation}
    \inpr[M]{\dual{v}^{n-k}_h}{\dual{\mu}_h^{n-k}} = (-1)^{k(n-k)} \dualpr[M]{\dual{v}^{n-k}_h}{\mu_h^k}, \qquad \forall \dual{v}^{n-k}_h \in \dual{\mathcal{V}}_{s, h}^{n-k},
\end{equation}
or symmetrically
\begin{equation}
    \dualpr[M]{v^k_h}{\dual{\mu}_h^{n-k}} =\inpr[M]{v^k_h}{\mu_h^k}, \qquad \forall v^k_h \in \mathcal{V}_{s, h}^{k},
\end{equation}
leading to the two following algebraic realization
\begin{equation}\label{eq:alg_hodge}
    \mathbf{M}^{n-k}_s \dual{\bm{\mu}}^{n-k} = (-1)^{k(n-k)} \mathbf{L}^{k}_s {\bm{\mu}}^k, \qquad \mathbf{L}^{n-k}_s \dual{\bm{\mu}}^{n-k}  = \mathbf{M}^{k}_s \bm{\mu}^k.
\end{equation}
It can be noticed that this discrete Hodge corresponds to a projection between dual polynomial space and therefore it inevitably introduces an error. This fundamental point is the main motivation behind the dual field method. The integration of the primal and adjoint system allows reconstructing primal and dual variables without relying on a discrete Hodge star. Instead, it is embedded into the adjoint system by means of the codifferential operator.

\paragraph{The Whitney form case}
Finally, we conclude this section by considering the special case of Whitney forms i.e. lowest-degree polynomials ($s=1$). The special case of the Whitney forms allows for a clear separation of topological and metric operations: the resulting discrete Stokes-Dirac structure is completely topological, whereas the adjoint Stokes-Dirac systems embeds the metric in the discrete representation of the codifferential \cite{bochev2006}.
 
For a generic $k$-Whitney form $\mu_h \in \mathcal{V}_{1, h}^k:= \mathcal{W}_h^k$, one has 
\begin{equation}\label{eq:whitney_basis}
    \mu_h = \sum_{i=1}^{N_{k}} w_i^{k}(\xi) \mu_i, \qquad \mu_i := \int_{\sigma_k^i} \tr_{\sigma_k^i} \mu_h, 
\end{equation}
where $N_k = \# \Delta_{k}(\mathcal{T}_h), \; \forall k=0, \dots, n$ is the number of $k$-simplices in the mesh, $\mu_i\in \bbR$ corresponds to the cochain coefficient associated with the simplex $\sigma_k^i$ and the function $w_i^{k}(\xi):= \phi_{1, i}^k(\xi)$ is the $k$-th order Whitney function related to the $k$-simplex indexed by $i$ and function of the point manifold $\xi$. The particularity of the Whitney form case is the exterior derivative, since only in this case the interpolation commutes with the exterior derivative (cf. \cite[Section IV.27]{whitney1957}). The exterior derivative of a $k$-Whitney form $\mu_h \in \mathcal{W}_h^k$ can be rewritten in terms of the \revTwo{incidence} matrix $\mathbf{d}^k \in \mathbb{R}^{N_{k+1}\times N_k}$ \revTwo{(the discrete representation of the exterior derivative \cite{hiptmair2001,kreeft2011mimetic})}
\begin{equation}\label{eq:alg_d}
\begin{aligned}
    \d \mu_h = \sum_{i=1}^{N_k} \d w_i^{k}(\xi) \mu_i = \sum_{i=1}^{N_{k+1}} \sum_{j=1}^{N_{k}} w_i^{k+1}(\xi) [\mathbf{d}^{k}]_{i}^j \mu_j, \qquad \forall k = 0, \dots, n-1.
\end{aligned}
\end{equation}
This implies that the incidence matrix  explicitly appears in the inner product
\revOne{
\begin{equation}\label{eq:alg_d_inner_W}
    \mathbf{D}_1^{k} = \mathbf{M}^{k+1}_1 \mathbf{d}^{k}, \qquad \forall k=0, \dots, n-1,
\end{equation}
where $\mathbf{D}_1^{k}$ is defined in \eqref{eq:alg_d_inner}, 
and in the duality product
\begin{equation}\label{eq:alg_d_dualpr_W}
    \mathbf{G}_1^{k} = \mathbf{L}^{k+1}_1 \mathbf{d}^{k}, \qquad \forall k=0, \dots, n-1,
\end{equation}
where $\mathbf{G}_1^{k}$ is defined in \eqref{eq:alg_d_dualpr}.
}

\revTwo{\subsection{Discrete primal and dual Stokes-Dirac structure}}
\revTwo{In this section the discrete representation of the primal Stokes-Dirac structure \eqref{eq:primal_SD} and the dual Dirac structure \eqref{eq:dual_SD} are presented. Each of the two will verify a power balance involving inner products. By combining primal and dual variables a discrete Stokes theorem involving only duality products is obtained. This means that proposed discretization mimics the power balance in Eq. \eqref{eq:powbal}  at the discrete level.

\subsubsection{Discrete primal Stokes-Dirac structure}
The discrete version of the primal  Stokes-Dirac structure \eqref{eq:primal_SD} is obtained by replacing the flows and efforts with their discrete counterparts:\\
\begin{equation}\label{eq:discr_primal_SD}
    \begin{aligned}
    \inpr[M]{\dual{v}^p_h}{\dual{f}^p_{1, h}} &= \inpr[M]{\dual{v}^p_h}{(-1)^r\d \dual{e}^{p-1}_{2, h}}, \\
      \inpr[M]{\dual{v}^{p-1}_h}{\dual{f}^{p-1}_{2, h}} &= (-1)^{a_1} \inpr[M]{\d\dual{v}^{p-1}_h}{ \dual{e}^p_{1, h}} - (-1)^{a_1} \dualpr[\partial M]{\dual{v}^{p-1}_h}{f_{\partial, h}^{q-1}}, \\
      \dual{e}_{\partial, h}^{p-1} &= (-1)^p \tr \dual{e}_{2, h}^{p-1} \in \tr \dual{\mathcal{V}}_{s, h}^{p-1}.
    \end{aligned} \qquad
    \begin{aligned}
    &\forall \dual{v}^p_h \in \dual{\mathcal{V}}_{s, h}^p, \\
    &\forall \dual{v}^{p-1}_h \in \dual{\mathcal{V}}_{s, h}^{p-1},\\
    &
    \end{aligned}
\end{equation}
The boundary flow $f_{\partial, h}^{q-1}$ is obtained by considering the trace of the associated polynomial family $f_{\partial, h}^{q-1}~\in~\tr \mathcal{V}_{s, h}^{q-1}$. 
\begin{proposition}\label{pr:discr_innerpowbal_primal}
The system \eqref{eq:discr_primal_SD} verifies the following power balance 
\begin{equation}\label{eq:discr_powbal_primal} 
(-1)^{p(n-p)}\inpr[M]{\dual{e}^p_{1, h}}{\dual{f}^p_{1, h}} + (-1)^{q(n-q)} \inpr[M]{\dual{e}^{p-1}_{2, h}}{\dual{f}^{p-1}_{2, h}} + \dualpr[\partial M]{\dual{e}^{p-1}_{\partial, h}}{f_{\partial, h}^{q-1}} = 0.
\end{equation}
\end{proposition}

\begin{proof}
Choosing $\dual{v}^p_h = \dual{e}^p_{1, h}, \; \dual{v}^{p-1}_h = \dual{e}^{p-1}_{2, h}$ leads to 
\begin{align}
      \inpr[M]{\dual{e}^p_{1, h}}{\dual{f}^p_{1, h}} &= (-1)^r \inpr[M]{\dual{e}^p_{1, h}}{\d \dual{e}^{p-1}_{2, h}}, \label{eq:primal_1}\\
      \inpr[M]{\dual{e}^{p-1}_{2, h}}{\dual{f}^{p-1}_{2, h}} &=(-1)^{a_1} \inpr[M]{\d \dual{e}^{p-1}_{2, h}}{\dual{e}^p_{1, h}} - (-1)^{a_1}\dualpr[\partial M]{\dual{e}_{2, h}^{p-1}}{f_{\partial, h}^{q-1}}. \label{eq:primal_2}
\end{align}
From Proposition \ref{pr:parity_a0_a1} it holds $(-1)^{a_1} = (-1)^{1+r+p(n-p)+q(n-q)}$. So, summing up Eqs. \eqref{eq:primal_1} and \eqref{eq:primal_2} with the appropriate coefficients, leads to
\begin{equation}
\begin{aligned}
    (-1)^{p(n-p)} \inpr[M]{\dual{e}^p_{1, h}}{\dual{f}^p_{1, h}} + (-1)^{q(n-q)} \inpr[M]{\dual{e}^{p-1}_{2, h}}{\dual{f}^{p-1}_{2, h}} &= (-1)^{r+p(n-p)}\dualpr[\partial M]{\dual{e}_{2, h}^{p-1}}{f^{q-1}_{\partial, h}}, \\
    &= -(-1)^{p}\dualpr[\partial M]{\dual{e}_{2, h}^{p-1}}{f^{q-1}_{\partial, h}}, \\
    &= -\dualpr[\partial M]{\dual{e}_{\partial, h}^{p-1}}{f^{q-1}_{\partial, h}}.
\end{aligned}
\end{equation}
This proves Eq. \eqref{eq:discr_powbal_primal}.
\end{proof}

Using the algebraic realization  of the inner product \eqref{eq:alg_inner}, the exterior derivative \eqref{eq:alg_d_inner} and the trace \eqref{eq:alg_trace}, the algebraic realization of the discrete primal Stokes-Dirac structure \eqref{eq:discr_primal_SD} is expressed as

\begin{equation}\label{eq:alg_primal}
\begin{aligned}
    \begin{bmatrix}
        \mathbf{M}^p_s & \mathbf{0} \\
        \mathbf{0} & \mathbf{M}^{p-1}_s
    \end{bmatrix}
    \begin{pmatrix}
    \dual{\mathbf{f}}^p_1 \\
    \dual{\mathbf{f}}^{p-1}_2
    \end{pmatrix} &=
    \begin{bmatrix}
        \mathbf{0} & (-1)^r\mathbf{D}^{p-1}_s \\
        (-1)^{a_1}(\mathbf{D}^{p-1}_s)^\top & \mathbf{0}
    \end{bmatrix}
    \begin{pmatrix}
    \dual{\mathbf{e}}^p_1 \\
    \dual{\mathbf{e}}^{p-1}_2 \\
    \end{pmatrix} - \begin{bmatrix}
        \mathbf{0} \\
        (-1)^{a_1} \mathbf{B}^{q-1}_{s}
    \end{bmatrix} \mathbf{f}^{q-1}_\partial, \\
    \dual{\mathbf{e}}^{p-1}_\partial 
    &= 
    \begin{bmatrix}
    \mathbf{0} & (-1)^p\mathbf{T}^{p-1}_s \\
    \end{bmatrix}
    \begin{pmatrix}
    \dual{\mathbf{e}}^p_1 \\
    \dual{\mathbf{e}}^{p-1}_2 \\
    \end{pmatrix}, 
\end{aligned}
\end{equation}
where $\mathbf{B}^{q-1}_{s}$ corresponds to a boundary control matrix, defined by
\begin{equation}\label{eq:control_mat}
    \mathbf{B}^{k}_s := (\mathbf{T}^{n-k-1}_s)^\top \mathbf{\Psi}^{k}_{s, \partial}, \qquad [\mathbf{B}^{k}_s]_{i}^j = \dualpr[\partial M]{\phi_{s, i}^{n-k-1}}{\psi^{k}_{s, j}}.
\end{equation}

\subsubsection{Discrete dual Stokes-Dirac structure}
The discrete counterpart of the dual Stokes-Dirac structure \eqref{eq:dual_SD} is given by
\begin{equation}\label{eq:discr_dual_SD}
    \begin{aligned}
    \inpr[M]{{v}^{q-1}_h}{{f}^{q-1}_{1, h}} &= -\inpr[M]{\d{v}^{q-1}_h}{{e}_{2, h}^q} - (-1)^{(p-1)(q-1)} \dualpr[\partial M]{{v}^{q-1}_h}{\dual{e}_{\partial, h}^{p-1}}, \\
    \inpr[M]{v^q_h}{f^q_{2, h}} &= \inpr[M]{v^q_h}{\d e^{q-1}_{1, h}}, \\
    f_{\partial, h}^{q-1} &= \tr e_{1, h}^{q-1} \in \tr \mathcal{V}_{s, h}^{q-1}.
    \end{aligned} \qquad
    \begin{aligned}
    &\forall {v}^{q-1}_h \in \mathcal{V}_{s, h}^{q-1}, \\
    &\forall v^q_h \in \mathcal{V}_{s, h}^{q}, \\
    &
    \end{aligned}
\end{equation}
In a reciprocal way with respect to the primal system, the boundary flow $\dual{e}_{\partial, h}^{p-1}$ belongs to the trace of the $p-1$ polynomial space, i.e. $\dual{e}_{\partial, h}^{p-1} \in \tr \mathcal{\dual{V}}_{s, h}^{p-1}$.
\begin{proposition}\label{pr:discr_innerpowbal_dual}
The  system \eqref{eq:discr_dual_SD} verifies the following power balance identity 
\begin{equation}\label{eq:discr_powbal_dual}
\inpr[M]{{e}^{q-1}_{1, h}}{f^{q-1}_{1, h}} + \inpr[M]{e^{q}_{2, h}}{{f}^q_{2, h}} + \dualpr[\partial M]{\dual{e}^{p-1}_{\partial, h}}{f_{\partial, h}^{q-1}} = 0. 
\end{equation}
\end{proposition}

\begin{proof}
Choosing ${v}^{q-1}_h = e^{q-1}_{1, h}, \; {v}^q_h = {e}^q_{2, h}$, leads to 
\begin{align}
      \inpr[M]{e^{q-1}_{1, h}}{{f}^{q-1}_{1, h}} &= -\inpr[M]{\d e^{q-1}_{1, h}}{{e}_{2, h}^q} - (-1)^{(p-1)(q-1)} \dualpr[\partial M]{e^{q-1}_{1, h}}{\dual{e}_{\partial, h}^{p-1}}, \label{eq:dual_1}\\
      \inpr[M]{{e}^q_{2, h}}{f^q_{2, h}} &= \inpr[M]{{e}^q_{2, h}}{\d e^{q-1}_{1, h}}.
      \label{eq:dual_2}\\
\end{align}
Summing Eqs. \eqref{eq:dual_1} and \eqref{eq:dual_2}, it is found 
\begin{equation}
\begin{aligned}
    \inpr[M]{e^{q-1}_{1, h}}{{f}^{q-1}_{1, h}} + \inpr[M]{{e}^q_{2, h}}{f^q_{2, h}} &= -(-1)^{(p-1)(q-1)} \dualpr[\partial M]{e^{q-1}_{1, h}}{\dual{e}_{\partial, h}^{p-1}}, \\
    &= -\dualpr[\partial M]{\dual{e}_{\partial, h}^{p-1}}{e^{q-1}_{1, h}}, \\ 
    &= -\dualpr[\partial M]{\dual{e}_{\partial, h}^{p-1}}{f_{\partial, h}^{q-1}}.
\end{aligned}
\end{equation}
\end{proof}

The algebraic realization of the discrete Stokes-Dirac structure \eqref{eq:discr_dual_SD} is expressed as

\begin{equation}\label{eq:alg_dual}
\begin{aligned}
    \begin{bmatrix}
        \mathbf{M}^{q-1}_s & \mathbf{0} \\
        \mathbf{0} & \mathbf{M}^q_s
    \end{bmatrix}
    \begin{pmatrix}
    \mathbf{f}^{q-1}_1 \\
    \mathbf{f}^q_2
    \end{pmatrix} &=
    \begin{bmatrix}
        \mathbf{0} & -(\mathbf{D}_s^{q-1})^\top \\
        \mathbf{D}^{q-1}_s & \mathbf{0}
    \end{bmatrix}
    \begin{pmatrix}
    \mathbf{e}^{q-1}_1 \\
    \mathbf{e}^q_2
    \end{pmatrix} - 
    \begin{bmatrix}
        (-1)^{(p-1)(q-1)}\mathbf{B}_s^{p-1}\\
        \mathbf{0}
    \end{bmatrix}\dual{\mathbf{e}}^{p-1}_\partial, \\
    \mathbf{f}_\partial^{q-1} &= \begin{bmatrix}
        \mathbf{T}_s^{q-1} & \mathbf{0}
    \end{bmatrix}\begin{pmatrix}
    \mathbf{e}^{q-1}_1 \\
    \mathbf{e}^q_2
    \end{pmatrix}.
\end{aligned}
\end{equation}

\remark{The Whitney forms case}

If the Whitney forms are considered, the algebraic system \eqref{eq:alg_primal} can be rewritten using \eqref{eq:alg_d_inner_W} as
\begin{equation}\label{eq:alg_primal_W}
\begin{aligned}
    \begin{bmatrix}
        \mathbf{M}^p_s & \mathbf{0} \\
        \mathbf{0} & \mathbf{M}^{p-1}_s
    \end{bmatrix}
    \begin{pmatrix}
    \dual{\mathbf{f}}^p_1 \\
    \dual{\mathbf{f}}^{p-1}_2
    \end{pmatrix} &=
    \begin{bmatrix}
        \mathbf{0} & (-1)^r\mathbf{M}^p_s\mathbf{d}^{p-1}_s \\
        (-1)^{a_1}(\mathbf{d}^{p-1}_s)^\top \mathbf{M}^p_s & \mathbf{0}
    \end{bmatrix}
    \begin{pmatrix}
    \dual{\mathbf{e}}^p_1 \\
    \dual{\mathbf{e}}^{p-1}_2 \\
    \end{pmatrix} - \begin{bmatrix}
        \mathbf{0} \\
        (-1)^{a_1} \mathbf{B}^{q-1}_{s}
    \end{bmatrix} \mathbf{f}^{q-1}_\partial, \\
    \dual{\mathbf{e}}^{p-1}_\partial 
    &= 
    \begin{bmatrix}
    \mathbf{0} & (-1)^p\mathbf{T}^{p-1}_s \\
    \end{bmatrix}
    \begin{pmatrix}
    \dual{\mathbf{e}}^p_1 \\
    \dual{\mathbf{e}}^{p-1}_2 \\
    \end{pmatrix},
\end{aligned}
\end{equation}
whereas the algebraic system \eqref{eq:discr_dual_SD} becomes
\begin{equation}\label{eq:alg_dual_W}
\begin{aligned}
    \begin{bmatrix}
        \mathbf{M}^{q-1}_s & \mathbf{0} \\
        \mathbf{0} & \mathbf{M}^q_s
    \end{bmatrix}
    \begin{pmatrix}
    \mathbf{f}^{q-1}_1 \\
    \mathbf{f}^q_2
    \end{pmatrix} &=
    \begin{bmatrix}
        \mathbf{0} & -(\mathbf{d}_s^{q-1})^\top \mathbf{M}^q_s\\
        \mathbf{M}^q_s \mathbf{d}^{q-1}_s & \mathbf{0}
    \end{bmatrix}
    \begin{pmatrix}
    \mathbf{e}^{q-1}_1 \\
    \mathbf{e}^q_2
    \end{pmatrix} - 
    \begin{bmatrix}
        (-1)^{(p-1)(q-1)}\mathbf{B}_s^{p-1}\\
        \mathbf{0}
    \end{bmatrix}\dual{\mathbf{e}}^{p-1}_\partial, \\
    \mathbf{f}_\partial^{q-1} &= \begin{bmatrix}
        \mathbf{T}_s^{q-1} & \mathbf{0}
    \end{bmatrix}\begin{pmatrix}
    \mathbf{e}^{q-1}_1 \\
    \mathbf{e}^q_2
    \end{pmatrix}.
\end{aligned}
\end{equation}
The Whitney forms are of particular interest as they unveil the metrical and topological structure of the equations. Notice that the first equation of \eqref{eq:alg_primal_W} and the second equation of \eqref{eq:alg_dual_W} are purely topological, as they mass matrix can be factored out.

\subsubsection{Combining the primal and dual system: discrete power balance based on duality products}
The first equation in \eqref{eq:discr_primal_SD} and the second equation in \eqref{eq:discr_dual_SD} are a mere projection of the canonical Stokes-Dirac structure \eqref{eq:StokesDirac}, i.e. the differential operator is enforced strongly. This means that by combining the primal and dual system a mimetic representation of the power balance in Eq. \eqref{eq:powbal}, which is based on duality products. This relies on the fact that finite element differential forms constitute a subcomplex of the de Rham complex.

\begin{proposition}\label{pr:discr_powbal}
The following discrete power balance, that combines variables from the primal and dual Stokes-Dirac structures holds
\begin{equation}\label{eq:discr_powbal}
    \dualpr[M]{e^{q-1}_{1, h}}{\dual{f}^p_{1, h}} + \dualpr[M]{\dual{e}^{p-1}_{2, h}}{f^q_{2, h}} + \dualpr[\partial M]{\dual{e}_{\partial, h}^{p-1}}{f_{\partial, h}^{q-1}} = 0.
\end{equation}
\end{proposition}

\begin{proof}
Consider the first equation in \eqref{eq:discr_primal_SD} and the second equation in \eqref{eq:discr_dual_SD}
\begin{equation}
    \begin{aligned}
    \inpr[M]{\dual{v}^p_h}{\dual{f}^p_{1, h}} &= \inpr[M]{\dual{v}^p_h}{(-1)^r\d \dual{e}^{p-1}_{2, h}}, \\
        \inpr[M]{v^q_h}{f^q_{2, h}} &= \inpr[M]{v^q_h}{\d e^{q-1}_{1, h}},
    \end{aligned} \qquad
    \begin{aligned}
    &\forall \dual{v}^{p}_h \in \dual{\mathcal{V}}_{s, h}^p, \\
    &\forall v^q_h \in \mathcal{V}_{s, h}^q. \\
    \end{aligned}
\end{equation}
\revOne{The trimmed polynomial family forms a subcomplex of the de Rham complex \cite[Lemma 3.8]{arnold2006acta}. This assures that $\d \dual{e}_{2, h}^{p-1} \in \mathcal{V}_{s, h}^p$ and $\d e^{q-1}_{1, h} \in \mathcal{V}_{s, h}^q$.} Since the inner product is bilinear, it holds
\begin{equation}
\begin{aligned}
    \inpr[M]{\dual{v}^p_h}{\dual{f}^p_{1, h} - (-1)^r\d \dual{e}^{p-1}_{2, h}} &= 0, \\
    \inpr[M]{v^q_h}{f^q_{2, h} - \d e^{q-1}_{1, h}} &= 0, 
\end{aligned} \qquad
    \begin{aligned}
    \forall \dual{v}^p_h \in \dual{\mathcal{V}}_{s, h}^p, \\
    \forall v^q_h \in \mathcal{V}_{s, h}^q. \\
    \end{aligned}
\end{equation}
Since the inner product is non-degenerate, one has 
\begin{equation*}
    \dual{f}^p_{1, h} = (-1)^r \d \dual{e}^{p-1}_{2, h}, \qquad f^q_{2, h} = \d e^{q-1}_{1, h}.
\end{equation*}
    Taking the duality product with $e^{q-1}_{1, h}$ and $\dual{e}^{p-1}_{2, h}$ and the two equations are summed, one obtains
    \begin{equation*}
        \dualpr[M]{e^{q-1}_{1, h}}{\dual{f}^p_{1, h}} + \dualpr[M]{\dual{e}^{p-1}_{2, h}}{f^q_{2, h}} = \dualpr[M]{e^{q-1}_{1, h}}{(-1)^r \d \dual{e}^{p-1}_{2, h}} + \dualpr[M]{\dual{e}^{p-1}_{2, h}}{\d e^{q-1}_{1, h}}.
    \end{equation*}
    The discrete Stokes theorem reported in Eq. \eqref{eq:int_byparts_d_disc} then gives
    \begin{equation*}
        \dualpr[M]{e^{q-1}_{1, h}}{(-1)^r \d \dual{e}^{p-1}_{2, h}} + \dualpr[M]{\dual{e}^{p-1}_{2, h}}{\d e^{q-1}_{1, h}} = - \dualpr[\partial M]{\dual{e}_{\partial, h}^{p-1}}{f_{\partial, h}^{q-1}},
    \end{equation*}
and the statement is proven.
\end{proof}

Proposition \ref{pr:discr_powbal} implies the following identity
\begin{equation*}
    (-1)^r(\dual{\mathbf{e}}^{p-1}_2)^\top(\mathbf{G}^{p-1}_s)^\top \mathbf{e}^{q-1}_1 + (\dual{\mathbf{e}}^{p-1}_2)^\top\mathbf{G}^{q-1}_s \mathbf{e}^{q-1}_1+ (-1)^p (\dual{\mathbf{e}}^{p-1}_2)^\top(\mathbf{T}_s^{p-1})^\top \mathbf{\Psi}^{q-1}_{s, \partial} \mathbf{T}_s^{q-1}\mathbf{e}^{q-1}_1 = 0.
\end{equation*}
Given the fact that this holds $\forall \mathbf{e}^{q-1}_1, \; \forall \dual{\mathbf{e}}^{p-1}_2$, the following matrix identity is obtained.
\begin{equation}\label{eq:alg_StokesTh_pH}
    (-1)^r(\mathbf{G}^{p-1}_s)^\top + \mathbf{G}^{q-1}_s + (-1)^p (\mathbf{T}_s^{p-1})^\top \mathbf{\Psi}^{q-1}_{s, \partial} \mathbf{T}_s^{q-1} = 0,
\end{equation}
This identity corresponds to the algebraic Stokes theorem \revOne{reported in Eq. \eqref{eq:alg_StokesTh}}.} In the Whitney forms case, the algebraic Stokes theorem in Eq. \eqref{eq:alg_StokesTh} reads
\begin{equation}\label{eq:alg_StokesTh_W}
    (-1)^{r}(\mathbf{L}^{p}_1 \mathbf{d}^{p-1})^\top + \mathbf{L}^q_1 \mathbf{d}^{q-1} + (-1)^p (\mathbf{T}^{p-1}_1)^\top \mathbf{\Psi}^{q-1}_{1, \partial} \mathbf{T}^{q-1}_1 = 0.
\end{equation}

\begin{remark}[Connection with the construction of \cite{kotyczka2018weak}]
The relation \eqref{eq:alg_StokesTh_W} corresponds to the one reported in Proposition 4 in \cite{kotyczka2018weak}. By factorizing the $\mathbf{L}$ matrices as 
$$\mathbf{L}^p_1 = (\mathbf{P}_e^{p})^\top \mathbf{P}_f^p, \qquad  \mathbf{L}^q_1 = (\mathbf{P}_e^{q})^\top \mathbf{P}_f^q, \qquad  \mathbf{L}^{q-1}_{1, \partial} = (\mathbf{T}^{p-1}_1)^\top \mathbf{S}^{q-1}_1, $$ 
where $\mathbf{S}^{q-1}_1:=\mathbf{\Psi}^{q-1}_{1, \partial} \mathbf{T}^{q-1}_1$, one can obtain an image representation of the Dirac structure in the minimal bond space (Proposition 5 in \cite{kotyczka2018weak}). From the image representation an explicit port-Hamiltonian system can be constructed. The Hodge operator is subsequently built using the projection matrix obtained from the same factorization, which allows representing the constitutive equation. Here we stick to a FEEC construction, very much in the same spirit as \cite{cardoso2020pfem}, and by the simultaneous employment of primal and adjoint system the need of using a discrete Hodge is avoided altogether.
\end{remark}

\revTwo{
\subsection{Discrete port-Hamiltonian systems}
In this section the discrete representation of the primal and dual port-Hamiltonian systems are detailed. \\

Given a matrix $\mathbf{A} \in \bbR^{n_R \times n_C}$ with $n_R$ rows and $n_C$ columns, the notation $[\mathbf{A}]_{R}^C$ (with $R = \{r_1, \dots, r_{\# R}\}, \; C = \{c_1, \dots, c_{\# C}\}$ index sets) indicates the matrix containing the rows and the columns of matrix $\mathbf{A}$ associated with the index set $R$  and $C$ respectively. Given a vector $\mathbf{x} \in \bbR^{n_R}$ the notation $[\mathbf{x}]_{R}$ indicates the vector containing the rows of vector $\mathbf{x}$ associated with the index set $R$ only. In particular in the following sections, let us denote with a slight abuse of notation the index sets corresponding to the  degrees of freedom of the interior of te domain and the subpartitions ${\Gamma}_1$ and ${\Gamma}_2$ respectively by $I,\, \Gamma_1,\, \Gamma_2$  (cf. Fig. \ref{fig:notation_dofs}). With this notation, the degrees of freedom for the inputs defined in Eq. \eqref{eq:u} are expressed as follows
\begin{equation}\label{eq:dofs_u}
    \mathbf{u}^{q-1}_1 = [\mathbf{T}^{q-1}_{s}]_{\Gamma_1} \mathbf{e}^{q-1}_1 = [\mathbf{e}^{q-1}_1]_{\Gamma_1}, \qquad
    \dual{\mathbf{u}}^{p-1}_2 = (-1)^p[\mathbf{T}^{p-1}_{s}]_{\Gamma_2}\dual{\mathbf{e}}^{p-1}_2 = (-1)^p[\dual{\mathbf{e}}^{p-1}_2]_{\Gamma_2}.
\end{equation}
The degrees of freedom of the outputs, defined in \eqref{eq:y}, are given by
\begin{align}
    \mathbf{y}^{q-1}_1  &=  [\mathbf{T}^{q-1}_{s}]_{\Gamma_2} \mathbf{e}^{q-1}_1 = [\mathbf{e}^{q-1}_1]_{\Gamma_2}, \label{eq:dofs_y1}\\
    \dual{\mathbf{y}}^{p-1}_2 &= (-1)^p [\mathbf{T}^{p-1}_{s}]_{\Gamma_1} \dual{\mathbf{e}}^{p-1}_2 = [\dual{\mathbf{e}}^{p-1}_2]_{\Gamma_1}. \label{eq:dofs_y2}
\end{align}

\begin{remark}
The index set $\Gamma_1$ for variables $\mathbf{e}^{q-1}_1$ is different from the same index set for variable $\dual{\mathbf{e}}^{p-1}_2$. For sake of notational lightness, we stick to this slightly abusive notation. Concerning the actual implementation, it is crucial that the index set associated for each of the boundary variables satisfy $\Gamma_1 \cap \Gamma_2 = \emptyset$. This requires that the geometrical entities at the intersection of the two are attributed to one or the other in an exclusive  
\end{remark}

\begin{figure}[tbh]%
\centering
\subfloat[][Physical domain and boundary splitting.]{%
	\label{fig:domain}%
	\includegraphics[width=0.46\columnwidth]{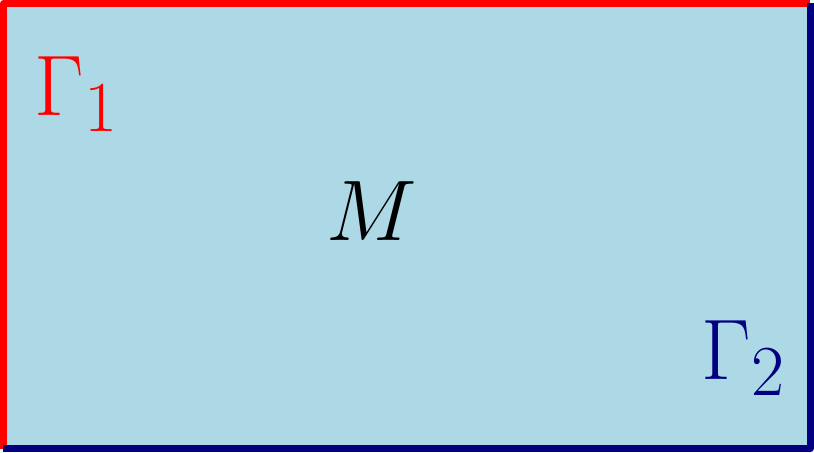}}%
\hspace{8pt}%
\subfloat[][Computational mesh used for the simulation.]{%
	\label{fig:mesh}%
	\includegraphics[width=0.51\columnwidth]{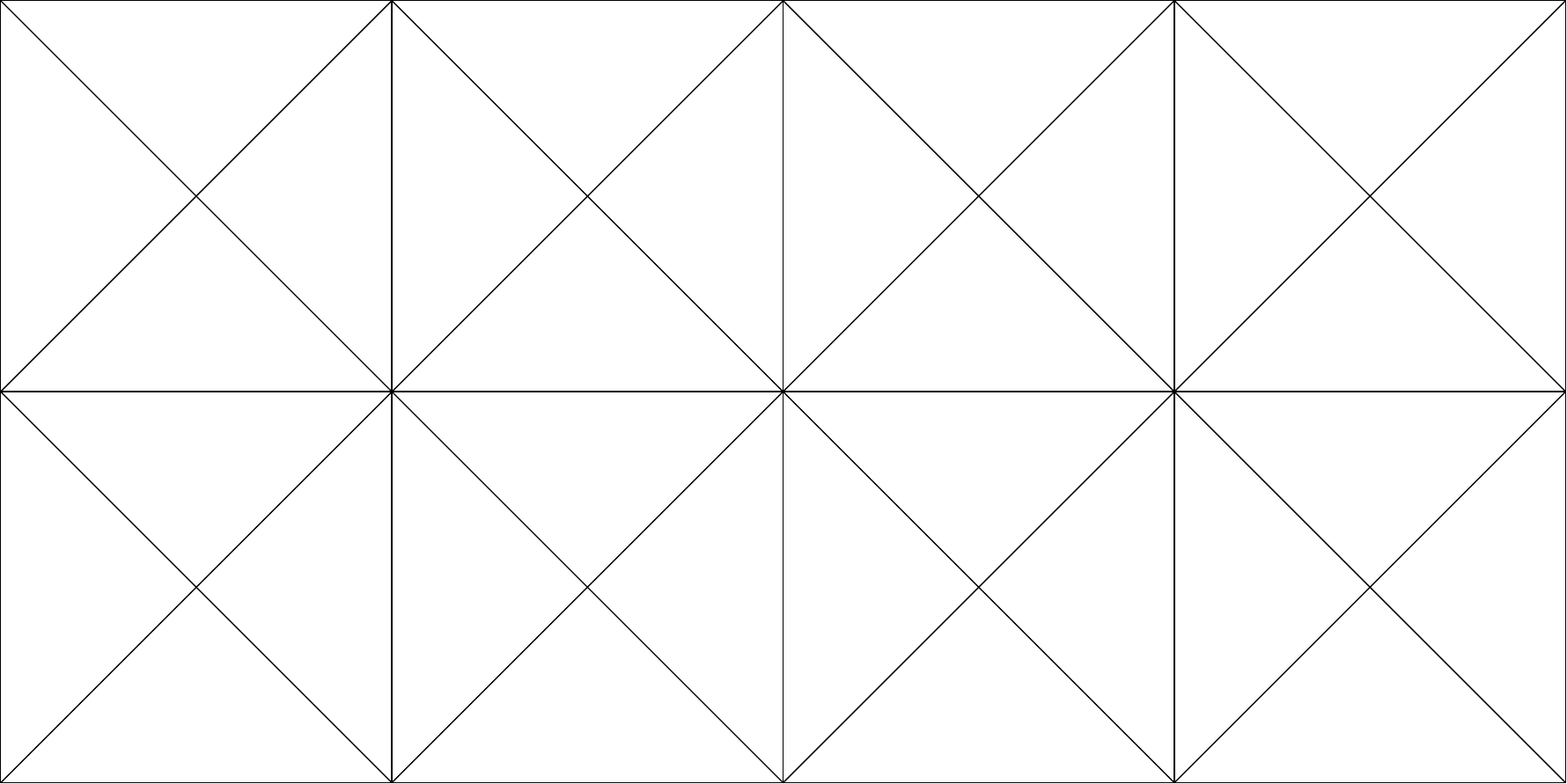}}%
\hspace{8pt}%
\subfloat[][Degrees of freedom for the 0-form.]{%
	\label{fig:zero_form}%
	\includegraphics[width=0.48\columnwidth]{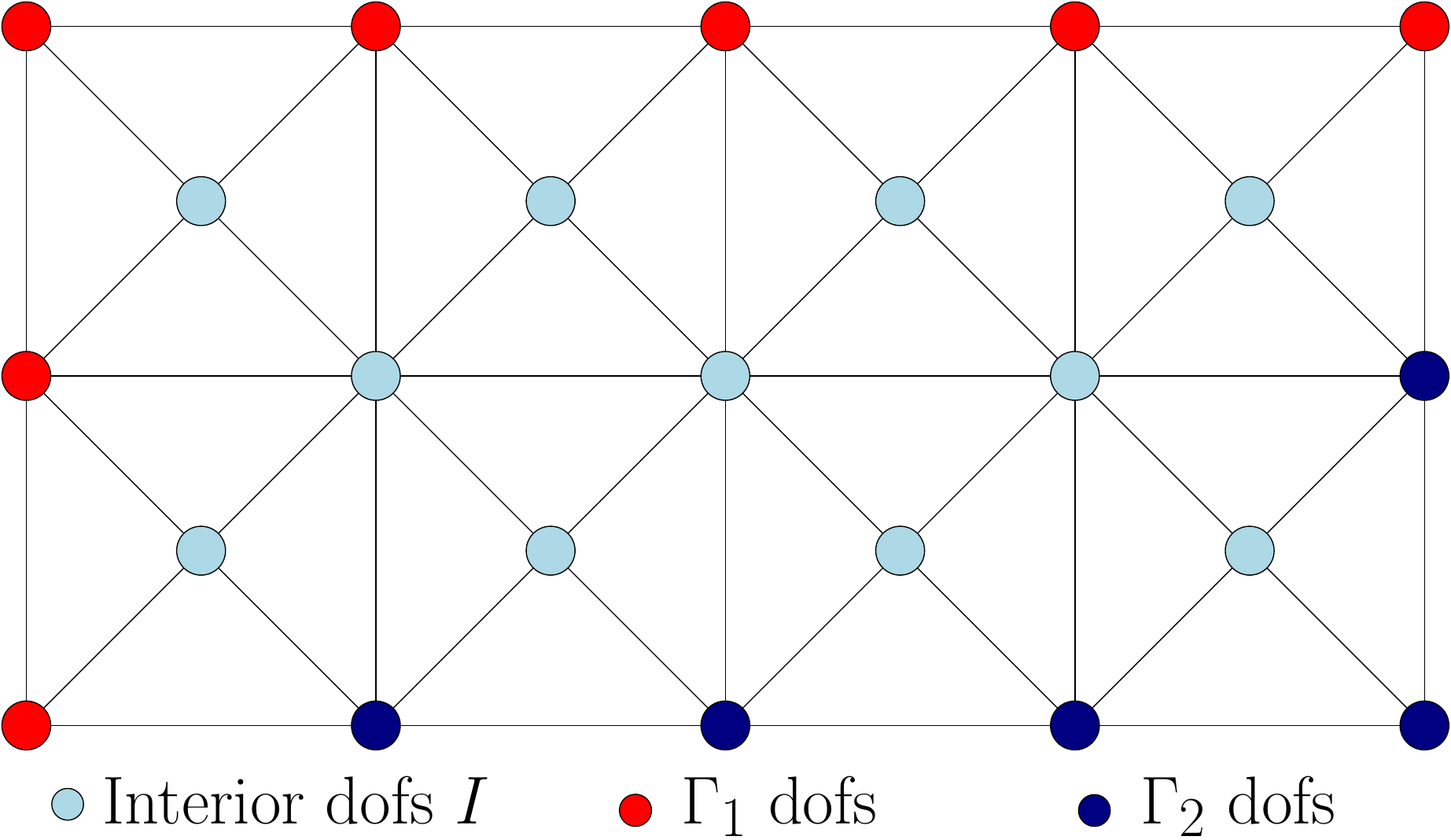}}
	\hspace{8pt}%
\subfloat[][Degrees of freedom for the 1-form.]{%
	\label{fig:one_form}%
	\includegraphics[width=0.46\columnwidth]{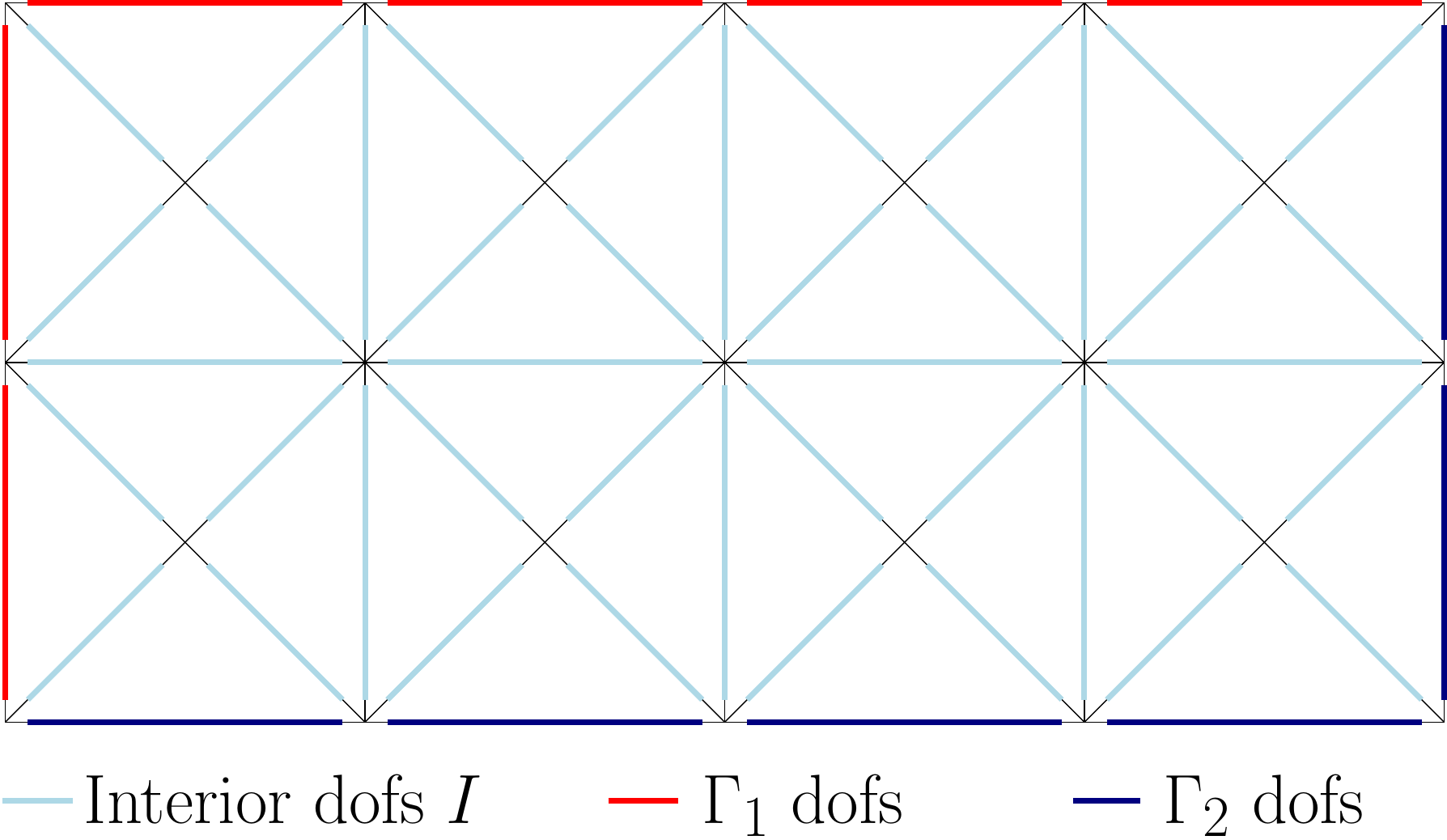}}
\caption{Example of a 2D domain with splitting of the boundary due to the boundary conditions. In the case $p=2, \; q=1$ the duality is between a 0-form  and a 1-form. Notice that the points at the intersection of $\Gamma_1$ and $\Gamma_2$ are attributed to $\Gamma_1$ as those will correspond to the essential boundary conditions for the 0-form.}%
\label{fig:notation_dofs}%
\end{figure}

\subsubsection{Discrete primal port-Hamiltonian system}
Starting from the weak formulation \eqref{eq:weak_primalPH}, the discrete primal port-Hamiltonian system is obtained: $\dual{e}^p_{1, h} \in \dual{\mathcal{V}}^{p}_{s, h}, \; \dual{e}^{p-1}_{2, h} \in \dual{\mathcal{V}}^{p-1}_{s, h}$ such that $(-1)^p \tr \dual{e}^{p-1}_{2, h} \vert_{\Gamma_2} = \dual{u}^{p-1}_{2, h}$ and
 \begin{equation}\label{eq:discr_weak_primalPH}
    \begin{aligned}
    \inpr[M]{\dual{v}^p_h}{\dual{C}^p \partial_t \dual{e}^p_{1, h}} &= (-1)^{p}\inpr[M]{\dual{v}^p_h}{\d \dual{e}^{p-1}_{2, h}}, \\
     \inpr[M]{\dual{v}^{p-1}_h}{\dual{E}^{p-1} \partial_t \dual{e}^{p-1}_{2, h}} &= (-1)^{p} \{- \inpr[M]{\d\dual{v}^{p-1}_h}{\dual{e}^p_{1, h}} + \dualpr[\Gamma_1]{\dual{v}^{p-1}_h}{u^{q-1}_{1, h}}\},
    \end{aligned}  \qquad 
    \begin{aligned}
    &\forall \dual{v}^p \in \dual{\mathcal{V}}^p_{s, h}, \\
    &\forall \dual{v}^{p-1} \in \dual{\mathcal{V}}^{p-1}_{s, h}(\Gamma_2), \\
    \end{aligned}
 \end{equation}
where $u_{1, h}^{q-1} \in \tr {\mathcal{V}}^{q-1}_{s, h}|_{\Gamma_1}$ and $\dual{\mathcal{V}}^{p-1}_{s, h}(\Gamma_2)$ is a polynomial space incorporating boundary conditions
\begin{equation}
        \dual{\mathcal{V}}^{p-1}_{s, h}(\Gamma_2) := \{\dual{w}^{p-1}_h \in \dual{\mathcal{V}}^{p-1}_{s, h} \vert \; \tr \dual{w}^{p-1}_h\vert_{\Gamma_2} = 0 \}. 
\end{equation}
The weak formulation \eqref{eq:discr_weak_primalPH} leads to the following discrete system
\begin{equation}\label{eq:alg_primalPH}
\begin{aligned}
    \begin{bmatrix}
        \mathbf{M}^p_{\dual{C}, s} & \mathbf{0} \\
        \mathbf{0} & [\mathbf{M}^{p-1}_{\dual{E}, s}]_{I \cup \Gamma_1}
    \end{bmatrix}
    \begin{pmatrix}
    \dot{\dual{\mathbf{e}}}^p_1 \\
    \dot{\dual{\mathbf{e}}}^{p-1}_2 \\
    \end{pmatrix} &=  (-1)^p
    \begin{bmatrix}
        \mathbf{0} & \mathbf{D}^{p-1}_s \\
        -[(\mathbf{D}_{s}^{p-1})^\top]_{I \cup \Gamma_1} & \mathbf{0}
    \end{bmatrix}
    \begin{pmatrix}
    {\dual{\mathbf{e}}}^p_1 \\
    {\dual{\mathbf{e}}}^{p-1}_2 \\
    \end{pmatrix} +
    \begin{bmatrix}
        \mathbf{0}\\
        (-1)^{p}[\mathbf{B}^{q-1}_{s}]_{I \cup \Gamma_1}^{\Gamma_1}
    \end{bmatrix}
    \mathbf{u}^{q-1}_1, \\
    (-1)^p[{\dual{\mathbf{e}}}^{p-1}_2]_{\Gamma_2} &= 
    \begin{bmatrix}
    \mathbf{0} & (-1)^p[\mathbf{T}^{p-1}_{s}]_{\Gamma_2} \\
    \end{bmatrix}
    \begin{pmatrix}
    {\dual{\mathbf{e}}}^p_1 \\
    {\dual{\mathbf{e}}}^{p-1}_2 \\
    \end{pmatrix} = \dual{\mathbf{u}}^{p-1}_2,
\end{aligned}
\end{equation}
where $[\mathbf{M}_{\dual{C}, s}^p]_{i}^j := \inpr[M]{\dual{\phi}^p_{s, i}}{\dual{C}^p \dual{\phi}^p_{s, j}}$ and $[\mathbf{M}_{\dual{E}, s}^{p-1}]_{i}^j := \inpr[M]{\dual{\phi}^{p-1}_{s, i}}{\dual{E}^{p-1}\dual{\phi}^{p-1}_{s, j}}$.
This system possesses an associated discrete Hamiltonian
\begin{equation}\label{eq:discr_H_primal}
\dual{H}_h^{p} := \frac{1}{2}\inpr[M]{\dual{e}_{1, h}^p}{\dual{C}^p \dual{e}_{1, h}^p} + \frac{1}{2}\inpr[M]{\dual{e}_{2, h}^{p-1}}{\dual{E}^{p-1} \dual{e}_{2, h}^{p-1}} = \frac{1}{2} (\dual{\mathbf{e}}^p_1)^\top \mathbf{M}^p_{\dual{C}, s} \dual{\mathbf{e}}^p_1 + \frac{1}{2} (\dual{\mathbf{e}}^{p-1}_2)^\top \mathbf{M}^{p-1}_{\dual{E}, s} \dual{\mathbf{e}}^{p-1}_2.
\end{equation}

\begin{proposition}\label{pr:discr_dotH_primal}
The energy rate of the Hamiltonian in Eq. \eqref{eq:discr_H_primal} is given by  
\begin{equation}\label{eq:discr_dotH_primal}
    \diff{\dual{H}_h^{p}}{t} = \dual{P}_h^p:= (\dual{\mathbf{u}}^{p-1}_2)^\top \widetilde{\mathbf{y}}^{p-1} + (\dual{\mathbf{y}}^{p-1}_2)^\top[\mathbf{\Psi}^{q-1}_{s, \partial}]_{\Gamma_1}^{\Gamma_1} \mathbf{u}^{q-1}_1, 
\end{equation} 
where $\dual{\mathbf{y}}^{p-1}_2 = (-1)^p [\dual{\mathbf{e}}^{p-1}_2]_{\Gamma_1}$ is defined in Eq. \eqref{eq:dofs_y2} and
\begin{equation} \label{eq:alg_til_y_primal} 
    \widetilde{\mathbf{y}}^{p-1} :=  (-1)^p[\mathbf{M}^{p-1}_{\dual{E}, s}]_{\Gamma_2} \dot{\dual{\mathbf{e}}}^{p-1}_2  +[(\mathbf{D}_{s}^{p-1})^\top]_{\Gamma_2} \dual{\mathbf{e}}^p_1,  
\end{equation}
\end{proposition}

\begin{proof}
From the dynamics \eqref{eq:alg_primalPH} it is obtained
\begin{equation*}
\begin{aligned}
    \diff{\dual{H}_h^p}{t} &= (\dual{\mathbf{e}}^p_1)^\top \mathbf{M}^p_{\dual{C}, s} \dot{\dual{\mathbf{e}}}^p_1 + (\dual{\mathbf{e}}^{p-1}_2)^\top \mathbf{M}^{p-1}_{\dual{E}, s} \dot{\dual{\mathbf{e}}}^{p-1}_2, \\
    &= (\dual{\mathbf{e}}^p_1)^\top \mathbf{M}^p_{\dual{C}, s} \dot{\dual{\mathbf{e}}}^p_1 + ([\dual{\mathbf{e}}^{p-1}_2]_{I\cup \Gamma_1})^\top [\mathbf{M}^{p-1}_{\dual{E}, s}]_{I\cup \Gamma_1} \dot{\dual{\mathbf{e}}}^{p-1}_2 + ([\dual{\mathbf{e}}^{p-1}_2]_{\Gamma_2})^\top [\mathbf{M}^{p-1}_{\dual{E}, s}]_{\Gamma_2} \dot{\dual{\mathbf{e}}}^{p-1}_2 , \\
    &= (-1)^p ([\dual{\mathbf{e}}^{p-1}_2]_{\Gamma_2})^\top [(\mathbf{D}_{s}^{p-1})^\top]_{\Gamma_2} \dual{\mathbf{e}}^p_1  + (-1)^{p}([\dual{\mathbf{e}}^{p-1}_2]_{I\cup \Gamma_1})^\top[\mathbf{B}^{q-1}_{s}]_{I \cup \Gamma_1}^{\Gamma_1} \mathbf{u}^{q-1}_1 \\
    &+ ([\dual{\mathbf{e}}^{p-1}_2]_{\Gamma_2})^\top [\mathbf{M}^{p-1}_{\dual{E}, s}]_{\Gamma_2} \dot{\dual{\mathbf{e}}}^{p-1}_2, \\
    &= (\dual{\mathbf{u}}^{p-1}_2)^\top \{(-1)^p[\mathbf{M}^{p-1}_{\dual{E}, s}]_{\Gamma_2} \dot{\dual{\mathbf{e}}}^{p-1}_2 +[(\mathbf{D}_{s}^{p-1})^\top]_{\Gamma_2} \dual{\mathbf{e}}^p_1\} + (\dual{\mathbf{y}}^{p-1}_2)^\top[\mathbf{\Psi}^{q-1}_{s, \partial}]_{\Gamma_1}^{\Gamma_1} \mathbf{u}^{q-1}_1.
\end{aligned}
\end{equation*}
Notice that the matrix $[\mathbf{\Psi}^{q-1}_{s}]_{\Gamma_1}^{\Gamma_1}$ is rectangular. Given definitions \eqref{eq:alg_til_y_primal}, \eqref{eq:dofs_y2}, the statement is proven. 
\end{proof}

\revOne{
\begin{remark}[The collocated nature of outputs $\widetilde{\mathbf{y}}^{p-1}$ in a pHDAE setting]\label{rmk:til_y_PHDAE}
The collocated output in Eqs. \eqref{eq:alg_til_y_primal} is a consequence of the strong imposition of the boundary conditions by direct assignment and represents the reaction to be applied at the boundary to guarantee to that the state follows the input trajectory. To highlight their collocated nature it is necessary to consider a differential-algebraic framework and in particular port-Hamiltonian descriptor systems \cite{beattie2018pHDAE}. System \eqref{eq:alg_primalPH} can be equivalently rewritten in the following canonical pHDAE form \cite{beattie2018pHDAE}
\begin{equation}
\begin{aligned}
    \mathbf{E}\dot{\mathbf{x}} &= \mathbf{J}\mathbf{x} + \mathbf{B}\mathbf{u}, \\
    \mathbf{y} &= \mathbf{B}^\top\mathbf{x},
\end{aligned}
\end{equation}
where the state, input and output variables read
\begin{equation}
    \mathbf{x} = \begin{pmatrix}
    {\dual{\mathbf{e}}}^p_1 \\
    [\dual{\mathbf{e}}^{p-1}_2]_{I \cup \Gamma_1} \\
    [\dual{\mathbf{e}}^{p-1}_2]_{\Gamma_2} \\
    \dual{\bm{\lambda}}^{p-1}
    \end{pmatrix}, \qquad 
    \mathbf{u} = \begin{pmatrix}
        \mathbf{u}^{q-1}_1 \\
        \dual{\mathbf{u}}^{p-1}_2
    \end{pmatrix}, \qquad
    \mathbf{y} = \begin{pmatrix}
        \widehat{\mathbf{y}}^{q-1}_2 \\
        \widetilde{\mathbf{y}}^{p-1} \\
    \end{pmatrix}.
\end{equation}
The state now includes a Lagrange multiplier $\dual{\bm{\lambda}}^{p-1}$ to enforce the constraint. The $\mathbf{E}, \; \mathbf{J}, \; \mathbf{B}$ matrices are given by
\begin{equation}
\begin{aligned}
    \mathbf{E} &= \begin{bmatrix}
        \mathbf{M}^p_{\dual{C}, s} & \mathbf{0} & \mathbf{0} & \mathbf{0} \\
        \mathbf{0} & [\mathbf{M}^{p-1}_{\dual{E}, s}]_{I \cup \Gamma_1}^{I \cup \Gamma_1} & [\mathbf{M}^{p-1}_{\dual{E}, s}]_{I \cup \Gamma_1}^{\Gamma_2} & \mathbf{0} \\
        \mathbf{0} & [\mathbf{M}^{p-1}_{\dual{E}, s}]_{\Gamma_2}^{I \cup \Gamma_1} & [\mathbf{M}^{p-1}_{\dual{E}, s}]_{\Gamma_2}^{\Gamma_2} & \mathbf{0} \\
        \mathbf{0} & \mathbf{0} & \mathbf{0} & \mathbf{0} \\
    \end{bmatrix}, \\
    \mathbf{J} &= (-1)^p
    \begin{bmatrix}
        \mathbf{0} & -[\mathbf{D}^{p-1}_s]^{I \cup \Gamma_1} & -[\mathbf{D}^{p-1}_s]^{\Gamma_2} & \mathbf{0} \\
        [(\mathbf{D}_{s}^{p-1})^\top]_{I \cup \Gamma_1} & \mathbf{0} & \mathbf{0} & \mathbf{0} \\
        [(\mathbf{D}_{s}^{p-1})^\top]_{\Gamma_2} & \mathbf{0} & \mathbf{0} & \mathbf{I} \\
        \mathbf{0} & \mathbf{0} & -\mathbf{I} & \mathbf{0}
    \end{bmatrix}, \qquad
    \mathbf{B} = 
    \begin{bmatrix}
        \mathbf{0} & \mathbf{0}\\
        (-1)^{p}[\mathbf{B}^{q-1}_{s}]_{I \cup \Gamma_1}^{\Gamma_1} & \mathbf{0} \\
        \mathbf{0} & \mathbf{0}\\
        \mathbf{0} & \mathbf{I}\\
    \end{bmatrix}.
\end{aligned}
\end{equation}

The variation of the energy is given by 
\begin{equation}\label{eq:Hprimal_mult}
    \diff{\dual{H}^{p}_h}{t} = (\dual{\mathbf{u}}^{p-1}_2)^\top \widetilde{\mathbf{y}}^{p-1} + (\widehat{\mathbf{y}}^{q-1}_2)^\top \mathbf{u}^{q-1}_1.
\end{equation}
Notice that the output $\widehat{\mathbf{y}}^{q-1}_2$  corresponds to
\begin{equation*}
\begin{aligned}
    \widehat{\mathbf{y}}^{q-1}_2 &= [(\mathbf{B}^{q-1}_{s})^\top]^{I \cup \Gamma_1}_{\Gamma_1} [\dual{\mathbf{e}}^{p-1}_2]_{I \cup \Gamma_1}, \\
    &=  [(\mathbf{\Psi}^{q-1}_{s, \partial})^\top \mathbf{T}_s^{p-1}]^{I \cup \Gamma_1}_{\Gamma_1} [\dual{\mathbf{e}}^{p-1}_2]_{I \cup \Gamma_1}, \\
    &=  ([\mathbf{\Psi}^{q-1}_{s, \partial}]^{\Gamma_1}_{\Gamma_1})^\top \dual{\mathbf{y}}^{p-1}_2.
\end{aligned}
\end{equation*}
Thus Eq. \eqref{eq:Hprimal_mult} coincide with \eqref{eq:discr_dotH_primal}. Analogous examples of this construction can be found in \cite[Remark 3.6]{altmann2021poro} and \cite[Section 5.4]{mehrmann2022control}.
\end{remark}
}
\subsubsection{Discrete dual port-Hamiltonian system}
Starting from the weak formulation \eqref{eq:weak_dualPH}, the discrete dual port-Hamiltonian system is obtained: ${e}^{q-1}_{1, h} \in {\mathcal{V}}^{q-1}_{s, h}, \; {e}^{q}_{2, h} \in {\mathcal{V}}^{q}_{s, h}$ such that $\tr {e}^{q-1}_{1, h} \vert_{\Gamma_1} = {u}^{q-1}_{1, h}$ and
\begin{equation}\label{eq:discr_weak_dualPH}
	\begin{aligned}
    	\inpr[M]{\dual{v}^{q-1}_h}{{C}^{q-1} \partial_t e^{q-1}_{1, h}} &= \inpr[M]{\d\dual{v}^{q-1}_h}{{e}^q_{2, h}} + (-1)^{(p-1)(q-1)} \dualpr[\Gamma_2]{{v}^{q-1}_h}{\dual{u}_{2, h}^{p-1}}, \\
		\inpr[M]{v^{q}_h}{E^{q} \partial_t {e}^{q}_{2, h}} &= -\inpr[M]{v^q_h}{\d e^{q-1}_{1, h}}, \\
	\end{aligned} \qquad 
	\begin{aligned}
		&\forall v^{q-1}_h \in \mathcal{V}^{q-1}_{s, h}(\Gamma_1), \\
		&\forall {v}^q_h \in \mathcal{V}^{q}_{s, h}, \\
	\end{aligned}
\end{equation}
where $\dual{u}_{2, h}^{p-1} \in \tr \dual{\mathcal{V}}_{s, h}^{p-1}|_{\Gamma_2}$ and $\mathcal{V}_{s, h}^{q-1}(\Gamma_1)$ is a polynomial space with boundary conditions
\begin{equation}
    \mathcal{V}^{q-1}_{s, h}(\Gamma_1) := \{w^{q-1}_h \in \mathcal{V}^{q-1}_{s, h}| \tr w^{q-1}_h |_{\Gamma_1} = 0 \}.
\end{equation}
The algebraic realization of the discrete weak formulation \eqref{eq:discr_weak_dualPH} reads
\begin{equation}\label{eq:alg_dualPH}
\begin{aligned}
    \begin{bmatrix}
        [\mathbf{M}^{q-1}_{{C}, s}]_{I \cup \Gamma_2} & \mathbf{0} \\
        \mathbf{0} & \mathbf{M}^q_{E, s}
    \end{bmatrix}
    \begin{pmatrix}
    \dot{\mathbf{e}}_1^{q-1} \\
    \dot{{\mathbf{e}}}_2^q
    \end{pmatrix} &= 
    \begin{bmatrix}
        \mathbf{0} & [(\mathbf{D}_{s}^{q-1})^\top]_{I \cup \Gamma_2}\\
        -\mathbf{D}^{q-1}_s & \mathbf{0}
    \end{bmatrix}
    \begin{pmatrix}
    \mathbf{e}_1^{q-1} \\
    {\mathbf{e}}_2^q
    \end{pmatrix} +
    \begin{bmatrix}
        (-1)^{(p-1)(q-1)}[\mathbf{B}_{s}^{p-1}]_{I \cup \Gamma_2}^{\Gamma_2} \\
        \mathbf{0}
    \end{bmatrix}\dual{\mathbf{u}}^{p-1}_2, \\
     [\mathbf{e}^{q-1}]_{\Gamma_1} &= 
    \begin{bmatrix}
    [\mathbf{T}^{q-1}_{s}]_{\Gamma_1} & \mathbf{0} \\
    \end{bmatrix}
    \begin{pmatrix}
    \mathbf{e}^{q-1}_1 \\
    \mathbf{e}^q_2
    \end{pmatrix} = \mathbf{u}^{q-1}_1,
\end{aligned}
\end{equation}
where $[\mathbf{M}_{{C}, s}^{q-1}]_{i}^j := \inpr[M]{{\phi}^{q-1}_{s, i}}{{C}^{q-1} {\phi}^{q-1}_{s, j}}$ and $[\mathbf{M}_{E, s}^{q}]_{i}^j := \inpr[M]{\phi^{q}_{s, i}}{E^{q} \phi^{q}_{s, j}}$. The associated discrete Hamiltonian for this system reads
\begin{equation}\label{eq:discr_H_dual}
    H_h^{q} := \frac{1}{2}\inpr[M]{{e}_{1, h}^{q-1}}{{C}^{q-1} {e}_{1, h}^p} + \frac{1}{2}\inpr[M]{\dual{e}_{2, h}^q}{{E}^q \dual{e}_{2, h}^q} = \frac{1}{2} ({\mathbf{e}}^{q-1}_1)^\top \mathbf{M}^{q-1}_{C, s} {\mathbf{e}}^{q-1}_1 + \frac{1}{2} ({\mathbf{e}}^q_2)^\top \mathbf{M}^q_{E, s} \mathbf{e}^q_2.  
\end{equation}

\begin{proposition}\label{pr:discr_dotH_dual}
The energy rate of the Hamiltonian in Eq. \eqref{eq:discr_H_dual} (arising from the mixed discretization) is given by  
\begin{equation}\label{eq:discr_dotH_dual}
    \diff{{H}_h^{q}}{t} = P_h^q := (\mathbf{u}^{q-1}_1)^\top \widetilde{\mathbf{y}}^{q-1} + (\dual{\mathbf{u}}^{p-1}_2)^\top[\mathbf{\Psi}^{q-1}_{s, \partial}]_{\Gamma_2}^{\Gamma_2} \mathbf{y}^{q-1}_1,
\end{equation} 
where ${\mathbf{y}}^{q-1}_1 = [\mathbf{e}^{q-1}_1]_{\Gamma_2}$ is defined in Eq. and \eqref{eq:dofs_y1}
\begin{equation}
    \widetilde{\mathbf{y}}^{q-1} := [\mathbf{M}^{q-1}_{{C}, s}]_{\Gamma_1} \dot{\mathbf{e}}^{q-1}_1 -[(\mathbf{D}_{s}^{q-1})^\top]_{\Gamma_1} {\mathbf{e}}^q_2, \label{eq:alg_til_y_dual} 
\end{equation}
\end{proposition}

\begin{proof}
From the dynamics \eqref{eq:alg_dualPH}, it is obtained
\begin{equation*}
\begin{aligned}
    \diff{H_h^q}{t} &= (\mathbf{e}^{q-1}_1)^\top \mathbf{M}^{q-1}_{{C}, s} \dot{\mathbf{e}}^{q-1}_1 + ({\mathbf{e}}^q_2)^\top \mathbf{M}^q_{E, s} \dot{\mathbf{e}}^q_2, \\  
    &=([{\mathbf{e}}^{q-1}_1]_{I\cup \Gamma_2})^\top [\mathbf{M}^{q-1}_{C, s}]_{I\cup \Gamma_2} \dot{\mathbf{e}}^{q-1}_1 + ([{\mathbf{e}}^{q-1}_1]_{\Gamma_1})^\top [\mathbf{M}^{q-1}_{C, s}]_{\Gamma_1} \dot{\mathbf{e}}^{q-1}_1 + ({\mathbf{e}}^q_2)^\top \mathbf{M}^q_{E, s} \dot{\mathbf{e}}^q_2,\\
    &=- ([{\mathbf{e}}^{q-1}_1]_{\Gamma_1})^\top[(\mathbf{D}_{s}^{q-1})^\top]_{\Gamma_1}{\mathbf{e}}^q_2  + (-1)^{(p-1)(q-1)} ({\mathbf{e}}^{q-1}_1)^\top_{I \cup \Gamma_2}[\mathbf{B}_{s}^{p-1}]_{I \cup \Gamma_2}^{\Gamma_2} \dual{\mathbf{u}}^{p-1}_2 \\
    &+  ([{\mathbf{e}}^{q-1}_1]_{\Gamma_1})^\top [\mathbf{M}^{q-1}_{C, s}]_{\Gamma_1} \dot{\mathbf{e}}^{q-1}_1,\\
    &=(\mathbf{u}^{q-1}_1)^\top \{[\mathbf{M}^{q-1}_{C, s}]_{\Gamma_1} \dot{\mathbf{e}}^{q-1}_1 -[(\mathbf{D}_{s}^{q-1})^\top]_{\Gamma_1} {\mathbf{e}}^q_2\} + (\dual{\mathbf{u}}^{p-1}_2)^\top[\mathbf{\Psi}^{q-1}_{s, \partial}]_{\Gamma_2}^{\Gamma_2} \mathbf{y}^{q-1}_1.
\end{aligned}
\end{equation*}
Given definitions \eqref{eq:alg_til_y_dual}, \eqref{eq:dofs_y1}, the statement is proven. 
\end{proof}

\revOne{As in Remark \ref{rmk:til_y_PHDAE}, the strong imposition in System \eqref{eq:discr_weak_dualPH} of the boundary condition leads to a descriptor system with power balance 
\begin{equation}\label{eq:Hdual_mult}
    \diff{H^q_h}{t} = (\mathbf{u}^{q-1}_1)^\top \widetilde{\mathbf{y}}^{q-1} + (\dual{\mathbf{u}}_2^{p-1})^\top {\mathbf{y}}^{p-1}_1,
\end{equation}
where $\widetilde{\mathbf{y}}^q$ corresponds to
\begin{equation*}
\begin{aligned}
    \mathbf{y}^{p-1}_1 =  [\mathbf{\Psi}^{q-1}_{s, \partial}]^{\Gamma_2}_{\Gamma_2} \mathbf{y}^{q-1}_1.
\end{aligned}
\end{equation*}}
}

\subsubsection{Recovering the power balance}
\revTwo{For the case of uniform boundary conditions, it is sufficient to consider only one of the two mixed discretization Eqs. \eqref{eq:alg_primalPH} and \eqref{eq:alg_dualPH} to obtain the preservation of the discrete power balance (as in the Partitioned Finite Element method \cite{cardoso2020pfem}). For example in the case in which $\Gamma_1 =\partial M, \; \Gamma_2 = \emptyset$, then the energy rate reads
\begin{equation}
    \diff{\dual{H}_h^p}{t} = (\dual{\mathbf{y}}^{p-1}_2)^\top\mathbf{\Psi}^{q-1}_{s, \partial} \mathbf{u}^{q-1}_1.
\end{equation}
Conversely, if $\Gamma_1 =\emptyset, \; \Gamma_2 = \partial M$, then the energy rate reads
\begin{equation}
    \diff{{H}_h^q}{t} = (\dual{\mathbf{u}}^{p-1}_2)^\top \mathbf{\Psi}^{q-1}_{s, \partial} \mathbf{y}^{q-1}_1.
\end{equation}
However, in this general case of mixed boundary conditions, both systems \eqref{eq:alg_primalPH} and \eqref{eq:alg_dualPH} are needed in order to recover a discrete Stokes-Dirac structure and its associated power balance.
}
\revOne{The right hand side of the continuous power balance} \eqref{eq:H_dot} written in terms of efforts can be expressed as
\begin{equation}
    P = \dualpr[M]{e^{q-1}_1}{\dual{C}^p\partial_t \dual{e}^p_1} + \dualpr[M]{\dual{e}^{p-1}_2}{E^q\partial_t e^q_2}.
\end{equation}
Its discrete version is then given by 
\begin{equation}
    P_h = \dualpr[M]{e^{q-1}_{1, h}}{\dual{C}^p\partial_t \dual{e}^p_{1, h}} + \dualpr[M]{\dual{e}^{p-1}_{2, h}}{E^q\partial_t e^q_{2, h}} = \mathbf{e}^{q-1}_1 \mathbf{L}_{\dual{C},s}^p \dot{\dual{\mathbf{e}}}^p_1 + \dual{\mathbf{e}}^{p-1}_2 \mathbf{L}_{E,s}^q \dot{{\mathbf{e}}}^q_2.
\end{equation}
where $[\mathbf{L}_{\dual{C}, s}^p]_{i}^j := \dualpr[M]{\phi^{q-1}_{s, i}}{\dual{C}^p \dual{\phi}^p_{s, j}}$ and $[\mathbf{L}_{\dual{E}, s}^p]_{i}^j := \dualpr[M]{\phi^{q-1}_{s, i}}{\dual{C}^p \dual{\phi}^p_{s, j}}$. The dual field formulation is such to preserve the power balance at the discrete level, under some regularity assumption for the physical coefficients. This is due to the fact that each mixed system \eqref{eq:discr_weak_primalPH} and \eqref{eq:discr_weak_dualPH} contains an equation (the one belonging to the Stokes-Dirac structure) that is strongly verified, leading to a discrete conservation law.

\begin{proposition}\label{pr:discr_Hdot}
If the tensors $\dual{C}^p, \; E^q$ verify the following regularity assumption
\begin{equation}\label{eq:reg_C}
    \dual{C}^p: H\dual{\Omega}^p(M) \rightarrow H\dual{\Omega}^p(M), \qquad E^q: H\Omega^q(M) \rightarrow H\Omega^q(M),
\end{equation}
then the discrete power satisfies
\begin{equation}
    P_h = \dualpr[\partial M]{\dual{e}_{\partial, h}^{p-1}}{f_{\partial, h}^{q-1}} = \dual{\mathbf{e}}^{p-1}_\partial \mathbf{\Psi}^{q-1}_{s, \partial} \mathbf{f}^{q-1}_\partial.
\end{equation}
\end{proposition}

\begin{proof}
Since the tensors satisfy the regularity assumption, the subcomplex property \cite[Lemma 3.8]{arnold2006acta} to the first line of \eqref{eq:discr_weak_primalPH} and the second line of \eqref{eq:discr_weak_dualPH} provides
\begin{equation*}
\dual{C}^p\partial_t\dual{e}^p_{1, h} = -(-1)^r\d \dual{e}^{p-1}_{2, h}, \qquad E^q\partial_t{e}^q_{2, h} = -\d e^{q-1}_{1, h}.
\end{equation*}
Taking the duality product against $e^{q-1}_{1, h}$ and $\dual{e}_{2, h}^{p-1}$ provides
\begin{equation*}
P_h= -\dualpr[M]{e^{q-1}_{1, h}}{(-1)^r\d \dual{e}^{p-1}_{2, h}}-\dualpr[M]{\dual{e}_{2, h}^{p-1}}{\d e^{q-1}_{1, h}}.
\end{equation*}
The application of the discrete Stokes Theorem \eqref{eq:int_byparts_d_disc} then gives
\begin{equation*}
P_h = \dualpr[\partial M]{(-1)^p \dual{e}^{p-1}_{2, h}}{e^{q-1}_{1, h}}= \dualpr[\partial M]{e_{\partial, h}^{p-1}}{f_{\partial, h}^{q-1}},
\end{equation*}
leading to the proof of the statement.
\end{proof}

\section{Time discretization}\label{sec:time_discr}

To ensure discrete conservation of energy, implicit Runge-Kutta methods based on Gauss-Legendre collocation points can be used \cite{sanzserna1992}. These methods are also the only collocation schemes that lead to an exact discrete energy balance in the linear case \cite{kotyczka2019discrete}. The implicit midpoint method is here used to illustrate the time discretization. \\

Consider a total simulation time $T_{\mathrm{end}}$ and a equidistant splitting given by the time step $\Delta t = T_{\mathrm{end}}/N_t$, where $N_t$ is the total number of simulation instants. The evaluation of a generic variable $\mathbf{x}$ at the time instant $t_n = n\Delta t$ is denoted by $\mathbf{x}_n$. The midpoint method is the simplest method in the class of collocation methods known as Gauss-Legendre methods. It applies to systems of the form
\begin{equation}
    \dot{\mathbf{x}} = \mathbf{f}(\mathbf{x}, t), \qquad \mathbf{x} \in \bbR^d, \qquad \mathbf{f}: \bbR^{d} \times [0, T_{\mathrm{end}}] \rightarrow \bbR^d.
\end{equation}
This method is a one stage method, i.e. only uses the information at the previous time, and seeks for the solution of the implicit equation
\begin{equation}
    \mathbf{x}_{n+1} = \mathbf{x}_{n} + \Delta t \; \mathbf{f}\left(\mathbf{x}_{n+1/2},  t_{n+1/2}\right),
\end{equation}
where $\mathbf{x}_{n+1/2}:= \frac{\mathbf{x}_{n+1} + \mathbf{x}_{n}}{2}$ and $t_{n+1/2}:= t_n + \frac{\Delta t}{2}$. The application of the implicit midpoint method to system \eqref{eq:alg_primalPH} with direct assignment of the boundary degrees of freedom leads to the algebraic system
\begin{equation}\label{eq:timediscr_primal}
    \dual{\mathbf{A}}^{p}_{\mathrm{bc}}
    \begin{pmatrix}
    \dual{\mathbf{e}}^p_{1, n+1} \\ [\dual{\mathbf{e}}^{p-1}_{2, n+1}]_{I\cup \Gamma_1}
    \end{pmatrix}
    = \dual{\mathbf{b}}^{p}_{\mathrm{bc}}.
\end{equation}
The matrix $\dual{\mathbf{A}}^{p}_{\mathrm{bc}}$ corresponds to 
\begin{equation}
\dual{\mathbf{A}}^{p}_{\mathrm{bc}} =
    \begin{bmatrix}
        \mathbf{M}^p_{\dual{C}, s} & -\frac{1}{2}(-1)^p\Delta t [\mathbf{D}^{p-1}_s]^{I\cup \Gamma_1} \\
        \frac{1}{2}(-1)^p\Delta t [(\mathbf{D}_{s}^{p-1})^\top]_{I\cup \Gamma_1} & [\mathbf{M}^{p-1}_{\dual{E}, s}]_{I\cup \Gamma_1}^{I\cup \Gamma_1}
    \end{bmatrix}.
\end{equation}
The $\dual{\mathbf{b}}^{p}_{\mathrm{bc}}$ vector incorporates the forcing due to the previous time step and the boundary conditions
\begin{equation}
\begin{aligned}
\dual{\mathbf{b}}^{p}_{\mathrm{bc}} = \begin{bmatrix}
        \mathbf{M}^p_{\dual{C}, s} & \frac{1}{2}(-1)^p\Delta t[\mathbf{D}^{p-1}_s]^{I\cup \Gamma_1} \\
        -\frac{1}{2}(-1)^p\Delta t [(\mathbf{D}_{s}^{p-1})^\top]_{I\cup \Gamma_1} & [\mathbf{M}^{p-1}_{\dual{E}, s}]_{I\cup \Gamma_1}^{I\cup \Gamma_1}
    \end{bmatrix}\begin{pmatrix}
    \dual{\mathbf{e}}^p_{1, n} \\
    [\dual{\mathbf{e}}^{p-1}_{2, n}]_{I\cup \Gamma_1}
    \end{pmatrix} \\
    - (-1)^p
    \begin{bmatrix}
         \mathbf{0} \\
        [\mathbf{M}^{p-1}_{\dual{E}, s}]_{I\cup \Gamma_1}^{\Gamma_2}
    \end{bmatrix}(\dual{\mathbf{u}}^{p-1}_{2, n+1} - \dual{\mathbf{u}}^{p-1}_{2, n})
    + \Delta t \begin{bmatrix}
        [\mathbf{D}^{p-1}_s]^{\Gamma_2} & \mathbf{0}\\
        \mathbf{0} & (-1)^p[\mathbf{B}^{q-1}_{s}]_{I \cup \Gamma_1}
    \end{bmatrix}
    \begin{pmatrix}
        \dual{\mathbf{u}}^{p-1}_{2, n+1/2} \\
        \mathbf{u}^{q-1}_{1, n+1/2}
    \end{pmatrix}.
\end{aligned}
\end{equation}

The implicit midpoint method, once applied to \eqref{eq:alg_dualPH}, leads to the system
\begin{equation}\label{eq:timediscr_dual}
    \mathbf{A}^{q}_{\mathrm{bc}}
    \begin{pmatrix}
    [\mathbf{e}^{q-1}_{1, n+1}]_{I\cup \Gamma_2} \\ \dual{\mathbf{e}}^q_{2, n+1}
    \end{pmatrix}
    = \mathbf{b}^{q}_{\mathrm{bc}},
\end{equation}
with
\begin{equation}
\mathbf{A}^{q}_{\mathrm{bc}} =
    \begin{bmatrix}
        [\mathbf{M}^{q-1}_{C, s}]^{I\cup \Gamma_2}_{I\cup \Gamma_2} & -\frac{1}{2}\Delta t [(\mathbf{D}^{q-1}_s)^\top]_{I\cup \Gamma_2} \\
        \frac{1}{2}\Delta t [\mathbf{D}_{s}^{q-1}]^{I\cup \Gamma_2} & \mathbf{M}^{q}_{E, s}
    \end{bmatrix}
\end{equation}
and 
\begin{equation}
\begin{aligned}
\mathbf{b}^{q}_{\mathrm{bc}} = 
\begin{bmatrix}
        [\mathbf{M}^{q-1}_{\dual{C}, s}]^{I\cup \Gamma_2}_{I\cup \Gamma_2} & \frac{1}{2}\Delta t [(\mathbf{D}^{q-1}_s)^\top]_{I\cup \Gamma_2} \\
        -\frac{1}{2}\Delta t [\mathbf{D}_{s}^{q-1}]^{I\cup \Gamma_2} & \mathbf{M}^{q}_{{C}, s}
\end{bmatrix}
    \begin{pmatrix}
    [\mathbf{e}^{q-1}_{1, n}]_{I\cup \Gamma_2} \\ {\mathbf{e}}^q_{2, n}
    \end{pmatrix} \\
    -
    \begin{bmatrix}
        [\mathbf{M}^{q-1}_{\dual{C}, s}]_{I\cup \Gamma_2}^{\Gamma_1} \\
        \mathbf{0} \\
    \end{bmatrix}(\mathbf{u}^{q-1}_{1, n+1} - \mathbf{u}^{q-1}_{1, n}) 
    + \Delta t \begin{bmatrix}
        \mathbf{0} & (-1)^{(p-1)(q-1)}[\mathbf{B}^{p-1}_{s}]_{I \cup \Gamma_2} \\
        -[\mathbf{D}^{q-1}_s]^{\Gamma_1} & \mathbf{0}\\
    \end{bmatrix}
    \begin{pmatrix}
        \mathbf{u}^{q-1}_{1, n+1/2} \\
        \dual{\mathbf{u}}^{p-1}_{2, n+1/2}
    \end{pmatrix}.
\end{aligned}
\end{equation}

\begin{proposition}\label{pr:discrtime_energyrate}
Given the recursions implemented by systems \eqref{eq:timediscr_primal} and \eqref{eq:timediscr_dual}, the following time discrete energy rate holds
\begin{equation}
\begin{aligned}
    \frac{\dual{H}^{p}_{h, n+1} - \dual{H}^{p}_{h, n}}{\Delta t} &= \dual{P}_{h, n+1/2}^{p} := (\dual{\mathbf{u}}^{p-1}_{2, n+1/2})^\top \widetilde{\mathbf{y}}^{p-1}_{n+1/2} +  (\dual{\mathbf{y}}^{p-1}_{2, n+1/2})^\top [\mathbf{\Psi}^{q-1}_{s}]_{\Gamma_1}^{\Gamma_1} \mathbf{u}^{q-1}_{1, n+1/2}, \\
    \frac{H^{q}_{h, n+1} - H^{q}_{h, n}}{\Delta t} 
    &= P_{h, n+1/2}^{q} := (\mathbf{u}^{q-1}_{1, n+1/2})^\top \widetilde{\mathbf{y}}^{q-1}_{n+1/2} +  (\dual{\mathbf{u}}^{p-1}_{2, n+1/2})^\top [\mathbf{\Psi}^{q-1}_{s}]_{\Gamma_2}^{\Gamma_2} \mathbf{y}^p_{n+1/2},
\end{aligned}
\end{equation}
where $\widetilde{\mathbf{y}}^{q-1}_{n+1/2}$ and $\widetilde{\mathbf{y}}^{p-1}_{n+1/2}$ correspond a midpoint discretization of Eqs. \eqref{eq:alg_til_y_primal} and \eqref{eq:alg_til_y_dual} respectively
\begin{equation}
\begin{aligned}
\widetilde{\mathbf{y}}^{p-1}_{n+1/2} &:= (-1)^p [\mathbf{M}^{p-1}_{\dual{E}, s}]_{\Gamma_2} \frac{\dual{\mathbf{e}}^{p-1}_{2, n+1}-\dual{\mathbf{e}}^{p-1}_{2, n}}{\Delta t}  +[(\mathbf{D}_{s}^{p-1})^\top]_{\Gamma_2} \dual{\mathbf{e}}^p_{1, n+1/2}, \\
\widetilde{\mathbf{y}}^{q-1}_{n+1/2} &:= [\mathbf{M}^{q-1}_{{C}, s}]_{\Gamma_1} \frac{\mathbf{e}^{q-1}_{1, n+1} - \mathbf{e}^{q-1}_{1, n}}{\Delta t} -[(\mathbf{D}_{s}^{q-1})^\top]_{\Gamma_1} {\mathbf{e}}^q_{2, n+1/2}.
\end{aligned}
\end{equation}
\end{proposition}
\begin{proof}
The time discrete energy rates are obtained by vector multiplication of Systems \eqref{eq:timediscr_primal} and \eqref{eq:timediscr_dual} by $\begin{pmatrix}
    \dual{\mathbf{e}}^p_{1, n+1/2} \\ [\dual{\mathbf{e}}^{p-1}_{2, n+1/2}]_{I\cup \Gamma_1}
\end{pmatrix}$ and $\begin{pmatrix}
    [\mathbf{e}^{q-1}_{1, n+1/2}]_{I\cup \Gamma_2} \\ {\mathbf{e}}^q_{2, n+1/2}
    \end{pmatrix}$ respectively. This computation provides the implicit midpoint discretization of Propositions \ref{pr:discr_dotH_primal} and \ref{pr:discr_dotH_dual}.
\end{proof}

The algebraic systems \eqref{eq:timediscr_primal} and \eqref{eq:timediscr_dual} satisfy a time discrete power balance, together with time discrete energy rates.
\begin{proposition}\label{pr:discrtime_Hdot}
Given the recursions implemented by systems \eqref{eq:timediscr_primal} and \eqref{eq:timediscr_dual}, the regularity assumption reported in Eq. \eqref{eq:reg_C}, and the following definition of time discrete power
\begin{equation*}
    P_{h, n+1/2} := \frac{1}{\Delta t}\dualpr[M]{e^{q-1}_{1, h, n+1/2}}{\dual{C}^p (\dual{e}^p_{1, h, n+1}- \dual{e}^p_{h, n})} + \frac{1}{\Delta t}\dualpr[M]{\dual{e}^{p-1}_{2, h, n+1/2}}{E^q ({e}^q_{2, h, n+1}-{e}^q_{2, h, n})},
\end{equation*}
it holds
\begin{equation}
    P_{h, n+1/2} = (\dual{\mathbf{e}}_{\partial, n+1/2}^{p-1})^\top \mathbf{\Psi}^{q-1}_{s, \partial}\mathbf{f}_{\partial, n+1/2}^{q-1}.
\end{equation}
\end{proposition}
\begin{proof}
The time discrete power balance corresponds to the implicit midpoint time discretization of Proposition \ref{pr:discr_Hdot}.
\end{proof}

\section{Numerical experiments}\label{sec:num_exp}

In this section, the dual field discretization methodology proposed in this work is tested for the wave and Maxwell equations in a box-shaped three-dimensional domain
$$M = \{ (x,y,z) \in [0, 1]\times[0, 1/2]\times[0, 1/2] \}.$$ 
The boundary sub-partitions are selected to be
\begin{equation*}
    \Gamma_1 = \{(x,y,z) \vert \; x=0 \cup y=0 \cup z=0\}, \qquad \Gamma_2 = \{(x, y, z) \vert \; x=1 \cup y=1/2 \cup z=1/2 \}.
\end{equation*}

Since we want to rely on existing and well-established librairies like \fenics and \firedrake, the equations need to be translate into vector calculus operations. By introducing the musical isomorphism, given by the flat $\flat$ and the sharp operator $\sharp$ (cf. \ref{app:vec_ext}), and the isomorphism $\beta$ converting vector fields in $n-1$ forms, the commuting diagram in Fig. \ref{fig:cd_ext_vec}, that provides the link between the de Rham complex and the standard operators and Sobolev space from vector calculus, is obtained.
\begin{figure}[h]
\centering
\begin{tikzcd}
H\Omega^0(M) \arrow[r, "\d"] \arrow[d, leftrightarrow, "Id"]
& H\Omega^{1}(M) \arrow[r, "\d"] \arrow[d, "\sharp", xshift=10pt] & H\Omega^2(M) \arrow[r, "\d"] \arrow[d, "\beta^{-1}", xshift=10pt]
& H\Omega^{3}(M) \arrow[d, "\star", xshift=10pt]  \\
H^1(M) \arrow[r, "\grad"]
& H^{\curl}(M) \arrow[r, "\curl"] \arrow[u, "\flat"] & H^{\div}(M) \arrow[r, "\div"] \arrow[u, "\beta"]
& L^2(M) \arrow[u, "\star^{-1}"]
\end{tikzcd}  
\caption{Equivalence of vector and exterior calculus Sobolev spaces.}
\label{fig:cd_ext_vec}
\end{figure}

Since the manifold is a subset of the Euclidean space, the metric tensor is the identity $g_{ij}~=~\delta_{ij}$. The finite element arising from the trimmed polynomial family on the computational mesh $\mathcal{T}_h$ are then equivalent to the well known continuous Galerkin (or Lagrange) elements $\mathcal{P}^-_s\Omega^0(\mathcal{T}_h) \equiv \mathrm{CG}_s(\mathcal{T}_h)$, Nédélec of the first kind $\mathcal{P}^-_s\Omega^1(\mathcal{T}_h) \equiv \mathrm{NED}_s^1(\mathcal{T}_h)$, Raviart-Thomas $\mathcal{P}^-_s\Omega^2(\mathcal{T}_h) \equiv \mathrm{RT}_s(\mathcal{T}_h)$ and discontinuous Galerkin $\mathcal{P}^-_s\Omega^3(\mathcal{T}_h) \equiv \mathrm{DG}_{s-1}(\mathcal{T}_h)$, as illustrated in  Figure \ref{fig:fe_ext_vec}.

\begin{figure}[h]
\centering
\begin{tikzcd}
H^1(M) \arrow[r, "\grad"] \arrow[d, "\Pi_{s, h}^{-, 0}"]
& H^{\curl}(M) \arrow[r, "\curl"] \arrow[d, "\Pi_{s, h}^{-, 1}"] & H^{\div}(M) \arrow[r, "\div"] \arrow[d, "\Pi_{s, h}^{-, 2}"]
& L^2(M) \arrow[d, "\Pi_{s, h}^{-, 3}"]  \\
\mathrm{CG}_s(\mathcal{T}_h) \arrow[r, "\grad"] 
& \mathrm{NED}_s^1(\mathcal{T}_h) \arrow[r, "\curl"] & \mathrm{RT}_s(\mathcal{T}_h) \arrow[r, "\div"]
& \mathrm{DG}_{s-1}(\mathcal{T}_h) 
\end{tikzcd}  
\caption{Equivalence between finite element differential forms and classical elements.}
\label{fig:fe_ext_vec}
\end{figure}

The finite element library \firedrake \cite{rathgeber2017firedrake} is used for the numerical investigation.

\begin{remark}
As argued in \cite{benner2015timebc}, the method of manufactured solution is not suited for boundary control problems. This is due to the fact that for finer discretizations the volume terms due to the forcing will dominate over the boundary terms. This is the reason why the numerical tests are set up considering an eigensolution rather than a manufactured one induced by a forcing.
\end{remark}

\subsection{The acoustic wave equation in $3D$}

The acoustic wave equation corresponds to the case $p=3$ and $q=1$. \revTwo{Using the same notation as in Sec. \ref{sec:ex_wave},} the energy variables are the top-form $\dual{v}^3:=\dual{\alpha}^3$ (corresponding to the pressure) and the one-form $\sigma^1:=\beta^{1}$ (corresponding to the linear momentum). If the physical coefficients are normalized to one (this can be easily achieved by re-scaling the time with respect to the speed of sound), the Hamiltonian is given by
\begin{equation}
    H(\dual{v}^3, \sigma^1) = \frac{1}{2} \int_M \dual{v}^3 \wedge \star \dual{v}^3 + \sigma^1 \wedge \star \sigma^1,
\end{equation}
with its variational derivatives given by
\begin{equation}
v^0:=\delta_{\dual{v}^3} H = \star \dual{v}^3, \qquad  \dual{\sigma}^2:= \delta_{\sigma^1} H = \star \sigma^1,  
\end{equation}
leading to the pH system
\begin{equation}\label{eq:wave_eq}
    \begin{pmatrix}
    \partial_t \dual{v}^3 \\
    \partial_t \sigma^1
    \end{pmatrix} =
    -
    \begin{bmatrix}
    0 & \d \\
    \d & 0 \\
    \end{bmatrix}
    \begin{pmatrix}
     v^0 \\
     \dual{\sigma}^2
    \end{pmatrix}, \qquad 
    \begin{aligned}
    \tr v^0 \vert_{\Gamma_1} &= u^{0}_1, \\
    -\tr \dual{\sigma}^2 \vert_{\Gamma_2} &= \dual{u}^2_2.
 \end{aligned}
\end{equation}

Given the functions
\begin{equation}
    g(x, y, z) = \cos(x) \sin(y) \sin(z), \qquad f(t) = 2 \sin(\sqrt{3} t) + 3 \cos(\sqrt{3} t),
\end{equation}
an exact solution of \eqref{eq:wave_eq} is given by
%the system  is identically solved by
\begin{equation}\label{eq:wave_exsol}
\begin{aligned}
\dual{v}^3_{\mathrm{ex}} &= \star g \diff{f}{t}, \\    
\sigma^1_{\mathrm{ex}} &= -\d{g} f, 
\end{aligned} \qquad 
\begin{aligned}
v^0_{\mathrm{ex}} &= g \diff{f}{t}, \\
\dual{\sigma}^2_{\mathrm{ex}} &= -\star \d{g} f,
\end{aligned}
\end{equation}
The exact solution provides the appropriate inputs to be fed into the system
\begin{equation}
    u^0_1 = \left. \tr v^0_{\mathrm{ex}} \right\vert_{\Gamma_1}, \qquad u^{2}_2 =  -\tr \dual{\sigma}^2_{\mathrm{ex}} \vert_{\Gamma_2}.
\end{equation}

\revTwo{The employment of the dual field discretization leads to the resolution of two systems:
\begin{itemize}
    \item the primal system \eqref{eq:alg_primalPH} of outer oriented variables $\dual{v}^3$ and $\dual{\sigma}^2$;
    \item the dual system \eqref{eq:alg_dualPH} of inner oriented variables $v^0$ and $\sigma^1$;
\end{itemize}
Each variables is discretized using the associated finite element differential forms:
\begin{itemize}
    \item discontinuous Galerkin elements DG$_{s-1}$ for $\dual{v}^3_h$;
    \item Raviart Thomas elements RT$_s$ for $\dual{\sigma}^2_h$.
    \item continuous Galerkin elements CG$_s$ for $v^0_h$;
    \item Nédélec elements of the first kind NED$_s^1$ for $\sigma^1_h$;
\end{itemize}
}

\subsubsection{Energy Conservation properties}
First the conservation properties of the scheme are verified against the exact solution \eqref{eq:wave_exsol}. The test is performed using $N_{\text{el}}=4$ elements for each side of the box-shaped domain and polynomial degree $s=3$. Concerning the time discretization, the total simulation time $T_{\text{end}}=5$ and the time step is taken to be $\Delta t= \frac{T_{\text{end}}}{200}$. \\

The energy rate conservation, expected by Prop. \ref{pr:discrtime_energyrate}, is reported in Fig. \ref{fig:con_Hdot_wave}. Fig. \ref{fig:con_P_wave} shows the fulfilment of the discrete power balance (Prop. \ref{pr:discrtime_Hdot}), whereas Fig. \ref{fig:err_flow_wave} provides the error between the discrete and exact power flow. The power flow error due to the polynomial interpolation is lower than $10^{-4}$. 
For what concerns the energy behaviour, three different energies are considered
\begin{equation*}
\begin{aligned}
    \dual{H}^{3}_h &= \frac{1}{2} \int_M \dual{v}^3_h \wedge \star \dual{v}^3_h + \dual{\sigma}^2_h \wedge \star \dual{\sigma}^2_h, \\
    H^{1}_h&= \frac{1}{2} \int_M v^0_h \wedge \star v^0_h + \sigma^1_h \wedge \star \sigma_h^1, \\
    \frac{H_{T, h}}{2} &= \frac{1}{2} \int_M v^0_h \wedge \dual{v}^3_h + \dual{\sigma}^2_h \wedge \sigma_h^1. \\
\end{aligned}
\end{equation*}
The first two energies are associated with the mixed discretization \eqref{eq:alg_primalPH} and \eqref{eq:alg_dualPH} respectively, whereas the last one is constructed using the dual field formulation. The errors on the energy rate and variation of energy, shown in Fig. \ref{fig:energy_wave}, are of the order $10^{-5}$, assessing the performance of the dual field formulation. In particular, one can notice that the energy $\frac{H_T}{2}$ stays between $\dual{H}^{3}_h, \; H^{1}_h$ at all times. The variation of energy is also computed using the power flow 
\begin{equation*}
    \Delta H = \int_0^t P_h(\tau) \d\tau, \qquad P_h = \int_M v^0_h \wedge \partial_t \dual{v}^3_h + \dual{\sigma}^2_h \wedge \partial_t \sigma_h^1.
\end{equation*}
As shown in Fig. \ref{fig:deltaH_wave} this quantity is less affected by the error for this particular test case.

\begin{figure}[p]%
\centering
\subfloat[][Conservation law for $\dot{\dual{H}}^{3}_h - \dual{P}_h^{3}$]{%
	\label{fig:con_H32_wave}%
	\includegraphics[width=0.48\columnwidth]{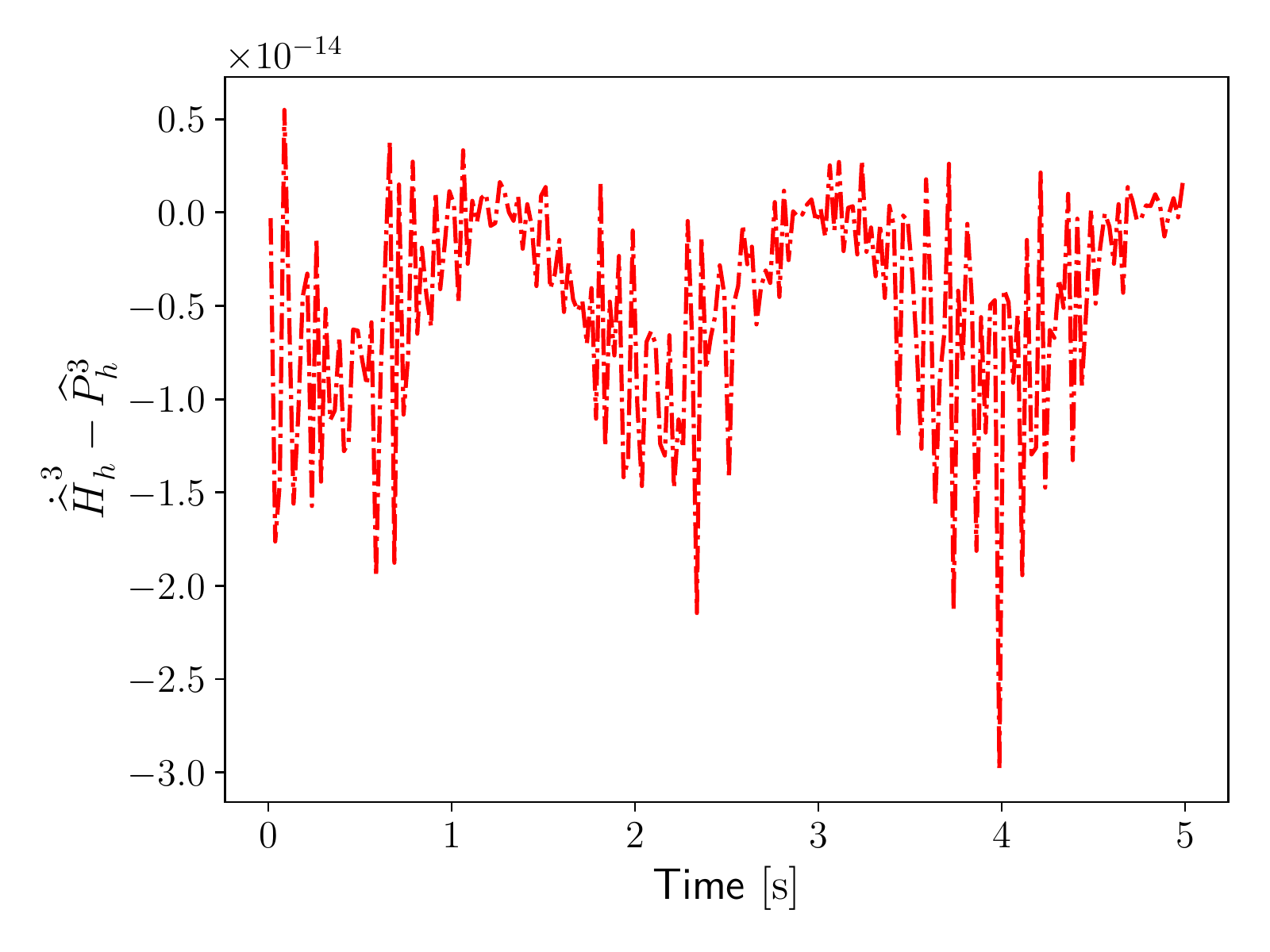}}%
\hspace{8pt}%
\subfloat[][Conservation law for $\dot{H}^{1}_h-P^{1}_h$]{%
	\label{fig:con_H10_wave}%
	\includegraphics[width=0.48\columnwidth]{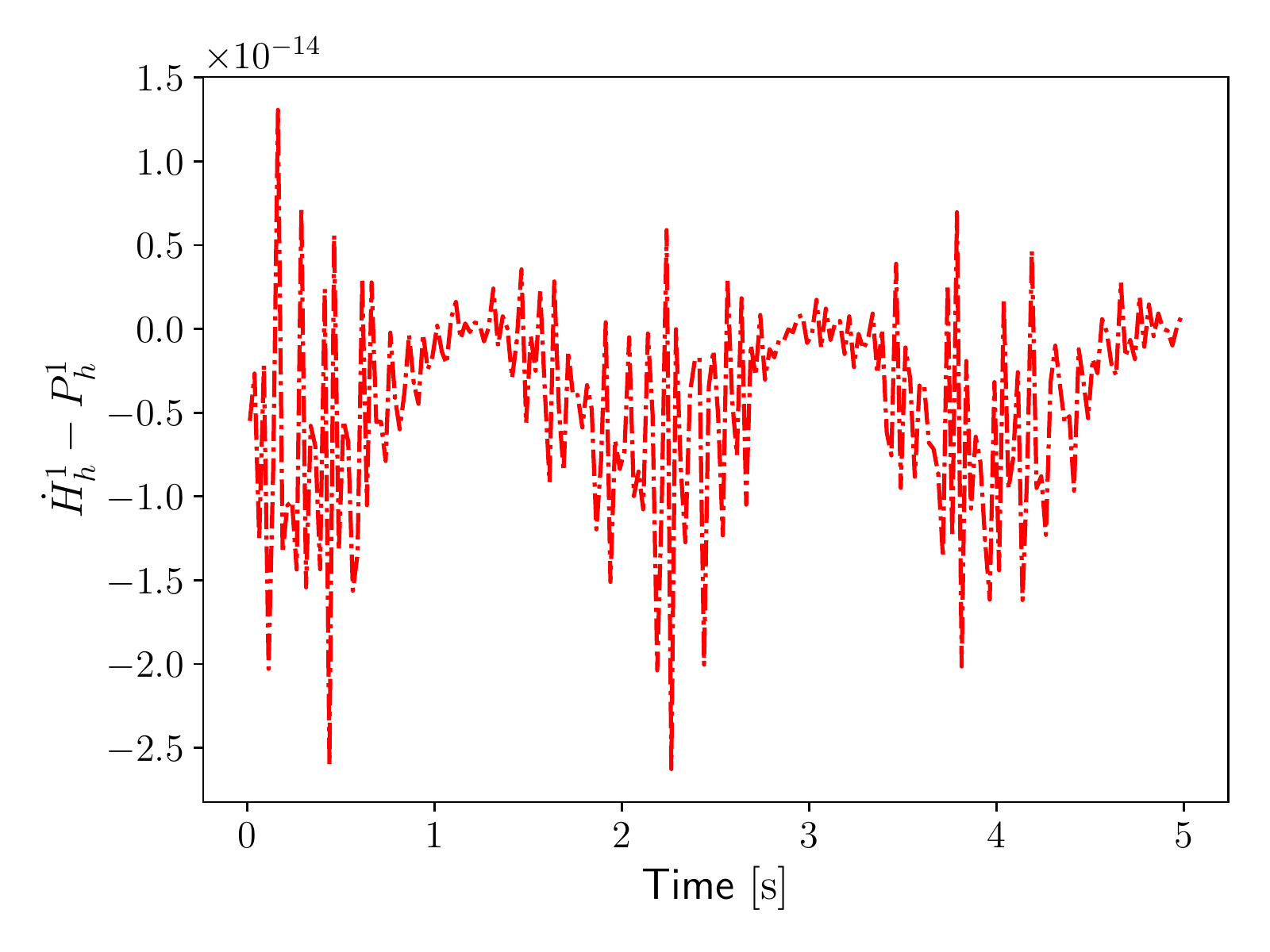}}%
\caption{Conservation properties given by Prop. \ref{pr:discrtime_energyrate} ($N_{\text{el}}=4,\; s=3$ and $\Delta t = \frac{5}{200}$).}%
\label{fig:con_Hdot_wave}%
\end{figure}

\begin{figure}[p]%
\centering
\subfloat[][Conservation of $P_h-<\dual{e}_{\partial, h}^2|f_{\partial, h}^0>_{\partial M}$]{%
	\label{fig:con_P_wave}%
	\includegraphics[width=0.48\columnwidth]{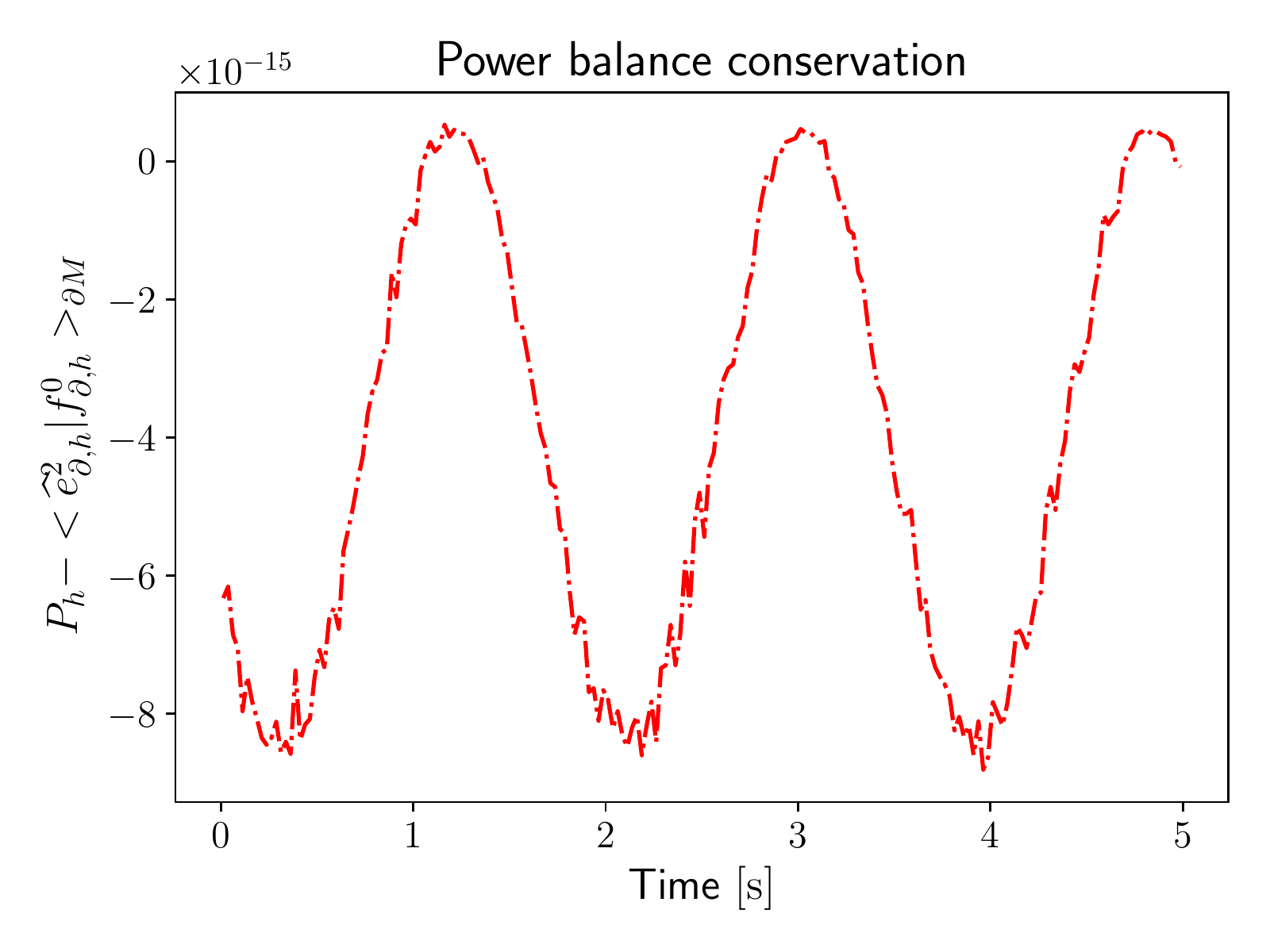}}%
\subfloat[][Error exact and interpolated boundary flow]{%
\label{fig:err_flow_wave}%
\includegraphics[width=0.48\columnwidth]{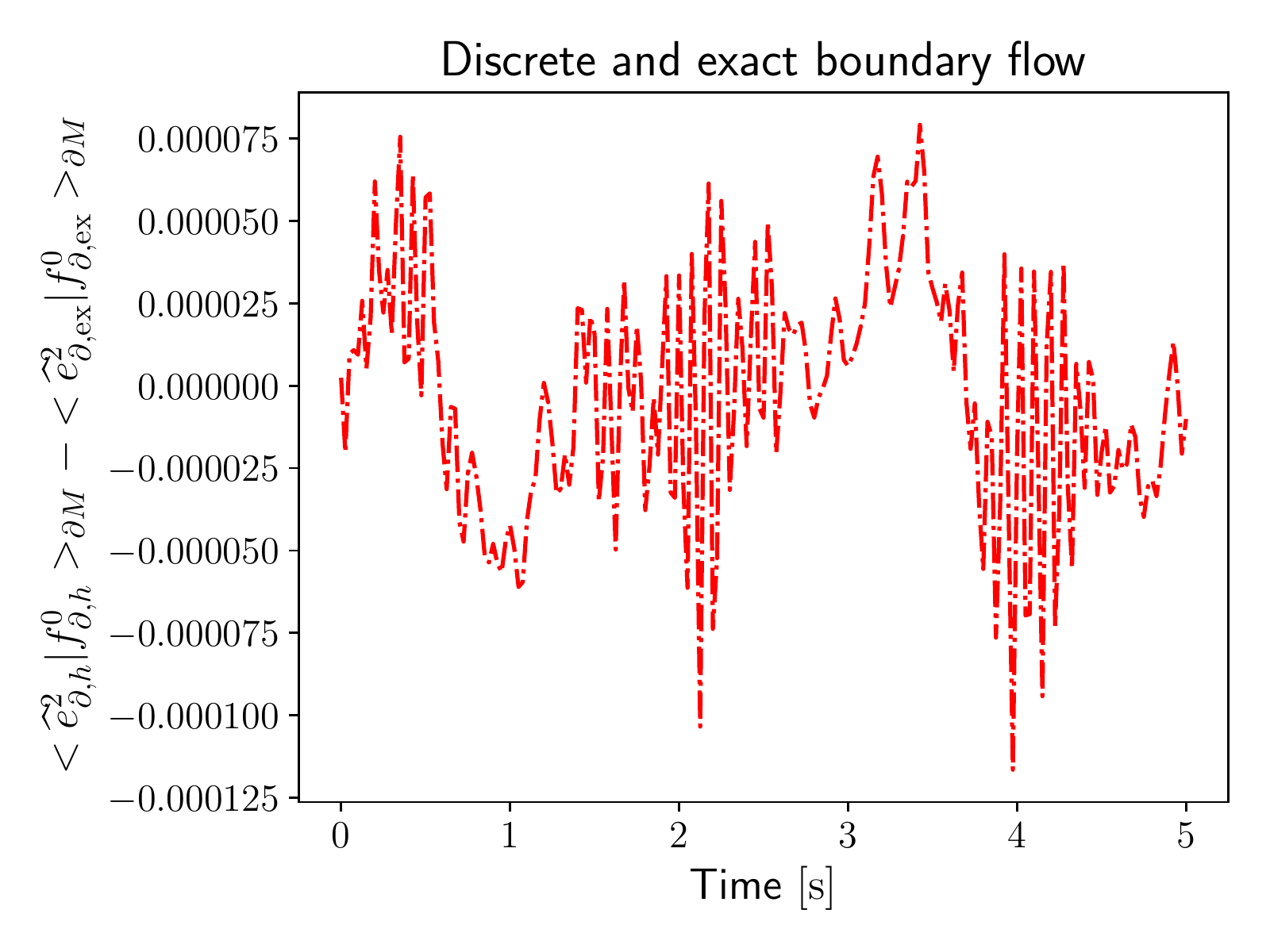}}%
\caption{Power balance given by Pr. \ref{pr:discrtime_Hdot} (left) and error on the power flow (right)  ($N_{\text{el}}=4,\; s=3, \;\Delta t = \frac{5}{200}$).}%
\label{fig:con_pow_wave}%
\end{figure}

\begin{figure}[p]%
\centering
\subfloat[][Error $\dot{H}$ and interpolated boundary flow]{%
\label{fig:err_dHdt_wave}%
\includegraphics[width=0.48\columnwidth]{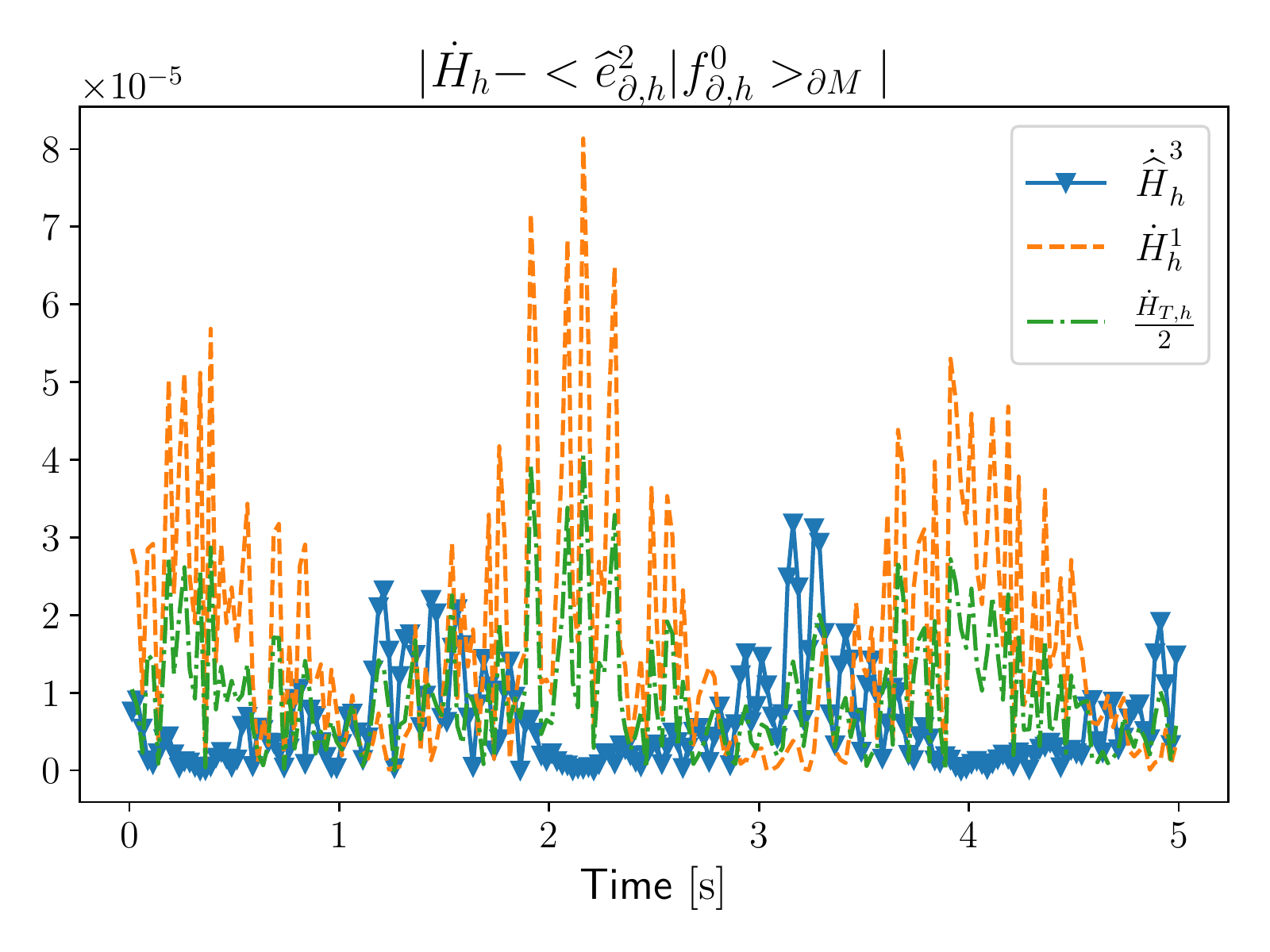}}%
\subfloat[][$\Delta H_h - \Delta H_{\mathrm{ex}}$]{%
	\label{fig:deltaH_wave}%
	\includegraphics[width=0.48\columnwidth]{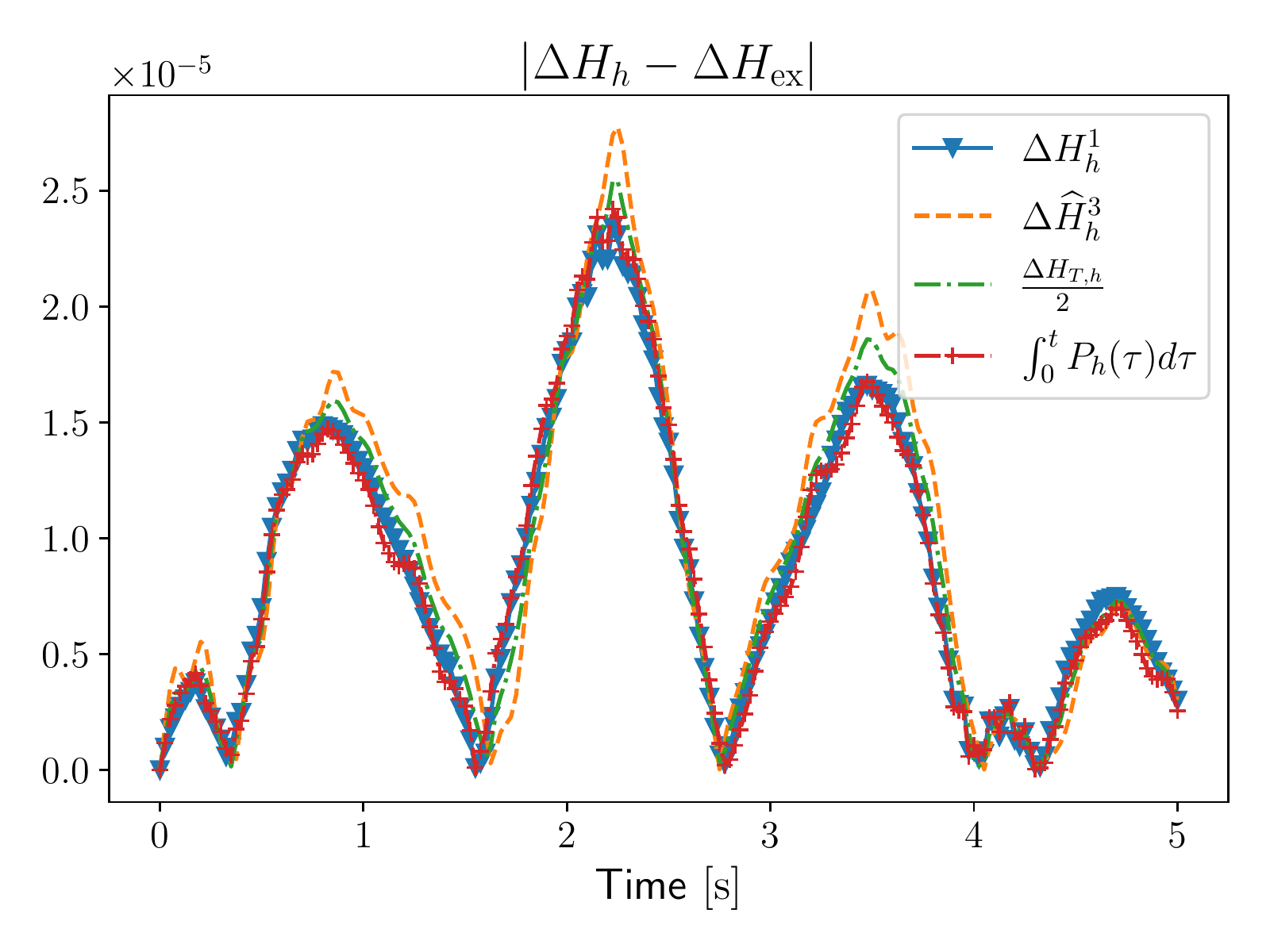}}%
\caption{Energy rate and energy variation error ($N_{\text{el}}=4,\; s=3$ and $\Delta t = \frac{5}{200}$).}%
\label{fig:energy_wave}%
\end{figure}

\begin{comment}
\begin{figure}[p]%
\centering
\subfloat[][$H_h - H_{\mathrm{ex}}$]{%
	\label{fig:H}%
	\includegraphics[width=0.48\columnwidth]{H_3D_DN.pdf}}%
\caption{Discrete and exact Hamiltonian for $\Delta t = \frac{5}{200}$.}%
\label{fig:H_ex}%
\end{figure}
\end{comment}

\subsubsection{Convergence results}
In this section, the convergence rate of the different variables is verified against the analytical solution \eqref{eq:wave_exsol}.  The error is evaluated in the $L^2$ norm at the ending time. Concerning the time discretization, the total simulation time $T_{\text{end}}=1$ and the time step is taken to be $\Delta t= \frac{T_{\text{end}}}{100}$. \\

In Fig. \ref{fig:conv_var_wave} the $L^2$ error trend against the exact solution is reported for all variables. It can be noticed that the error goes as $h^s$, exception made for $v^0$ (cf. Fig. \ref{fig:err_p0}) that for $s=1, 2$ stays between $h^s$ and $h^{s+1}$. Indeed, the presence of mixed and time-varying boundary conditions leads to a lower convergence rate than the expected theoretical order of $h^{s+1}$ for homogeneous boundary conditions (see \cite{wu2021hodgewave} for the error analysis of the Hodge wave equation under one case of homogeneous conditions). The other variables exhibit the same convergence trend as predicted by the analysis in \cite{wu2021hodgewave}. The $L^2$ difference between the dual representation of the variables is reported in Fig. \ref{fig:diff_dual_wave}. The difference of the dual representation converges as $h^s$. This is in accordance with the results obtained in \cite{zhang2021mass}, where the dual field formulation is employed to solve the Navier-Stokes equations in periodic domains.
 
\begin{figure}[p]%
\centering
\subfloat[][$L^2$ error for $\dual{v}^3_h$]{%
	\label{fig:err_p3}%
	\includegraphics[width=0.48\columnwidth]{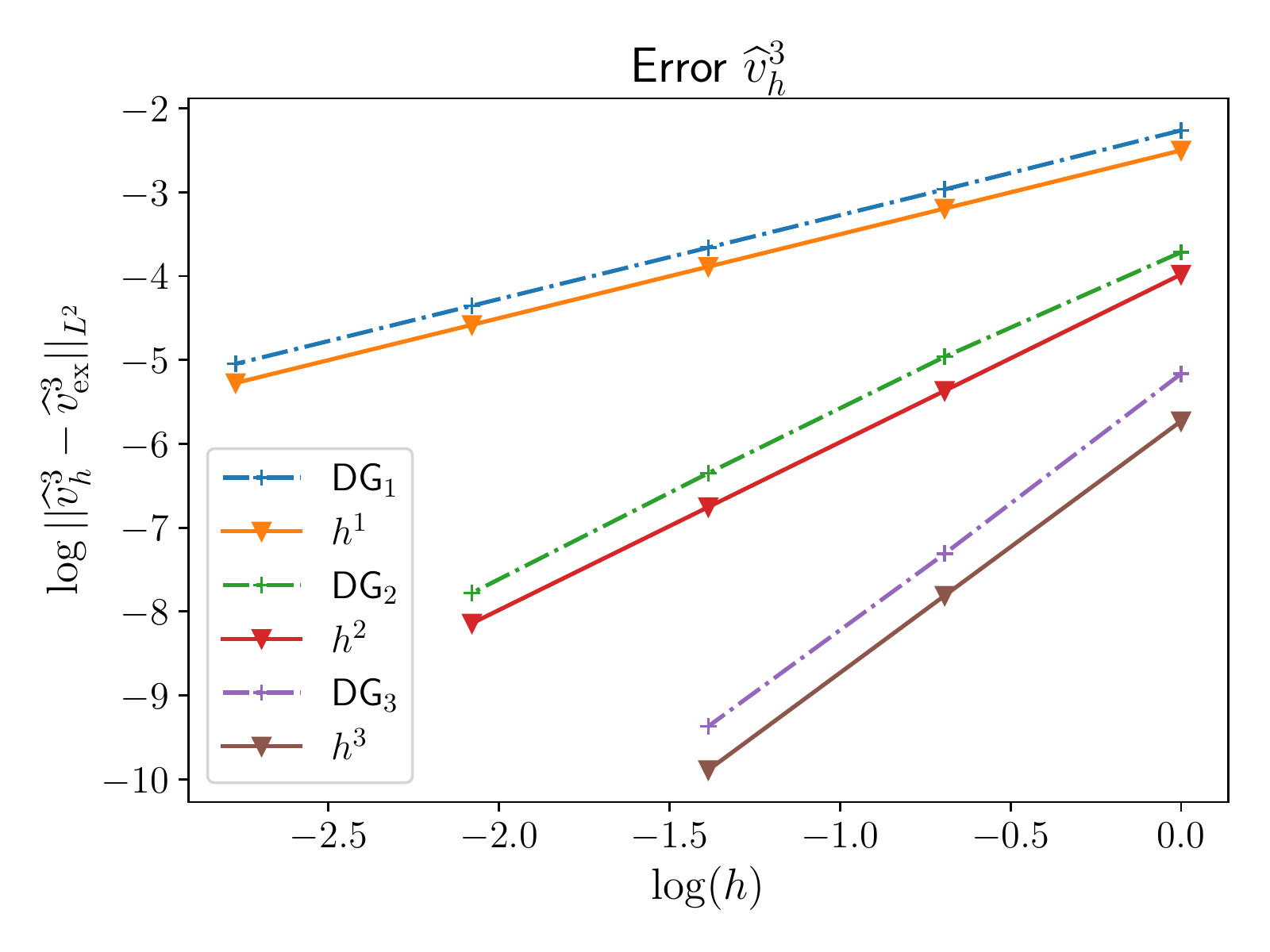}}%
\hspace{8pt}%
\subfloat[][$L^2$ error for $v^0_h$]{%
	\label{fig:err_p0}%
	\includegraphics[width=0.48\columnwidth]{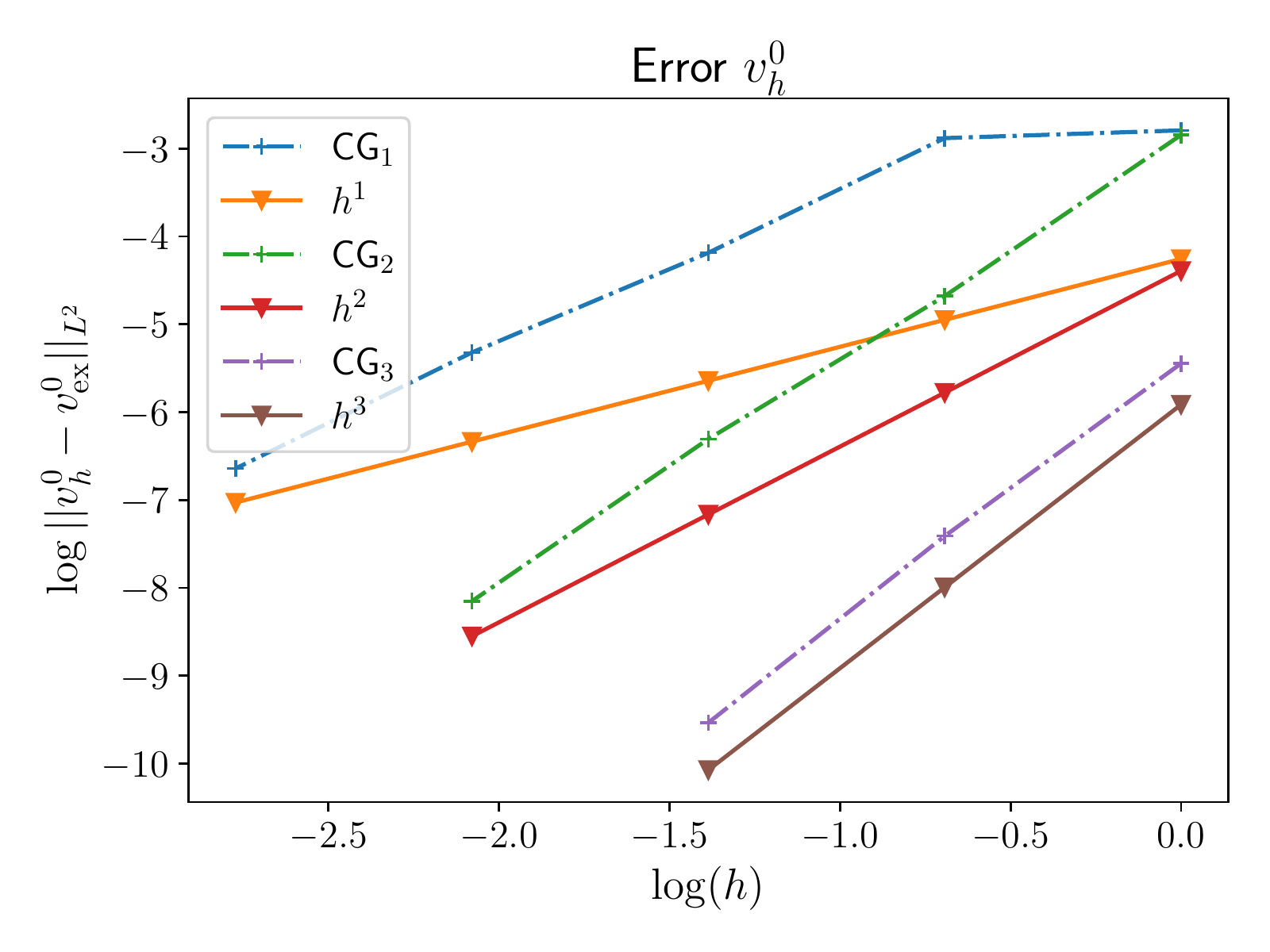}}%
\hspace{8pt}%
\subfloat[][$L^2$ error for $\sigma^1_h$]{%
	\label{fig:err_u1}%
	\includegraphics[width=0.48\columnwidth]{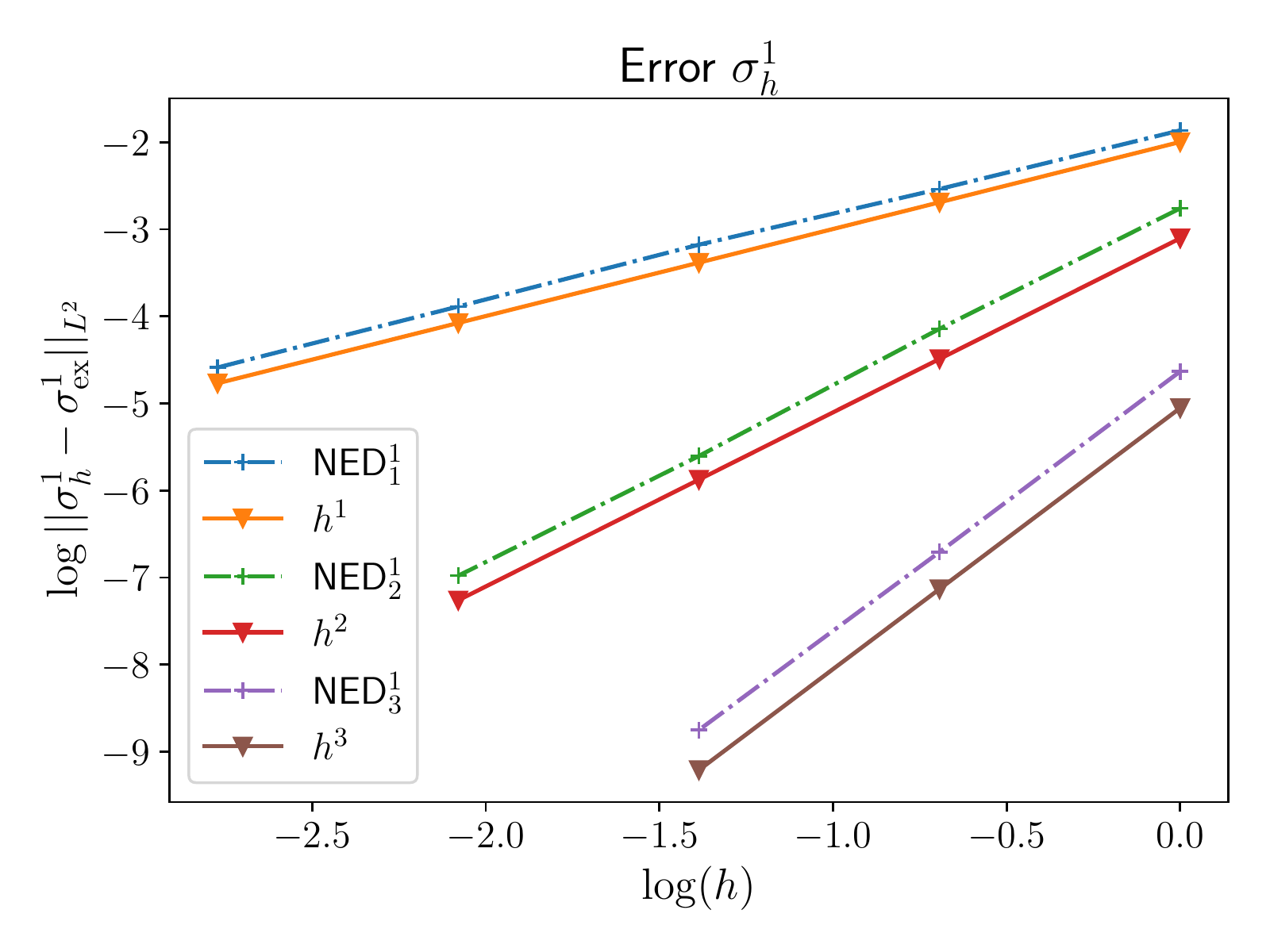}}
	\hspace{8pt}%
\subfloat[][$L^2$ error for $\dual{\sigma}^2_h$]{%
	\label{fig:err_u2}%
	\includegraphics[width=0.48\columnwidth]{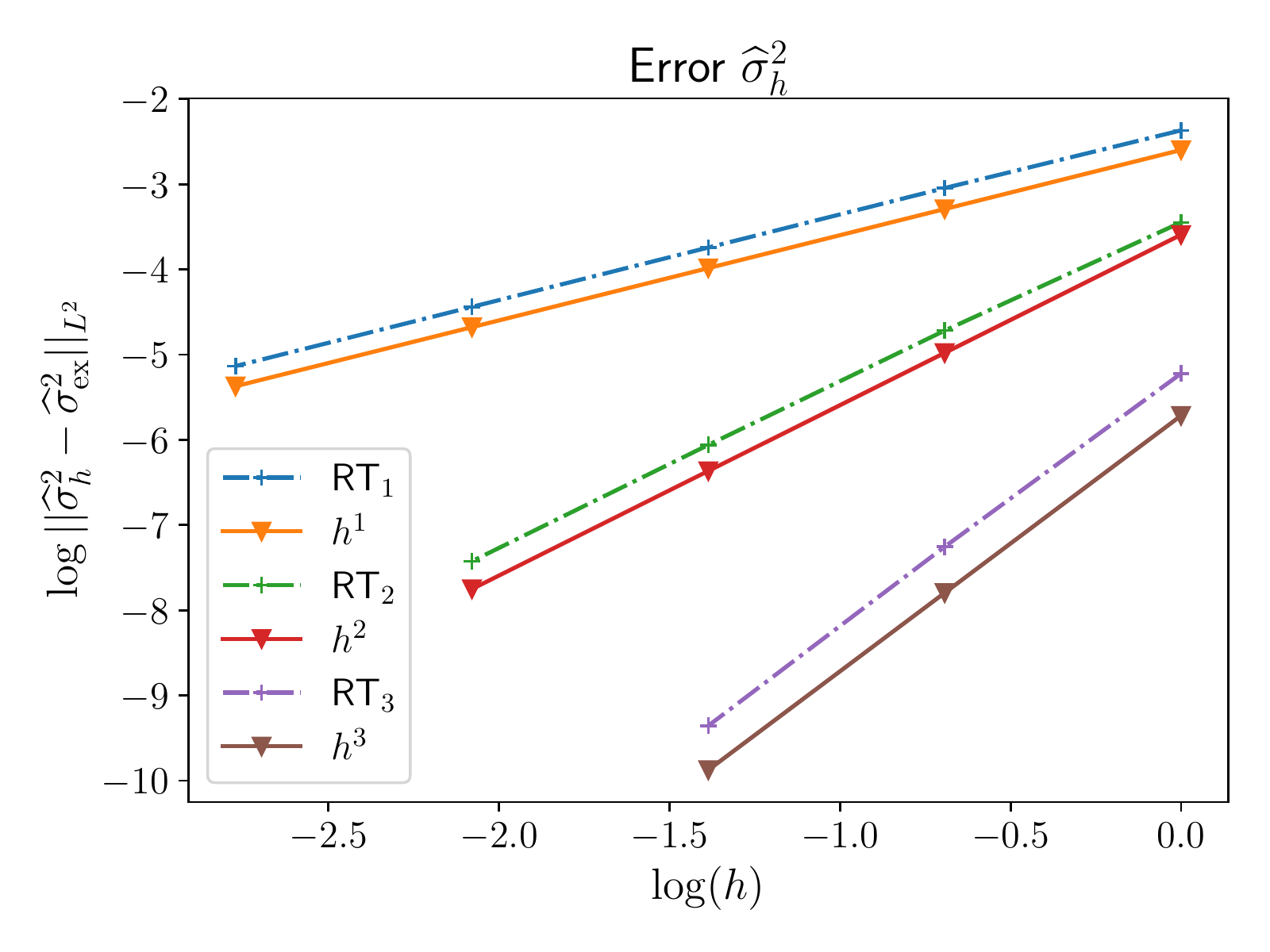}}
\caption{Convergence rate for the different variables in the wave equation at $T_{\text{end}}=1$ for $\Delta t = \frac{1}{100}$.}%
\label{fig:conv_var_wave}%
\end{figure}

\begin{figure}[p]%
\centering
\subfloat[][$L^2$ norm of the difference $\dual{v}^3_h - v^0_h$]{%
	\label{fig:diff_p30}%
	\includegraphics[width=0.48\columnwidth]{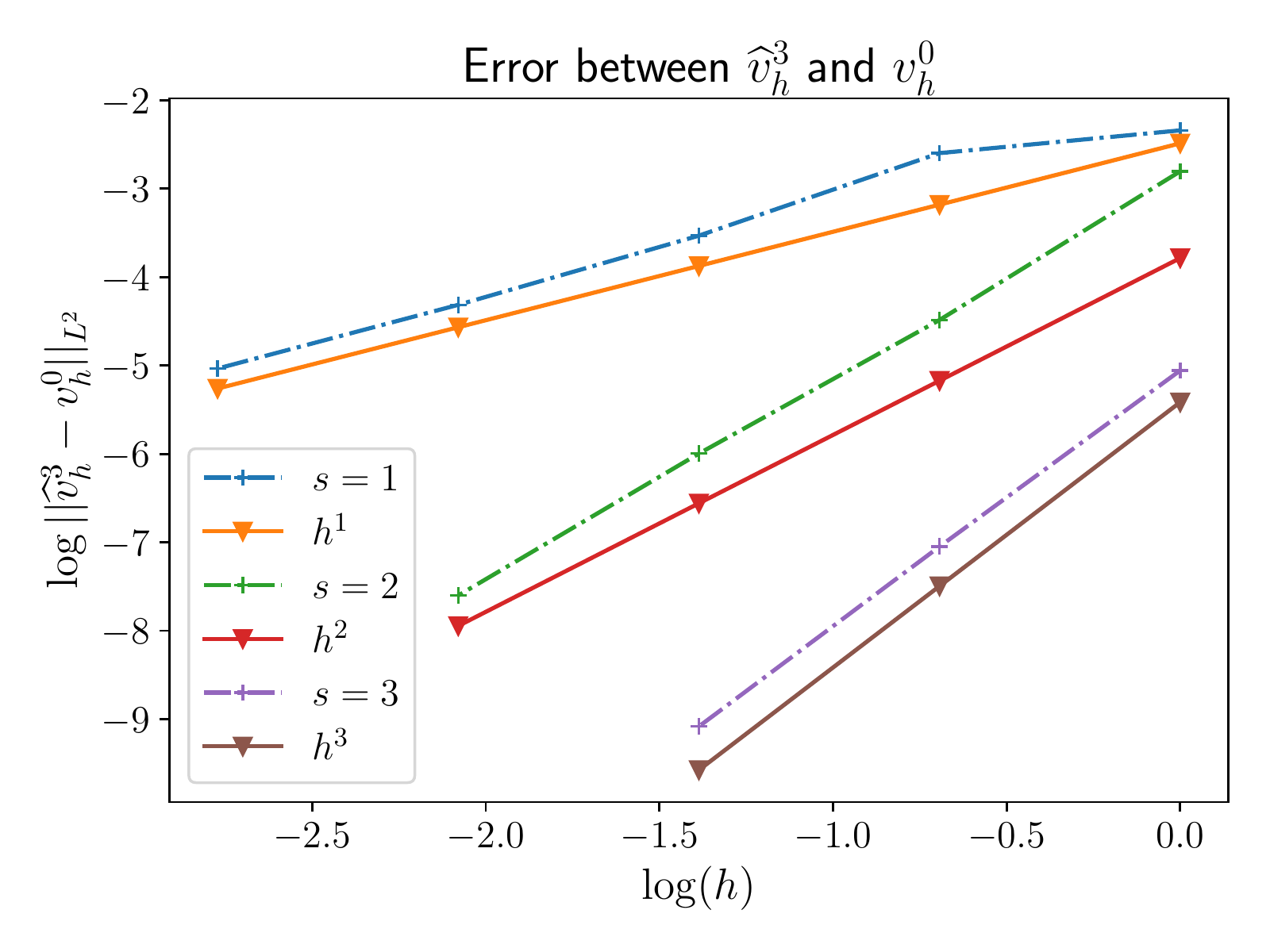}}%
\hspace{8pt}%
\subfloat[][$L^2$ norm of the difference $\sigma^1_h - \dual{\sigma}^2_h$]{%
	\label{fig:diff_q12}%
	\includegraphics[width=0.48\columnwidth]{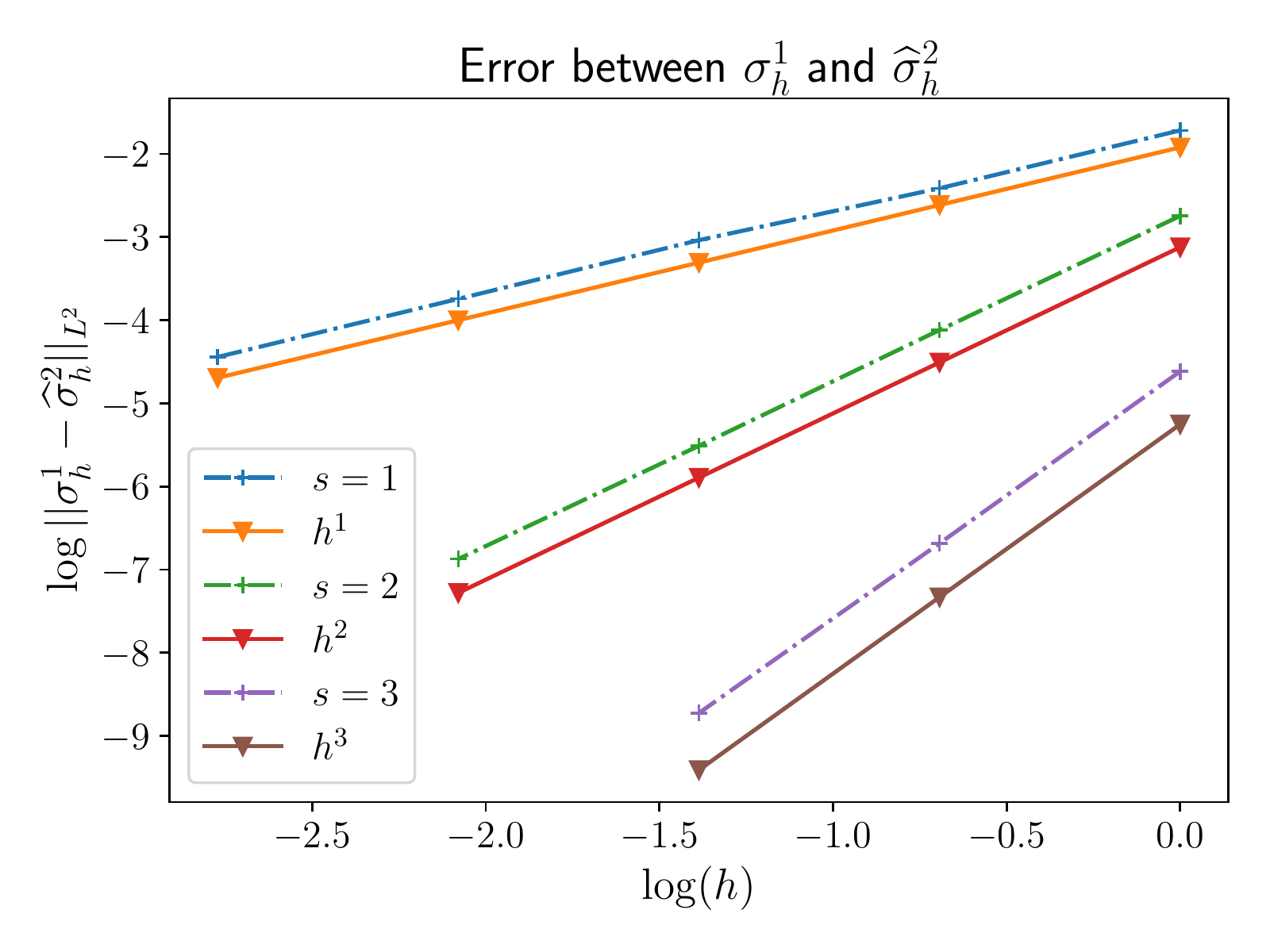}}%
\caption{$L^2$ difference of the dual representation of the solution for the wave equation at $T_{\text{end}}=1$ for $\Delta t = \frac{1}{100}$.}%
\label{fig:diff_dual_wave}%
\end{figure}

\subsection{The Maxwell equations in $3D$}
The Maxwell equations corresponds to the case $p=2, \; q=2$. The energy variables correspond to the electric displacement two form $\dual{d}^2 = \dual{\alpha}^2$ and the magnetic field $b^2=\beta^2$. The Hamiltonian reads
\begin{equation}
    H(\dual{d}^2, b^2) = \frac{1}{2} \int_M \varepsilon^{-1} \dual{d}^2 \wedge \star \dual{d}^2 + \mu^{-1} b^2 \wedge \star b^2,
\end{equation}
where $\varepsilon \in \bbR$ is the electric permittivity and $\mu \in \bbR$ is the magnetic permeability. The variational derivative of the Hamiltonian are given by 
\begin{equation}
    e^1 := \delta_{\dual{d}^2} H = \varepsilon^{-1} \star \dual{d}^2, \qquad \dual{h}^1 := \delta_{b^2} H = \mu^{-1} \star b^2.
\end{equation}
Variables $e^1, \; \dual{h}^1$ are the electric field and the magnetizing field respectively. Since the reduction of the constitutive equation is such to keep only the efforts variables and their duals, the following dynamical system is obtained.
\begin{equation}\label{eq:maxwell_eq}
\begin{bmatrix}
\varepsilon & 0 \\
0 & \mu \\
\end{bmatrix}
    \begin{pmatrix}
    \partial_t \dual{e}^2\\
    \partial_t h^2
    \end{pmatrix} = 
    \begin{bmatrix}
    0 & \d^1 \\
    -\d^1 & 0 \\
    \end{bmatrix}
    \begin{pmatrix}
    {e}^1\\
    \dual{h}^1
    \end{pmatrix},
\end{equation}
where $\dual{e}^2 = \star e^1, \; h^2= \star \dual{h}^1$. Given the functions
\begin{equation}
    \bm{g}(x, y, z) = \begin{pmatrix}
    -\cos(x)\sin(y)\sin(z) \\
    0 \\
    \sin(x)\sin(y)\cos(z)
    \end{pmatrix}, \qquad f(t) = \frac{\sin(\omega t)}{\omega},
\end{equation}
where $\omega = \sqrt{3} c$ and $c=(\sqrt{\mu \varepsilon})^{-1}$, the system \eqref{eq:maxwell_eq} is solved by the eigenmode 
\begin{equation}\label{eq:maxwell_exsol}
\begin{aligned}
\dual{e}^2_{\mathrm{ex}} &= \mu \star \bm{g}^\flat \diff{f}{t}, \\    
h^2_{\mathrm{ex}} &= -\d{\bm{g}^\flat} f, 
\end{aligned} \qquad
\begin{aligned}
e^1_{\mathrm{ex}} &= \mu \bm{g}^\flat \diff{f}{t}, \\    
\dual{h}^1_{\mathrm{ex}} &= -\star \d{\bm{g}^\flat} f.
\end{aligned}
\end{equation}
The exact solution provides the appropriate inputs to be fed into the system
\begin{equation}
    u^1_1 = \left. \tr e^1_{\mathrm{ex}}  \right\vert_{\Gamma_1}, \qquad \dual{u}^1_2 = \tr \dual{h}^1_{\mathrm{ex}} \vert_{\Gamma_2}.
\end{equation}

\revTwo{
The employment of the dual field discretization leads to the resolution of two systems:
\begin{itemize}
    \item the primal system \eqref{eq:alg_primalPH} of outer oriented variables $\dual{e}^2, \dual{h}^1$;
    \item the dual system \eqref{eq:alg_dualPH} of inner oriented variables ${e}^1, h^2$.
\end{itemize}
The discrete variables are represented by
\begin{itemize}
    \item Raviart-Thomas elements RT$_s$ for $e^2_h$ and $h^2_h$,
    \item Nédélec elements of the first kind NED$_s^1$ for $e^1_h$ and $h^1_h$.
\end{itemize}
}
For the numerical test the electric permittivity and magnetic permeability take the values
\begin{equation*}
    \mu = \frac{3}{2}, \qquad \varepsilon = 2.
\end{equation*}

\subsection{Conservation properties}

The conservation properties of the scheme are verified against the exact solution \eqref{eq:maxwell_exsol}. The test is performed using $N_{\text{el}}=4$ elements for each side of the box-shaped domain and polynomial degree $s=3$. Once again, the total simulation time $T_{\text{end}}=5$ and the time step is taken to be $\Delta t= \frac{T_{\text{end}}}{200}$. \\

An important feature of the Maxwell equations \eqref{eq:maxwell_eq} is that they verify the following constraints
\begin{equation*}
    \d^2 \dual{e}^2 = 0, \qquad \d^2 h^2=0.
\end{equation*}
This result follows by taking the exterior derivative of each line of system \eqref{eq:maxwell_eq}, under the assumption that the initial conditions respect these constraints. Mixed finite element strategies, like the ones proposed in \cite{asad2019maxwell,farle2013,payen2020}, cannot satisfy both constraints as they do not employ a dual representation for each variable. Instead, the dual field formulation naturally capture this aspect as shown in Fig. \ref{fig:div_maxwell}. The energy rate conservation and discrete power balance are reported in Figs. \ref{fig:con_Hdot_maxwell} and \ref{fig:con_P_maxwell} respectively. The numerical test confirms once again the expected behaviour. The power flow error due to the polynomial interpolation is of the order of $10^{-5}$ (cf. Fig. \ref{fig:err_flow_wave}). 
For what concerns the energy behaviour, three different energies are once again considered
\begin{equation*}
\begin{aligned}
    \dual{H}^{2}_h&= \frac{1}{2} \int_M \varepsilon \dual{e}^2_h \wedge \star \dual{e}^2_h + \mu \dual{h}^1_h \wedge \star \dual{h}_h^1, \\
    H^{2}_h&= \frac{1}{2} \int_M \varepsilon e^1_h \wedge \star e^1_h + \mu h^2_h \wedge \star h_h^2, \\
    \frac{H_T}{2} &= \frac{1}{2} \int_M \varepsilon e^1_h \wedge \dual{e}^2_h + \mu \dual{h}^1_h \wedge h_h^2. \\
\end{aligned}
\end{equation*}
Fig. \ref{fig:energy_maxwell} shows that the error on the energies are of the order $10^{-5}$. The dual field energy $\frac{H_T}{2}$ stays in the middle of $\dual{H}^{2}_h, \; H^{2}_h$. The variation of energy is also computed using the power flow 
\begin{equation*}
    \Delta H = \int_0^t P_h(\tau) \d\tau, \qquad P_h = \int_M \varepsilon e^1_h \wedge \partial_t \widehat{e}^2_h + \mu \widehat{h}^1_h \wedge \partial_t h_h^2.
\end{equation*}
Indeed this variation of the energy is not the most precise (cf. Fig. \ref{fig:deltaH_maxwell}). A rigorous error analysis is needed to assess the conditions under which one of these energies perform better.

\begin{figure}[tbh]%
\centering
\subfloat[][$L^2$ norm divergence of $\dual{e}_h^2$]{%
	\label{fig:divE2}%
	\includegraphics[width=0.48\columnwidth]{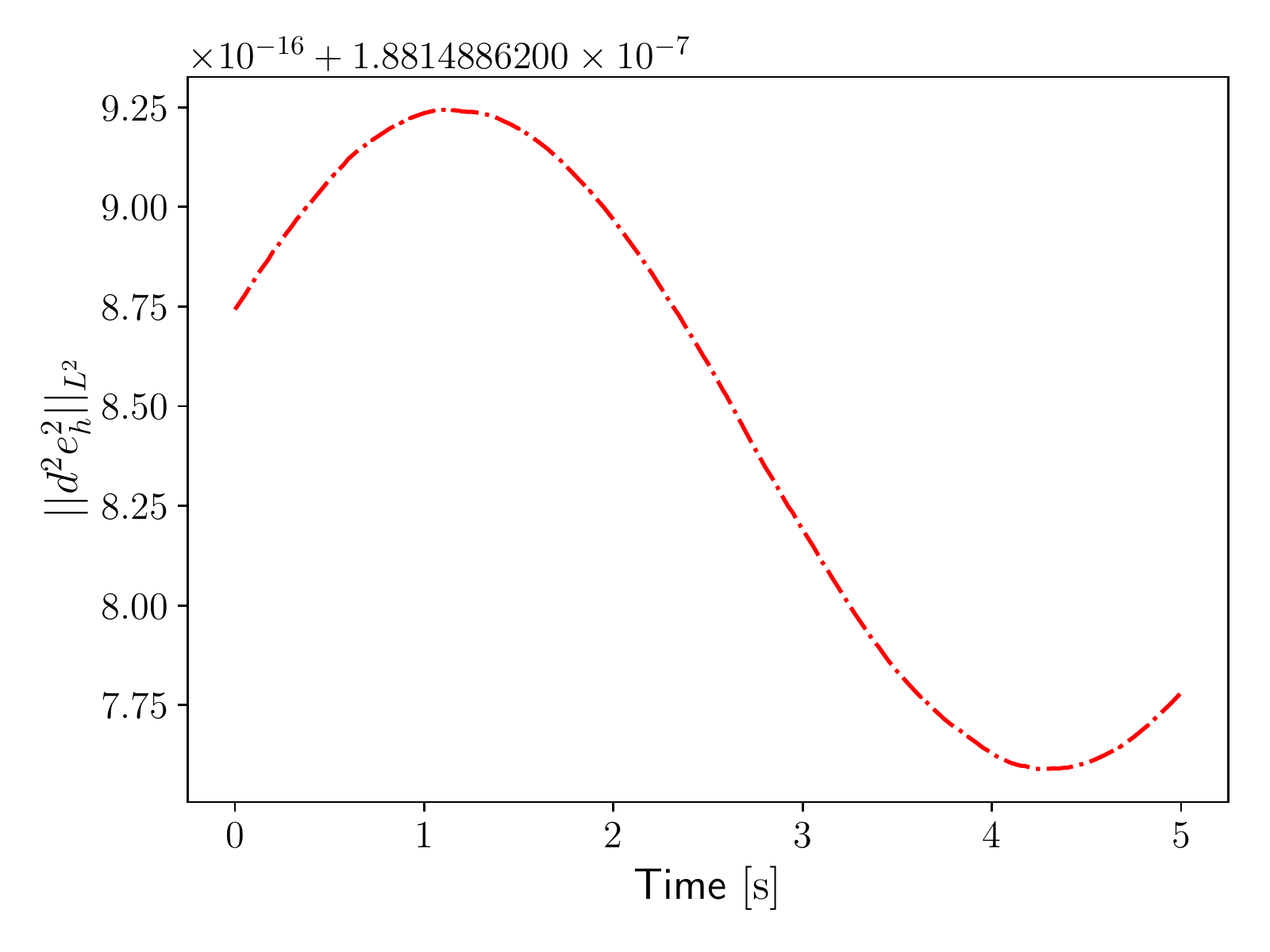}}%
\hspace{8pt}%
\subfloat[][$L^2$ norm divergence of $h_h^2$]{%
	\label{fig:divH2}%
	\includegraphics[width=0.48\columnwidth]{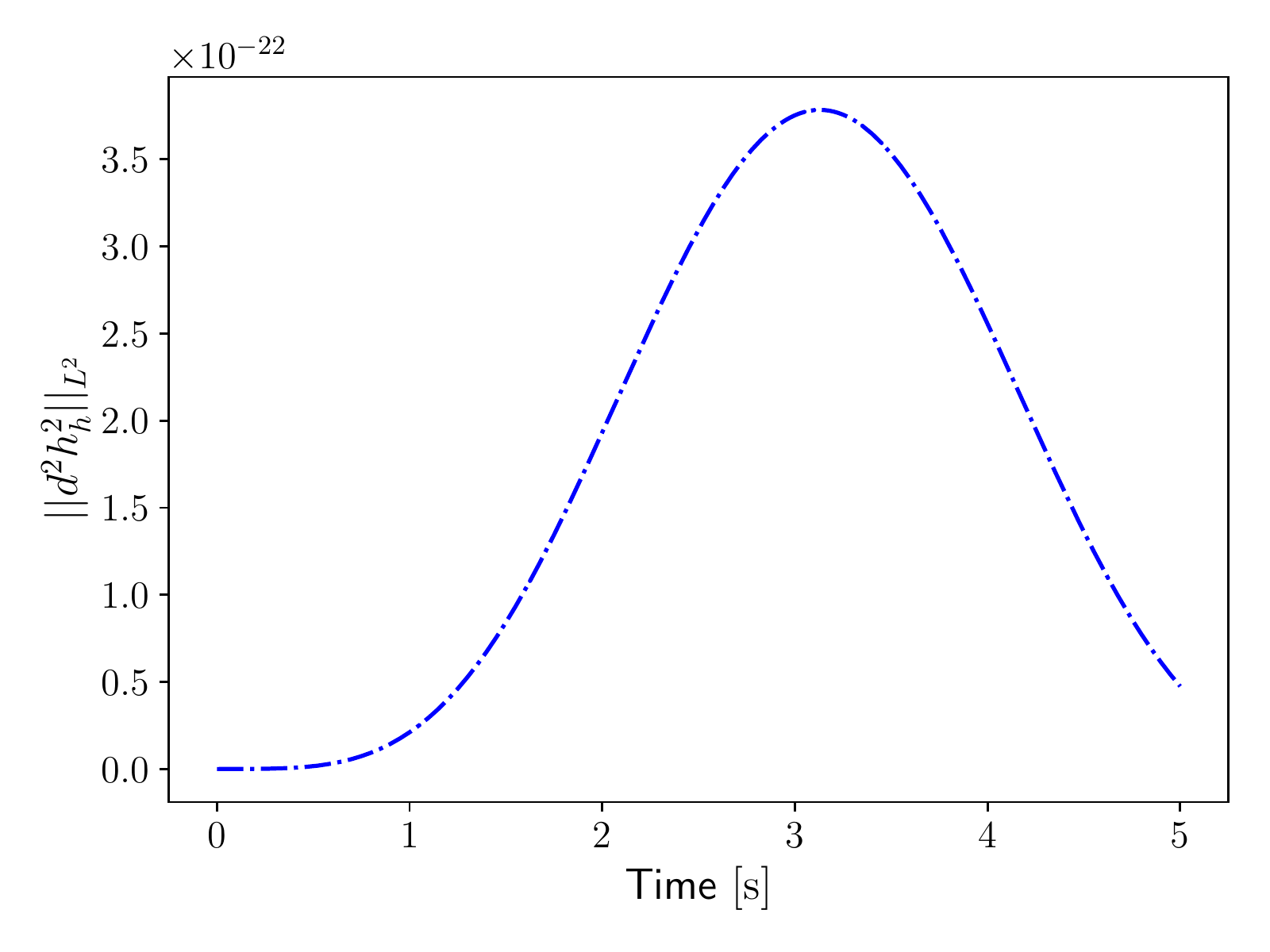}}%
\caption{$L^2$ norm divergence of the two forms $\dual{e}^2_h, h^2_h$}%
\label{fig:div_maxwell}%
\end{figure}

\begin{figure}[p]%
\centering
\subfloat[][Conservation law for $\dot{\dual{H}}^{2}_h-\dual{P}^{2}_h$.]{%
	\label{fig:con_H_E2H1_maxwell}%
	\includegraphics[width=0.48\columnwidth]{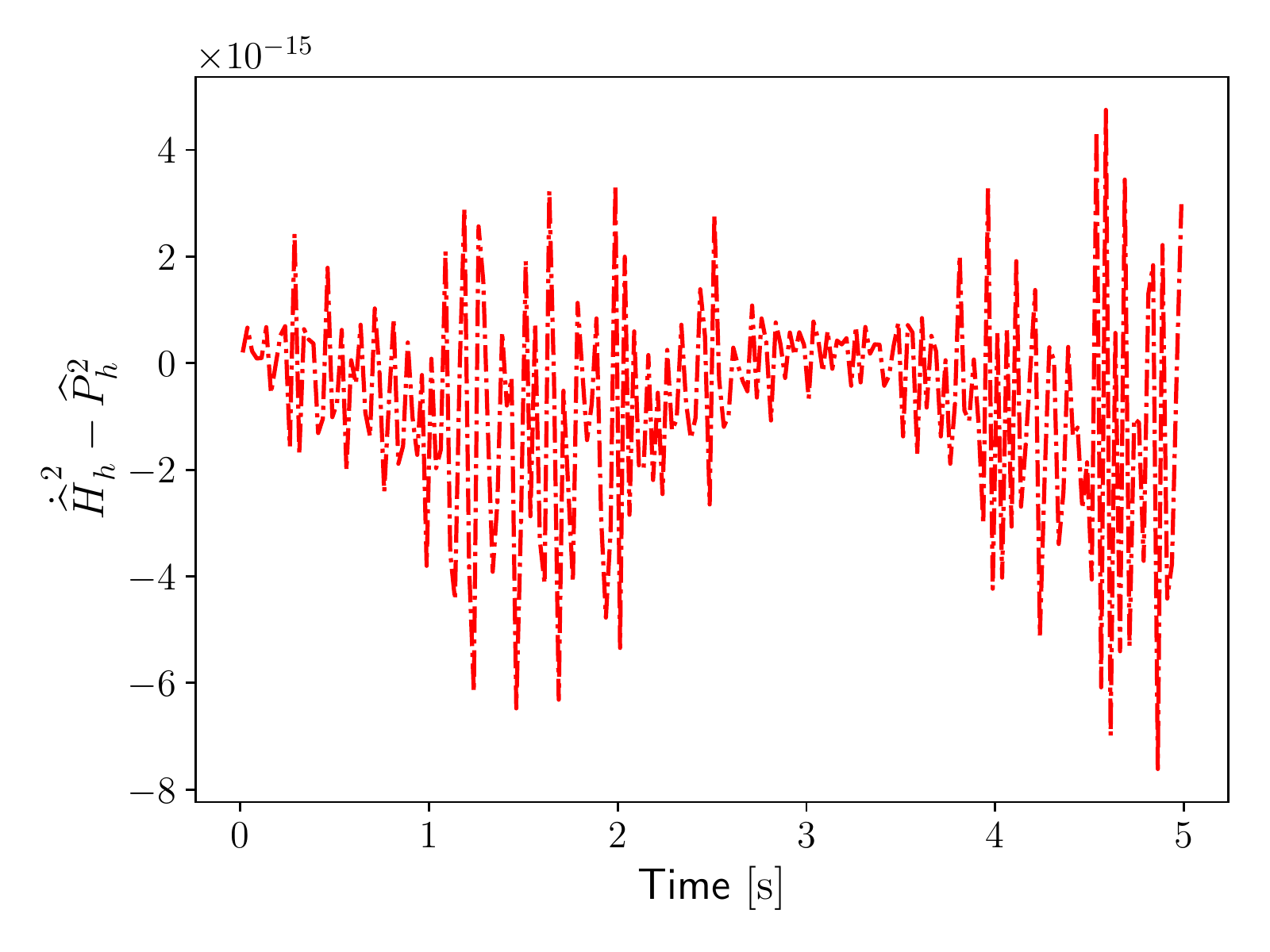}}%
\hspace{8pt}%
\subfloat[][Conservation law for $\dot{H}^{2}_h-P^{2}_h$.]{%
	\label{fig:con_H_H2E1_maxwell}%
	\includegraphics[width=0.48\columnwidth]{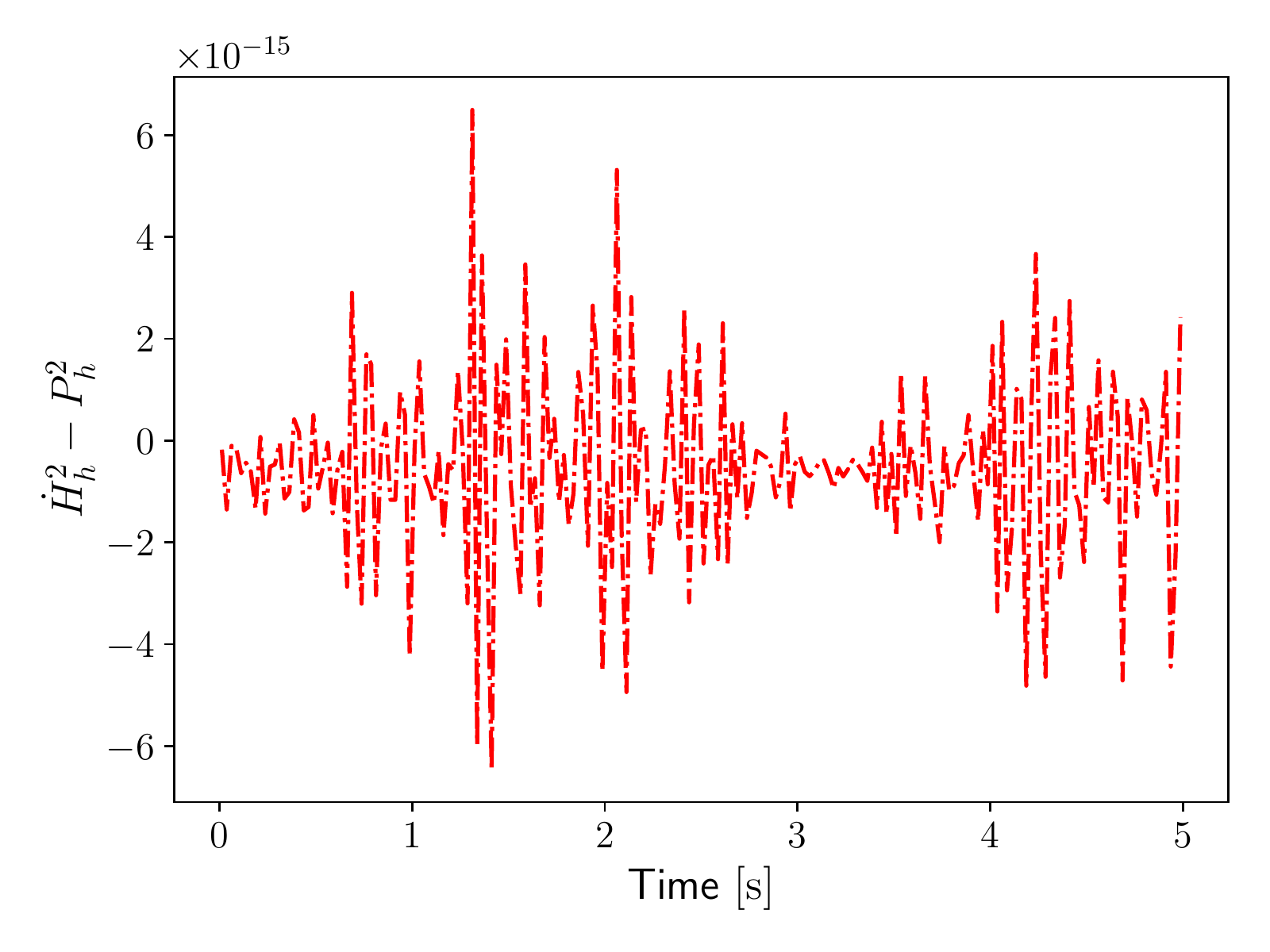}}%
\caption{Conservation properties given by Proposition \ref{pr:discrtime_energyrate} ($N_{\text{el}}=4,\; s=3$ and $\Delta t = \frac{5}{200}$).}%
\label{fig:con_Hdot_maxwell}%
\end{figure}

\begin{figure}[p]%
\centering
\subfloat[][Conservation of $P_h-<\dual{e}_{\partial, h}^1| f_{\partial, h}^1>_{\partial M}$.]{%
	\label{fig:con_P_maxwell}%
	\includegraphics[width=0.48\columnwidth]{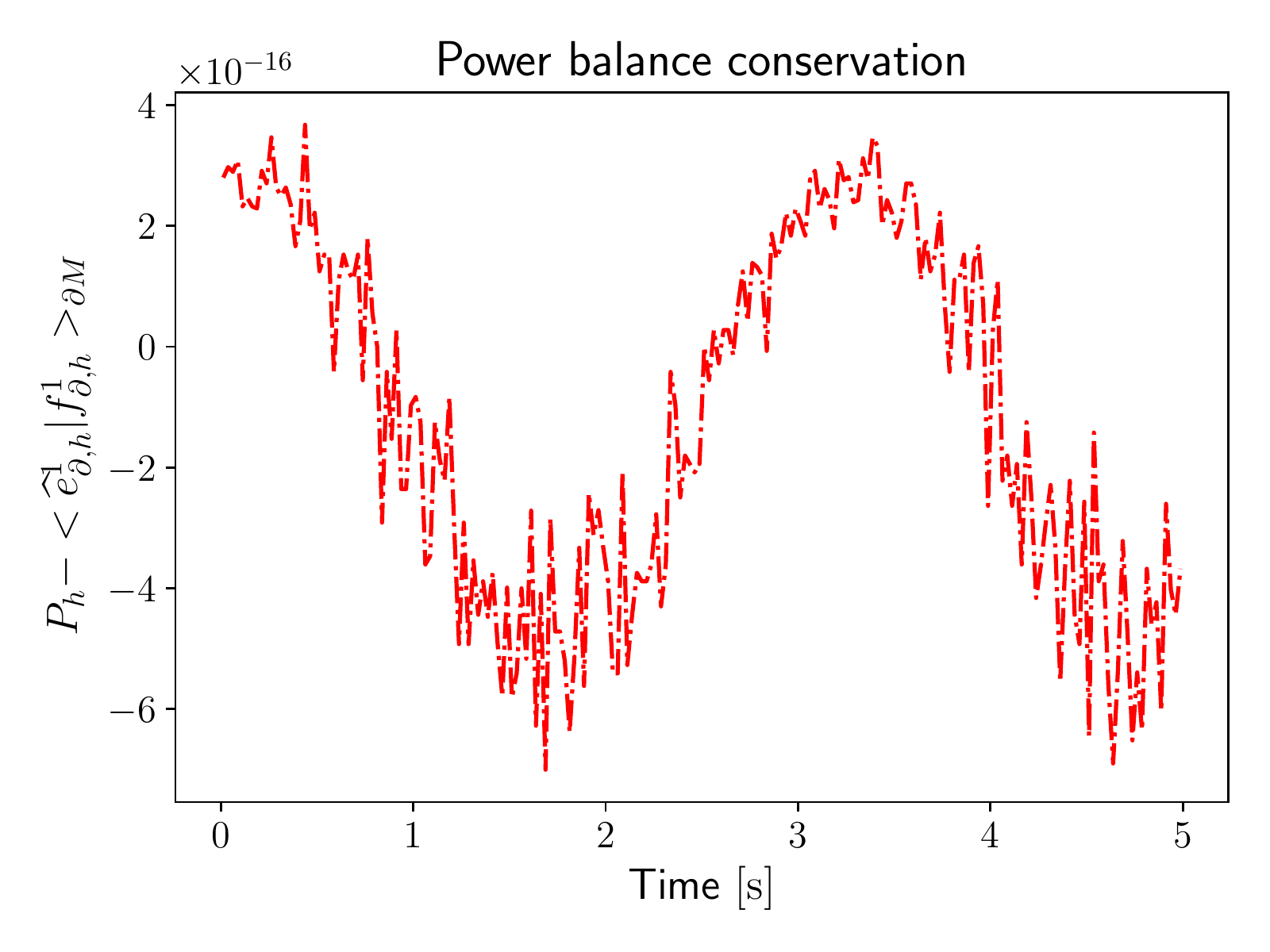}}%
\subfloat[][Error exact and interpolated boundary flow.]{%
\label{fig:err_flow_maxwell}%
\includegraphics[width=0.48\columnwidth]{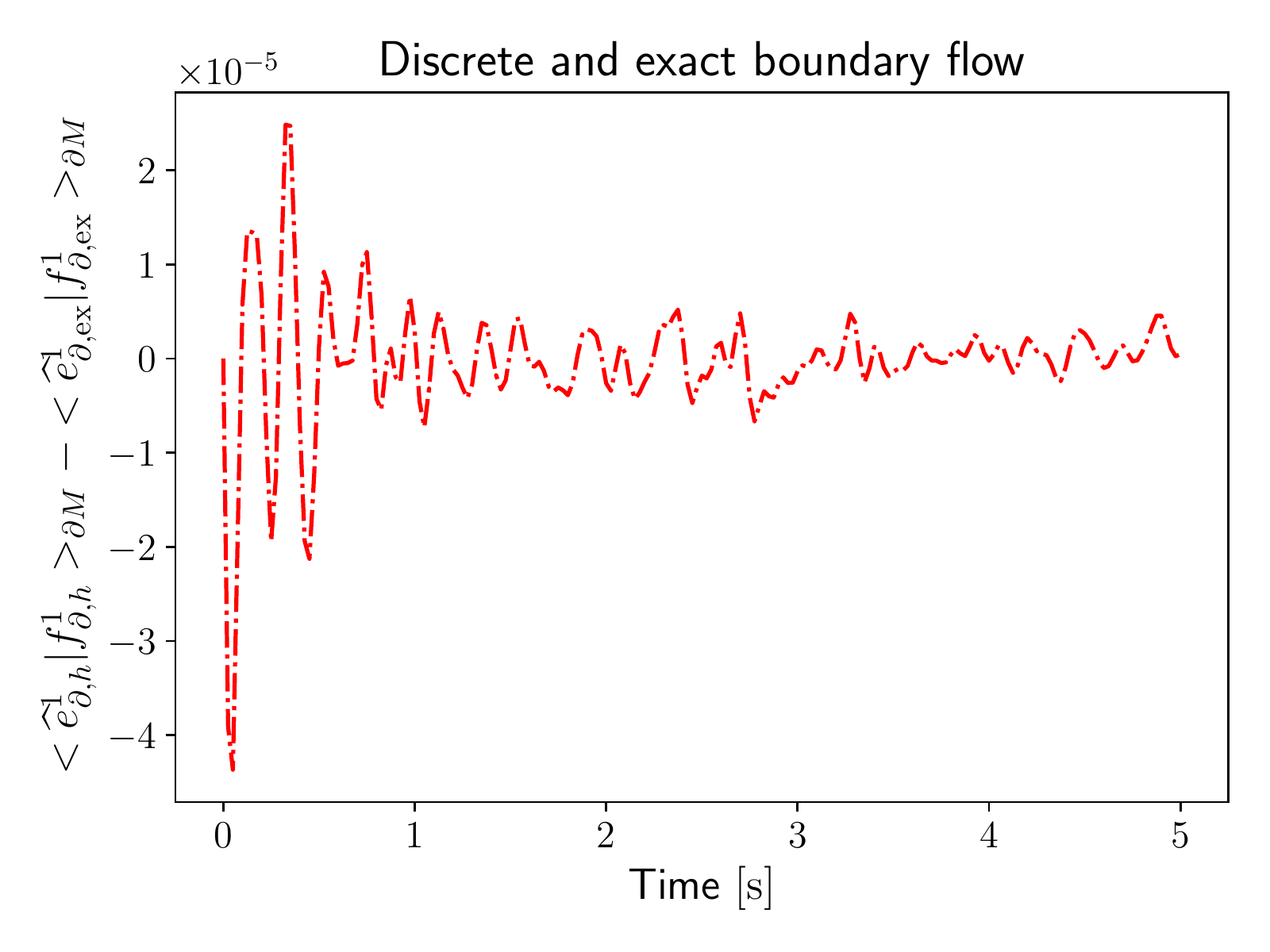}}%
\caption{Power balance given by Pr. \ref{pr:discrtime_Hdot} (left) and error on the power flow (right)  ($N_{\text{el}}=4,\; s=3, \; \Delta t = \frac{5}{200}$).}%
\label{fig:con_pow_maxwell}%
\end{figure}

\begin{figure}[p]%
\centering
\subfloat[][Error $\dot{H}$ and interpolated boundary flow]{%
\label{fig:err_dHdt_maxwell}%
\includegraphics[width=0.48\columnwidth]{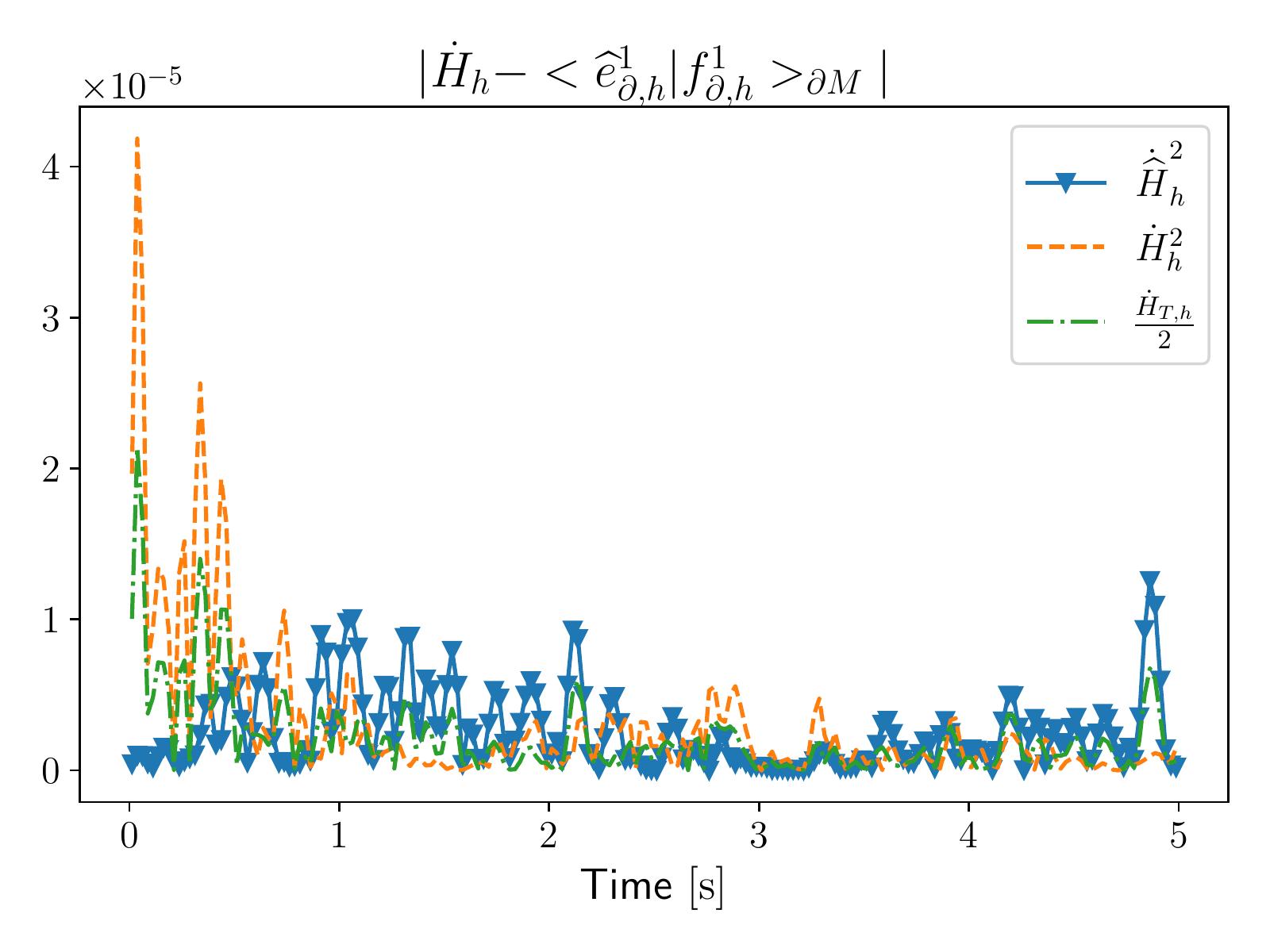}}%
\subfloat[][$\Delta H_h - \Delta H_{\mathrm{ex}}$]{%
	\label{fig:deltaH_maxwell}%
	\includegraphics[width=0.48\columnwidth]{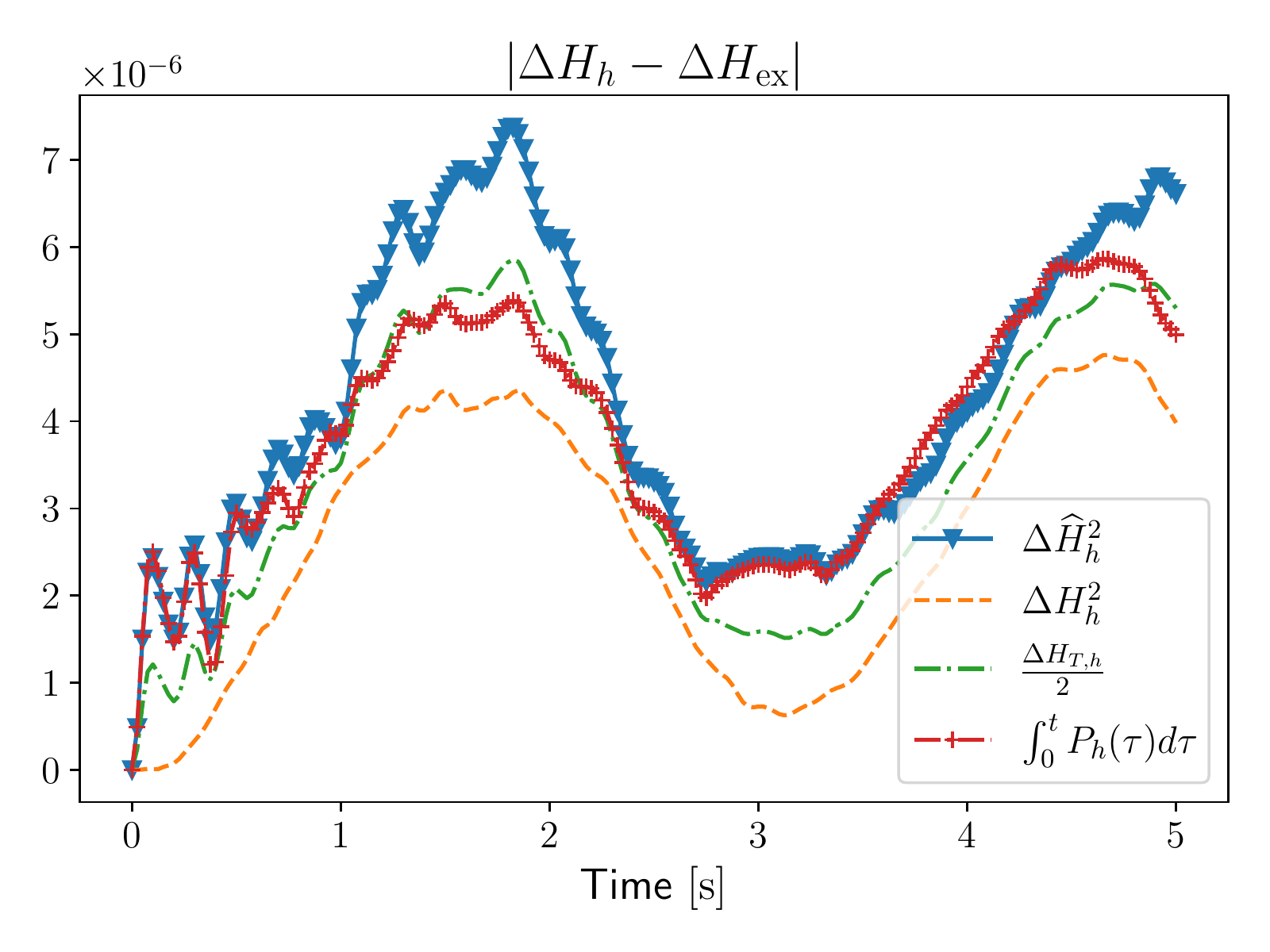}}%
\caption{Energy rate and energy variation error ($N_{\text{el}}=4,\; s=3$ and $\Delta t = \frac{5}{200}$)}%
\label{fig:energy_maxwell}%
\end{figure}

\subsubsection{Convergence results}
The convergence rate of the variables with respect to the exact solution \eqref{eq:maxwell_exsol} is measured in the $L^2$ norm of the error at the final time $T_{\mathrm{end}}=1$ with time step $\Delta t = \frac{1}{100}$. \\

All variables converge with a rate given by $h^s$ (see Fig. \ref{fig:conv_var_maxwell}). However, it appears that for $s=2, 3$ variables $\dual{e}^2_h$ and $h^2_h$ (Figs. \ref{fig:err_E2}, \ref{fig:err_H2}) the convergence rate is a little less than $h^s$. A rigorous error analysis is needed to gain more insight about the observed behaviour. An a priori analysis using a mixed finite element scheme can be found in \cite{asad2019maxwell}. However, therein only homogeneous and uniform boundary conditions are considered.

\begin{figure}[p]%
\centering
\subfloat[][$L^2$ error for $\dual{e}^2_h$]{%
	\label{fig:err_E2}%
	\includegraphics[width=0.48\columnwidth]{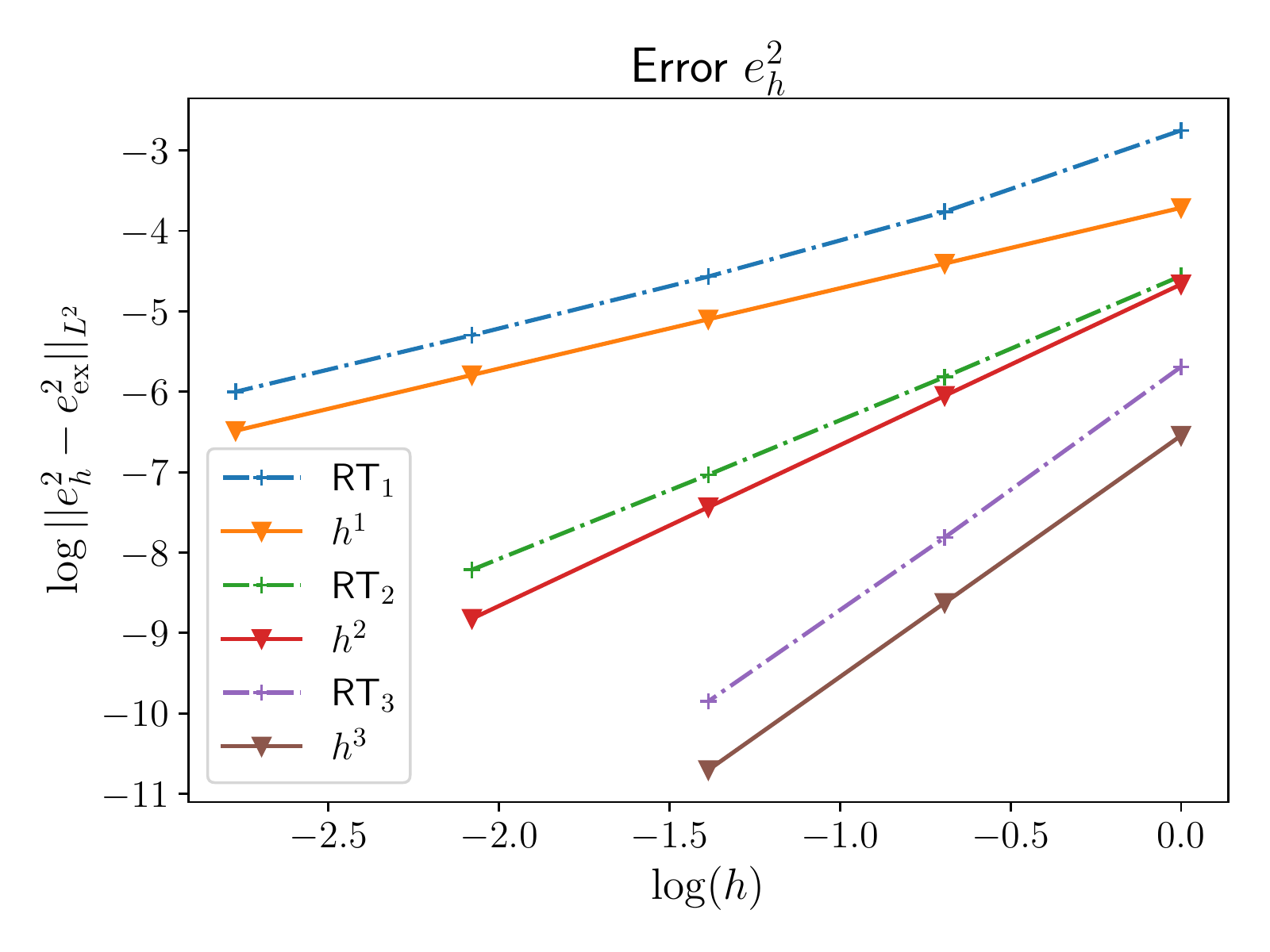}}%
\hspace{8pt}%
\subfloat[][$L^2$ error for $e^1_h$]{%
	\label{fig:err_E1}%
	\includegraphics[width=0.48\columnwidth]{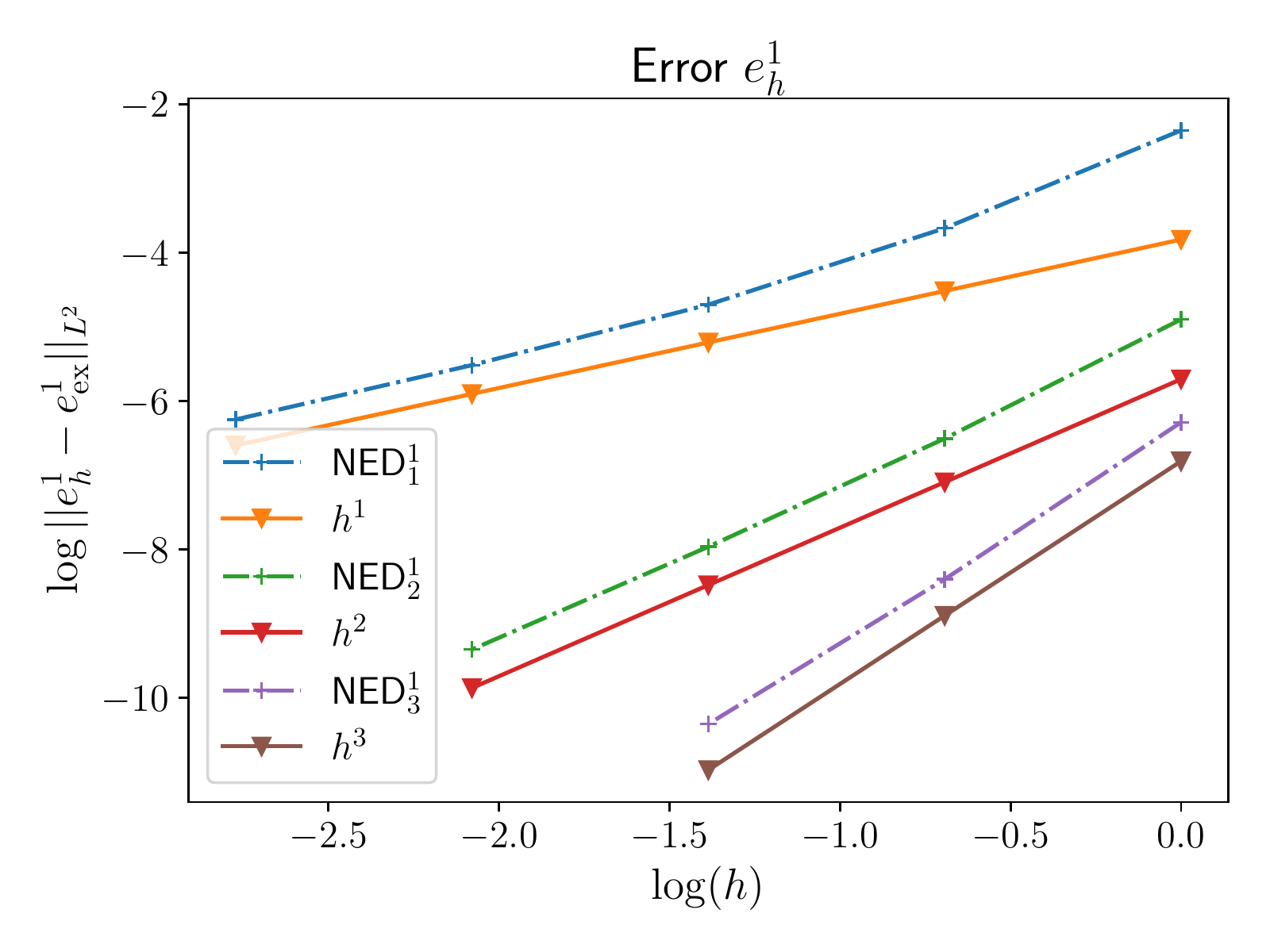}}%
\hspace{8pt}%
\subfloat[][$L^2$ error for $h^2_h$]{%
	\label{fig:err_H2}%
	\includegraphics[width=0.48\columnwidth]{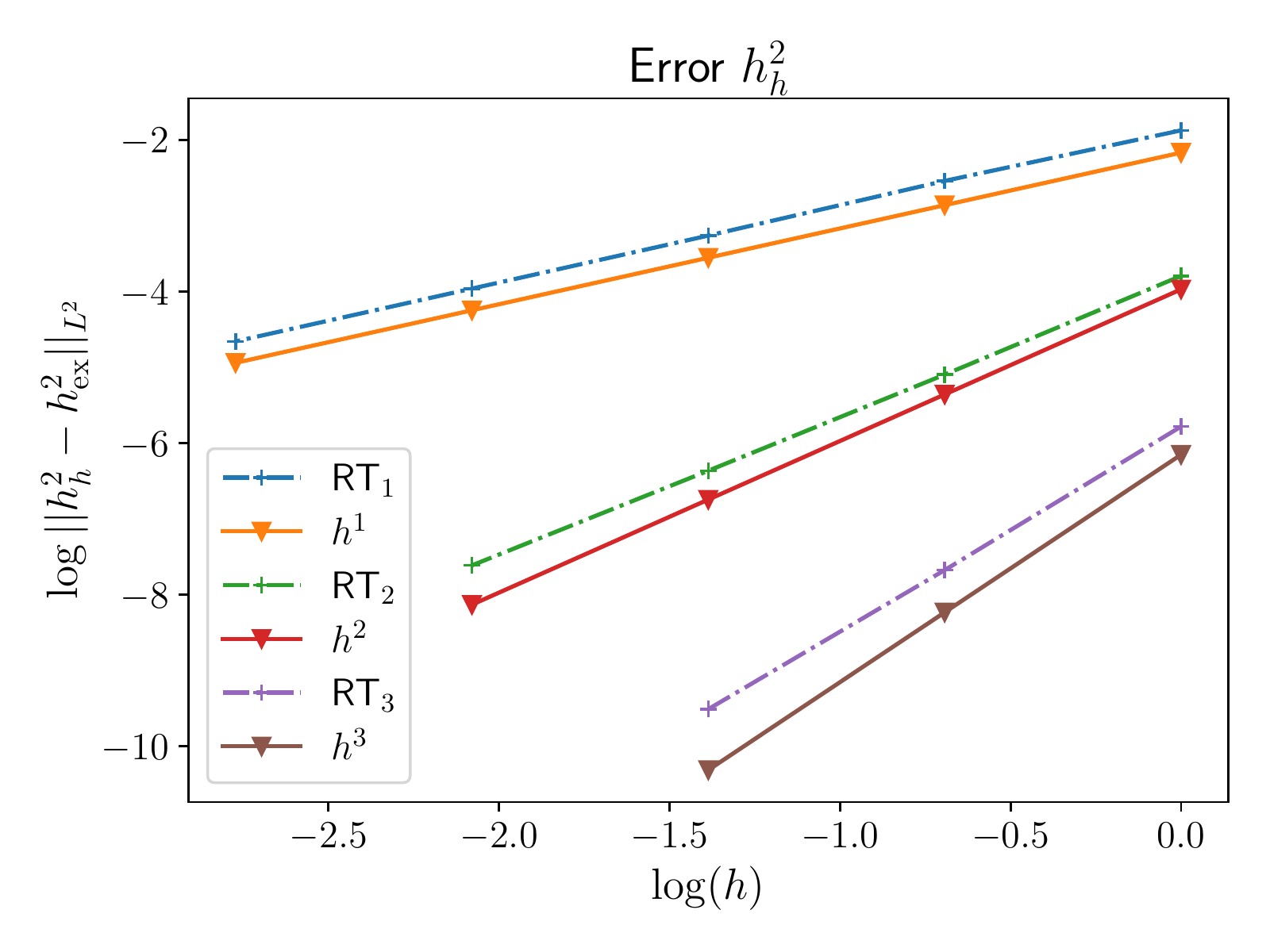}}
	\hspace{8pt}%
\subfloat[][$L^2$ error for $\dual{h}^1_h$]{%
	\label{fig:err_H1}%
	\includegraphics[width=0.48\columnwidth]{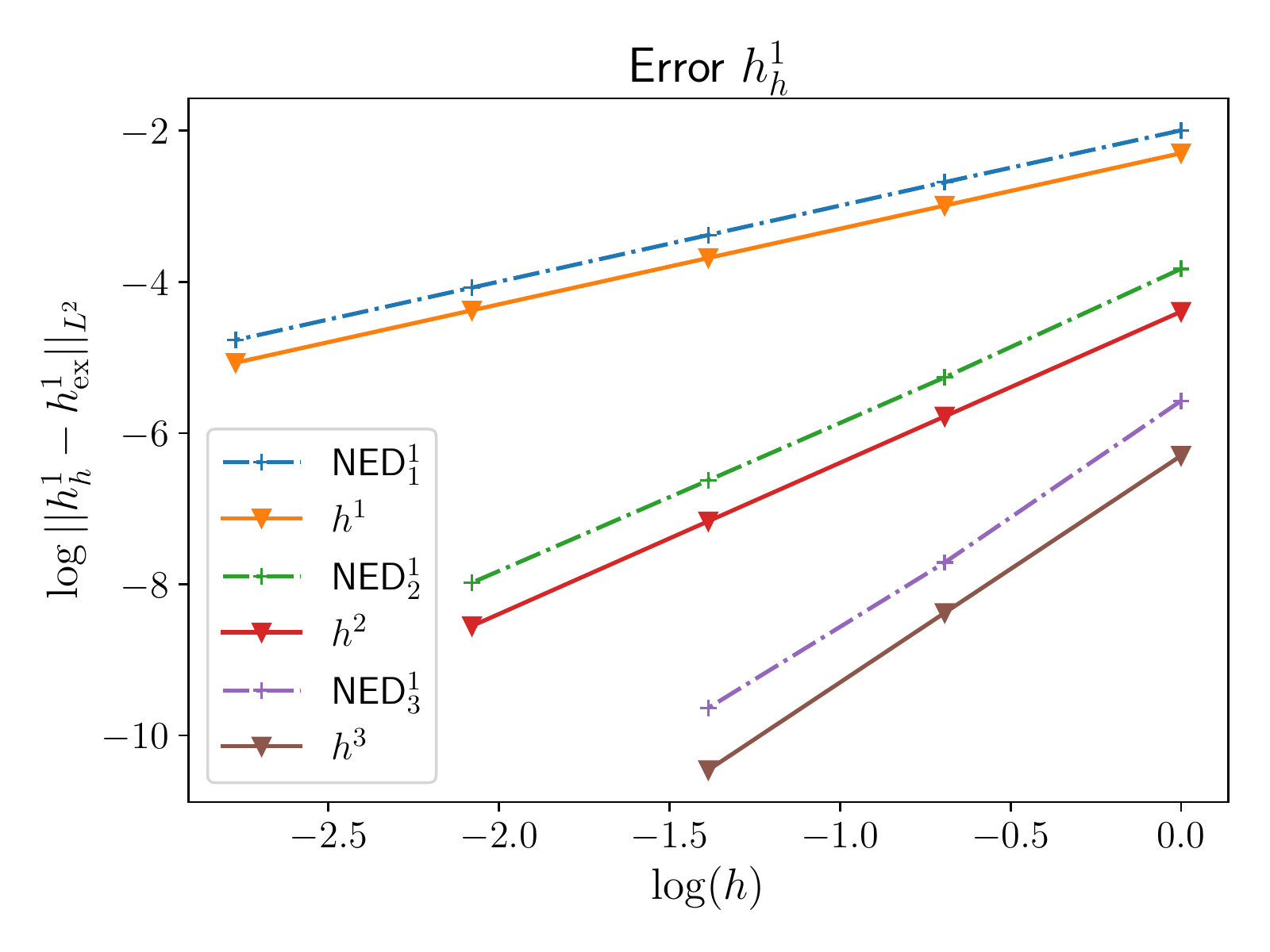}}
\caption{Convergence rate for the different variables in the Maxwell equations at $T_{\text{end}}=1$ for $\Delta t = \frac{1}{100}$.}%
\label{fig:conv_var_maxwell}%
\end{figure}

\begin{figure}[p]%
\centering
\subfloat[][$L^2$ norm of the difference $e^1_h - \dual{e}^2_h$]{%
	\label{fig:diff_E21}%
	\includegraphics[width=0.48\columnwidth]{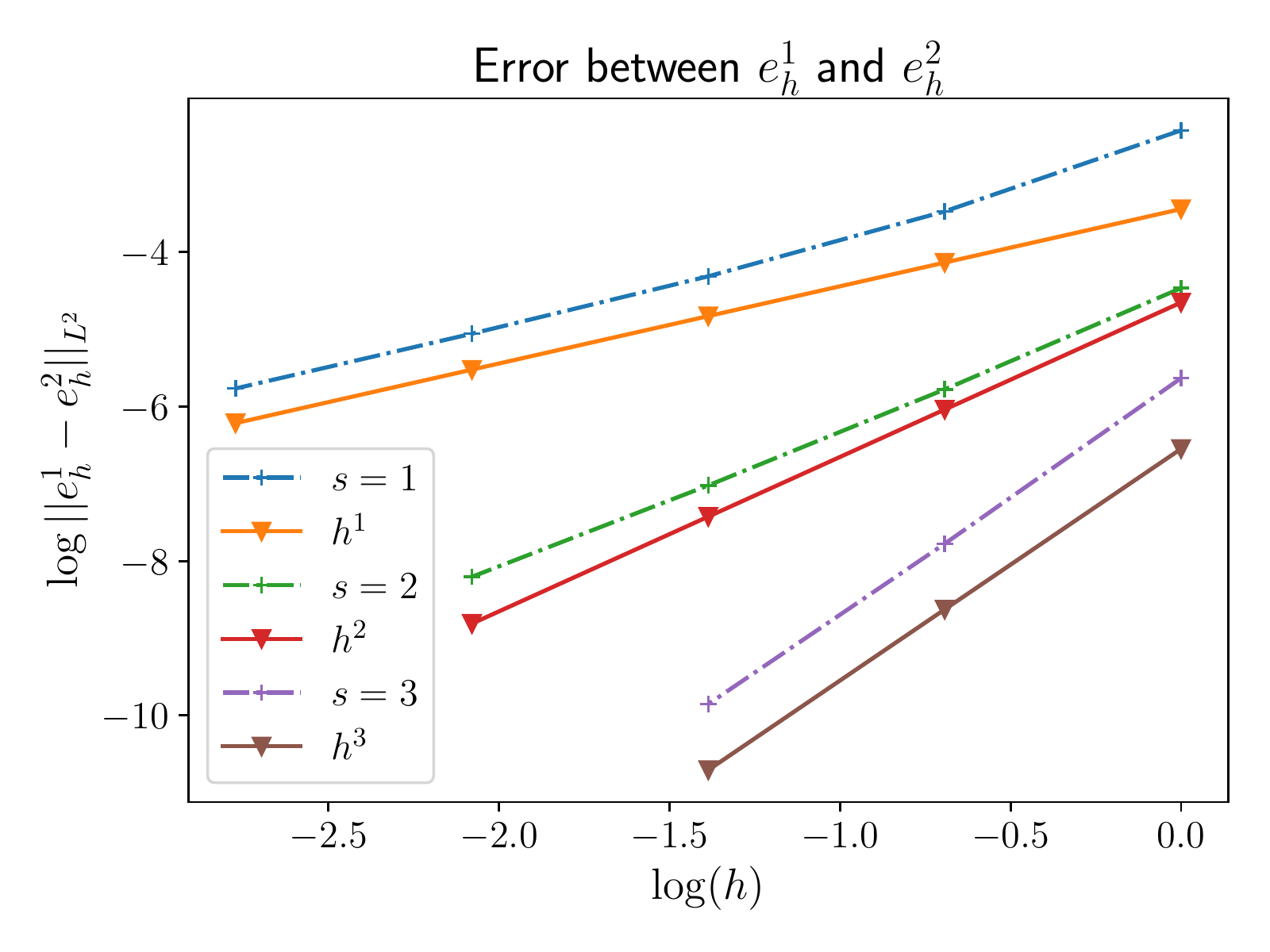}}%
\hspace{8pt}%
\subfloat[][$L^2$ norm of the difference $\dual{h}^1_h - h^2_h$]{%
	\label{fig:diff_H21}%
	\includegraphics[width=0.48\columnwidth]{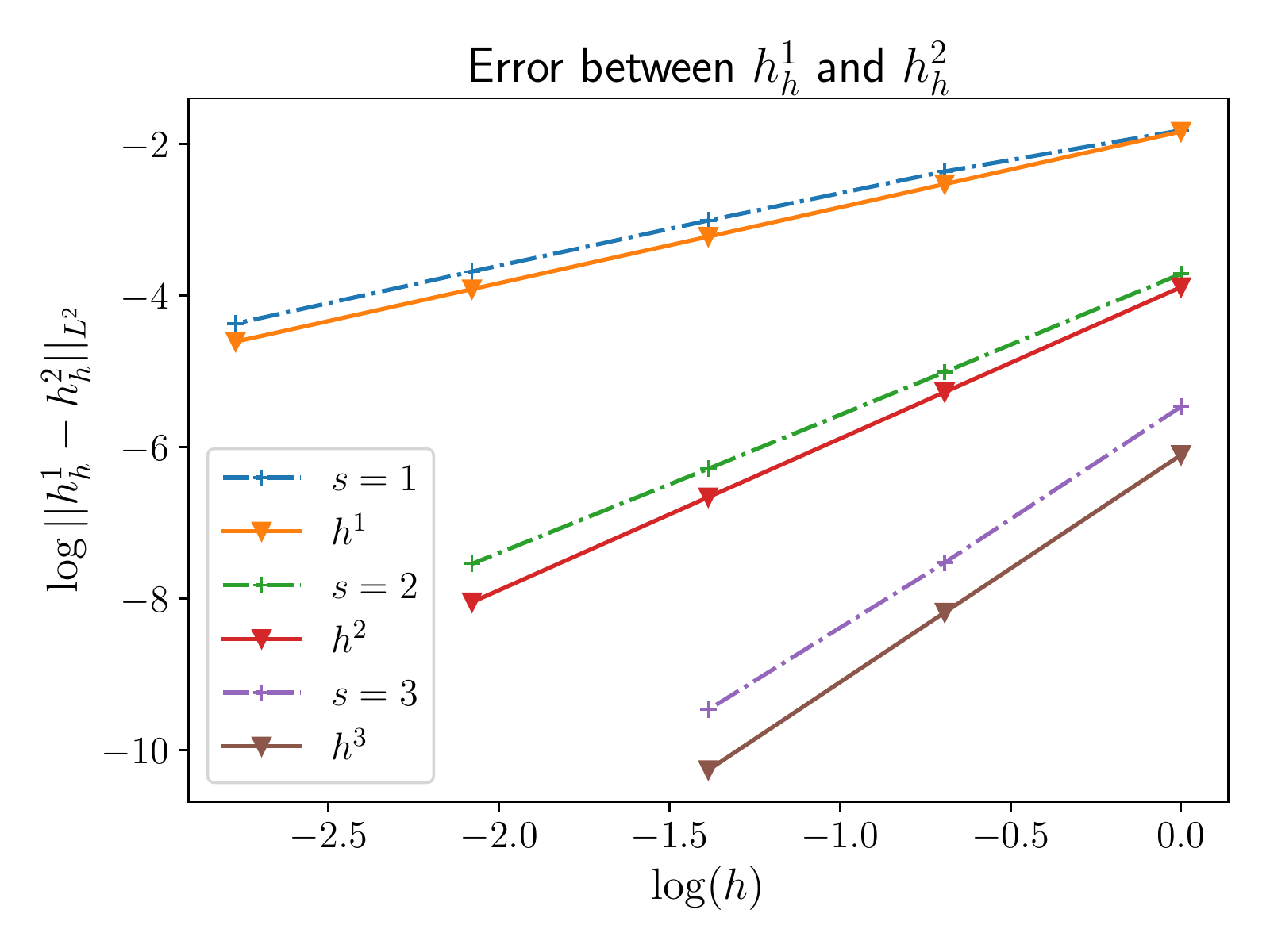}}%
\caption{$L^2$ difference of the dual representation of the solution for the Maxwell Eqs. at $T_{\text{end}}=1$ for $\Delta t = \frac{1}{100}$.}%
\label{fig:diff_dual_maxwell}%
\end{figure}

\section{Additional insights on the choice of the dual variables \label{sec:additionalinsights}}
Finally, we conclude by a short discussion on an alternative way of defining the diffeomorphism $\Phi$ in \eqref{eq:Phi_Diffeo} which relates the original energy variables to the dual ones and gives rise to the adjoint Stokes-Dirac structure.
Differently than presented in Sec. \ref{sec:asds}, we herewith take a definition of the Adjoint Dirac-Structure which will create an interesting symmetry between the primary system and the adjoint system. This will have as a consequence that the material operators $\dual{A}^p, \; B^q$ will then appear in the adjoint Dirac-Structure. This is not useful for discretization purposes, but is much more natural from a physical point of view, because the metrical property of space are strictly related not only to the Hodge, but to the coupling of the Hodge with the material properties.

We herewith then redefine the dual states $(\alpha^{n-p},\; \dual{\beta}^{n-q})$ to achieve a redundant representation of the state space in such a way that the dependencies on both the metric and material properties disappear. We will also redefine the adjoint Hamiltonian $\widetilde{H}(\alpha^{n-p},\dual{\beta}^{n-q})$ and we will do it here in such a way that the sum of the original and adjoint Hamiltonian is independent from the metric and material properties.
\begin{proposition}
Given any $k$-form $\omega^k$ and a symmetric positive definite isomorphism $A^k$ between $k$ forms, if we define  $\dual{\omega}^{n-k}:= \star A^k \omega^k$, and
$\dual{C}^{n-k}:=(-1)^{k(n-k)}\star A^{-1} \star$ (which is symmetric positive definite), we obtain the following relations:
\begin{equation}\label{eq:EcoE}
    \inpr[M]{A^k\omega^k}{\omega^k}=\inpr[M]{\dual{\omega}^{n-k}}{\dual{C}^{n-k}\dual{\omega}^{n-k}}, \qquad \forall \omega^k \in \Omega^k(M),
\end{equation}
and
\begin{equation}\label{eq:noMetric}
    \inpr[M]{A^k \omega^k}{\omega^k}+\inpr[M]{\dual{\omega}^{n-k}}{\dual{C}^{n-k} \dual{\omega}^{n-k}}
    =2\dualpr[M]{\omega^k}{\dual{\omega}^{n-k}}.
\end{equation}
\end{proposition}
\begin{proof}
Considering that $A^k$ is an isomorphism, we have that $\omega^k=(-1)^{k(n-k)}A^{-1} \star \dual{\omega}^{n-k}$. Since the Hodge star is an isometry, it is found
\begin{equation}
\inpr[M]{A^k \omega^k}{\omega^k} = \inpr[M]{\star A^k \omega^k}{\star \omega^k} = \inpr[M]{\dual{\omega}^{n-k}}{\dual{C}^{n-k} \dual{\omega}^{n-k}}
\end{equation}
Furthermore, one has
\begin{equation}
    \inpr[M]{\star A \omega^k}{\star \omega^k} = \dualpr[M]{\omega^k}{\dual{\omega}^{n-k}},
\end{equation}
leading to the second equation.
\end{proof}
We can therefore use the previous results by defining the following flow relations:
\[
{f}^{n-k}_1:=\star \dual{A}^p \dual{f}^p_1
\qquad 
\dual{f}^{n-q}_2:=\star B^q f^q_2, 
\]
\noindent and the pullback of the previous map defines the effort relations:
\[
{e}^{n-p}_1=(-1)^{p(n-p)}\star A^p \dual{e}^p_1,
\qquad
\dual{e}^{n-q}_2=(-1)^{q(n-q)}\star B^q {e}^q_2.
\]
The Stokes-Dirac structure \eqref{eq:StokesDirac} is rewritten in terms of the co-differential map defined in Eq. \eqref{eq:codif}:
\begin{equation}\label{eq:AdjStokesDirac2}
    \begin{pmatrix}
    {f}^{n-p}_1 \\
    \dual{f}^{n-q}_2
    \end{pmatrix} = 
    \begin{bmatrix}
        0 &  (-1)^{r + n(q+1)+1 +q(n-q)}
        {A}^{n-p} \d^* B^q \\
        (-1)^{n(p+1)+1+p(n-p)} \dual{B}^{n-q} \d^* A^p & 0 \\
    \end{bmatrix}
    \begin{pmatrix}
        \dual{e}^{p}_1\\
        {e}^{q}_2
    \end{pmatrix},
\end{equation}
where ${A}^{n-p} = (-1)^{p(n-p)} \star \dual{A}^p \star, \; \dual{B}^{n-q} = (-1)^{q(n-q)} \star A^q \star$. 
If $\dual{A}^p, \; \dual{A}^q$ are not regular enough, then they cannot be differentiated and it is necessary to invert them and bring them to the left side
\begin{equation}
\begin{bmatrix}
    {C}^{n-p} & 0 \\
    0 & \dual{E}^{n-q} \\
\end{bmatrix}
     \begin{pmatrix}
    {f}^{n-p}_1 \\
    \dual{f}^{n-q}_2
    \end{pmatrix} = 
    \begin{bmatrix}
        0 &  (-1)^{r + n(q+1)+1 +q(n-q)}\d^* B^q \\
        (-1)^{n(p+1)+1+p(n-p)} \d^* \dual{A}^p & 0 \\
    \end{bmatrix}
    \begin{pmatrix}
        \dual{e}^{p}_1\\
        {e}^{q}_2
    \end{pmatrix}.
\end{equation}
This system can be then put into weak form considering the integration by parts applied to the codifferential. Relation \eqref{eq:EcoE}  directly gives the representation which  can be used to define the Hamiltonian of the adjoint system, which then represents what in physical system theory is called the co-energy
\begin{equation}
\widetilde{H}({\alpha}^{n-p}, \dual{\beta}^{n-q})  = \int_M \frac{1}{2} {C}^{n-p} {\alpha}^{n-p} \wedge \star {\alpha}^{n-p} + \frac{1}{2}  \dual{B}^{n-q} \dual{\beta}^{n-q}\wedge \star \dual{\beta}^{n-q}.
\end{equation}
 Furthermore \eqref{eq:noMetric} immediately shows that by creating a double representation of the system with a different dual state,  the sum of the energy and co-energy, can be expressed as purely as function of the states and no extra metrical properties
\begin{equation}
   {H}(\dual{\alpha}^p, {\beta}^q)  +\widetilde{H}({\alpha}^{n-p}, \dual{\beta}^{n-q})
   = \int_M  \dual{\alpha}^p\wedge{\alpha}^{n-p}+{\beta}^q\wedge\dual{\beta}^{n-q},
\end{equation}
\noindent which shows that the sum of the energy and co-energy is independent of the metric properties expressed by $A^p$ and $A^q$ achieving a perfect symmetry. It can also be seen that in this alternative definition of the adjoint Dirac-Structure, the metric properties of space represented by the Hodge are always taken together with the physical properties of space represented by the $A$ operators as it would be expected from a physical point of view.
The insights presented above could be instructive in extending our proposed discretization scheme to nonlinear port-Hamiltonian system.

\section{Conclusion}

In this contribution, the dual field formulation is employed for the systematic discretization of linear port-Hamiltonian systems under generic boundary conditions. The proposed methodology is entirely based on the finite element exterior calculus framework. The dual field formulation solves the problems associated with the construction of a discrete Hodge operator (that typically requires dual topological meshes to preserve its isomorphic character) by relying on the adjoint system. The employment of the adjoint system introduces the boundary conditions explicitly by means of the integration by parts formula. This leads to two decoupled mixed discretizations that, once solved, allow retrieving a discrete power balance, regardless of the underlying boundary conditions. This guarantees that the proposed discretization method gives rise to a Dirac structure. This is of crucial importance for multiphysics applications, as well as the design of model-based control laws. \\

This methodology opens the door to a number of interesting developments. As argued in \cite{zhang2021mass}, the employment of dual representations for the unknowns provides useful indicators for adaptive meshing, as one can use the norm of the difference between dual variables as an a posteriori estimator. Furthermore, by using a time staggered discretization as in \cite{zhang2021mass} the boundary conditions could be imposed in a completely weak manner. This would allow the construction of an explicit state-space model, thus avoiding the complications associated with differential algebraic systems. An interesting aspect concerns the usage of algebraic dual polynomials (proposed in \cite{jain2021algdual}) so that dual solutions can be represented in a pair of algebraically dual polynomial spaces.\\

The dual field formulation has been successfully employed to tackle the rotational term of the Navier-Stokes equations in a linear manner, leading to a computationally efficient scheme that conserves mass, energy and helicity. For this reason, we expect a non linear extension of this method to be feasible and competitive with respect to state of the art solutions. An interesting development concerns the extension of the proposed methodology to elasticity problems. These problems require a non trivial extension of the canonical Stokes-Dirac structure, that is based on the de Rham complex, as the differential operators included in the underlying complex, the elasticity complex, are not topological but metrical.

\section*{Funding}

This work was supported by the PortWings project funded by the European Research Council [Grant Agreement No. 787675]

\bibliography{biblio_revision}

\appendix

\section{Proofs}\label{app:proofs}
\revTwo{
\paragraph{Proof the discrete integration by parts formula for the cases of interest}
In this section the formula
\begin{equation*}
    \dualpr[M]{\d\mu_h}{\lambda_h} + (-1)^k\dualpr[M]{\mu_h}{\d\lambda_h} = \dualpr[\partial M ]{\mu_h}{\lambda_h}, \qquad  \forall \mu_h \in \mathcal{V}_{s,h}^k, \; \forall\lambda_h \in \mathcal{V}_{s,h}^{n-k-1},
\end{equation*}
will be proven for the wave and Maxwell equations. Consider for example the wave equation in a 3 dimensional domain $M \subset \bbR^3$, equivalent to the case  $k=0$ in Formula \eqref{eq:int_byparts_d_H}. Using a vector calculus notation (but keeping the actual form degree as exponent), the formula rewrites as follows
\begin{equation*}
    \int_M \{  \grad \mu^0 \cdot \bm{\lambda}^2 + \mu^0 \div \bm{\lambda}^2 \} \; \d M = \int_{\partial M} \mu^0 (\bm{\lambda}^2 \cdot \bm{n}) \d{\Gamma}, \qquad \mu^0 \in H^1(M), \quad \bm{\lambda}^2 \in H^{\div}(M).
\end{equation*}
where $\d\Gamma$ denotes the measure at the boundary $\partial M$. For this example, the discrete counterpart based on the trimmed polynomial family is immediately verified for conforming elements $\mathcal{V}_{s, h}^k \subset H\Omega^k(M)$ since the first variable is in $H^1(M) = H\Omega^0(M)$. Using Continuous Galerkin elements $\mathrm{CG}_s(\mathcal{T}_h) \subset H^1(M)$  for $\mu^0_h$  and Raviart Thomas $\mathrm{RT}_s(\mathcal{T}_h) \in H^{\div}(M)$ for $\bm{\lambda}_h^2$ (where $s$ is the polynomial degree for the finite elements) leads to the following integration by parts when the contribution of each cell of the mesh $T \in \mathcal{T}_h$ is summed up
\begin{equation*}
    \sum_{T \in \mathcal{T}_h}\int_T \{\grad \mu^0_h \cdot \bm{\lambda}^2_h + \mu^0_h \div \bm{\lambda}^2_h\} \;  \d\bm{x} = \sum_{T \in \mathcal{T}_h} \int_{\partial T} \mu^0_h \; (\bm{\lambda}^2_h \cdot \bm{n}) \, \d\bm{s}, \qquad \mu^0_h \in \mathrm{CG}_s(\mathcal{T}_h), \quad \bm{\lambda}^2_h \in \mathrm{RT}_s(\mathcal{T}_h).
\end{equation*}
From the finite elements properties, $\mu^0_h$ is continuous across cells, as well as the normal component of $\bm{\lambda}^2_h$. Therefore, the inter-cell terms of the boundary integral vanish, leading to
\begin{equation}\label{eq:discr_intbyparts_wave}
    \int_M\{\grad \mu^0_h \cdot \bm{\lambda}^2_h + \mu^0_h \div \bm{\lambda}^2_h\} \;  \d{M} =\int_{\partial M} \mu^0_h \; (\bm{\lambda}^2_h \cdot \bm{n})\, \d\Gamma.
\end{equation}
The second case of interest for the paper is the one of the Maxwell equations in 3 dimensional domains $M \subset \bbR^3$, corresponding to the case $k=1$. The integration by parts \eqref{eq:int_byparts_d_H} for this case is written in vector calculus as
\begin{equation*}
    \int_M \{  \curl\bm{\mu}^1 \cdot \bm{\lambda}^1 - \bm{\mu}^1 \cdot\, \curl \bm{\lambda}^1 \} \; \d{M} = \int_{\partial M} \bm{\mu}^1 \cdot (\bm{\lambda}^1 \times \bm{n}) \, \d\Gamma, \qquad \bm{\mu}^1 \in H^1(M, \bbR^3), \; \bm{\lambda}^1 \in H^{\curl}(M),
\end{equation*}
where $H^1(M; \bbR^3) := [H^1(M)]^3$ is the $H^1$ space for vector fields. The same formula  can be rewritten using the tangential trace and the twisted tangential trace as follows \cite[Eq. 27]{buffa2002} 
\begin{equation*}
\int_M \{  \curl \bm{\mu}^1 \cdot \, \bm{\lambda}^1 - \bm{\mu}^1 \cdot \, \curl \bm{\lambda}^1 \} \; \d M = \int_{\partial M}  \{\bm{n} \times (\bm{\mu}^1 \times \bm{n})\} \cdot (\bm{\lambda}^1 \times \bm{n}) \, \d\Gamma.
\end{equation*}
The discrete counterpart based on the trimmed polynomial family then uses Nédélec elements of the first kind NED$_s^1(\mathcal{T}_h) \subset H^{\curl}(M)$ for both $\bm{\mu}^1_h$ and $\bm{\lambda}^1_h$
\begin{equation*}
\sum_{T \in \mathcal{T}_h} \int_T \{  \curl \bm{\mu}^1_h \cdot \bm{\lambda}^1_h - \bm{\mu}^1_h \cdot \curl \bm{\lambda}^1_h \} \; \d\bm{x} = \sum_{T \in \mathcal{T}_h} \int_{\partial T}  \{\bm{n} \times (\bm{\mu}^1_h \times \bm{n})\} \cdot (\bm{\lambda}_h^1 \times \bm{n})\, \d\bm{s},
\end{equation*}
where $\bm{\mu}^1_h, \; \bm{\lambda}^1_h \in \mathrm{NED}_s^1(\mathcal{T}_h)$. Nédélec elements are not $H^1(M,\bbR^3)$ conforming, i.e. $\mathrm{NED}_s^1\not\subset H^1(M, \bbR^3)$. However, their tangential component is continuous across cells. Therefore, the inter-cell terms of the last integral vanishes, leading to
\begin{equation}\label{eq:discr_intbyparts_maxwell}
\int_M \{  \curl \bm{\mu}^1_h \cdot \bm{\lambda}^1_h - \bm{\mu}^1_h \cdot \curl \bm{\lambda}^1_h \} \; \d M = \int_{\partial M}  \{\bm{n} \times (\bm{\mu}^1_h \times \bm{n})\} \cdot (\bm{\lambda}_h^1 \times \bm{n}) \d\Gamma.
\end{equation}
Formulas \eqref{eq:discr_intbyparts_wave} and \eqref{eq:discr_intbyparts_maxwell} demonstrates \eqref{eq:int_byparts_d_disc} for the cases of interest.

\paragraph{Proof of Proposition \ref{pr:parity_a0_a1}}
A necessary and sufficient condition for two coefficients to have the same parity is that their sum is even
\begin{equation*}
    a_0 + 1 \equiv 0 \mod{2}, \qquad  a_1 + 1 + r + p(n-p) + q(n-q) \equiv 0 \mod{2}.
\end{equation*}
Considering that $p+q=n+1$, it is obtained
\begin{equation*}
\begin{aligned}
    a_1 + 1 + r + p(n-p) + q(n-q) &\equiv np + n + 2+ pq+1 + p(n-p) + q(n-q)   \mod{2},\\
            &\equiv np + n + pq + 1+ np -p + nq - q, \mod{2} \\
            &\equiv n + pq +1  -p + nq - q, \mod{2} \\
            &\equiv q(p+n) \mod{2}, \\
            &\equiv q(2n -q +1) \mod{2}, \\
            &\equiv q(q - 1) \equiv 0 \mod{2}.
\end{aligned}
\end{equation*}
where it has been used 
$$p^2\equiv p, \; \mod 2, \qquad q^2\equiv q, \; \mod 2.$$
A similar computation then shows
\begin{equation*}
    a_0 + 1\equiv 0 \mod{2}.
\end{equation*}
}

\section{Vector calculus and differential forms}\label{app:vec_ext}
To illustrate how existing finite element libraries can be used to implement the dual fields discretization, it is important to highlight how exterior and vector calculus are related. Let's assume that $M$ is a three dimensional Riemannian manifolds $\mathrm{dim}(M)=3$with metric $g$ and associated tangent bundle $TM$ and cotangent bundle $T^*M =\Omega^1(M)$. Denoting a generic point in the manifold as
$\xi$, using the metric tensor and the Hodge operator, vector fields can be converted into one-forms or $n-1$-forms and vice-versa.

\begin{definition}[Flat operator]
The flat isomorphism $\flat$ 
\begin{equation}
    \flat: T_\xi M \rightarrow \Omega^1_\xi(M),
\end{equation}
defined by
\begin{equation}
v^\flat(w):= g_\xi(v, w), \qquad \forall w \in T_\xi M,
\end{equation}
converts vector fields into one-forms by using the metric structure of the manifold.
\end{definition}
The inverse operator is called the sharp operator.
\begin{definition}[Sharp operator]
The sharp isomorphism $\sharp$
\begin{equation}
    \sharp: \Omega^1_\xi(M) \rightarrow T_\xi M
\end{equation}
defined by
\begin{equation}
    g_\xi(\omega^\sharp, v) := \omega(v), \qquad \forall w \in T_\xi M,
\end{equation}
converts one forms into vector fields and it is the inverse of the flat operator
\end{definition}
By combining the flat and Hodge one can convert vector fields to $n-1$ forms
\begin{equation}
\begin{aligned}
    \beta: T M &\rightarrow \Omega^{n-1}(M), \\
             v &\rightarrow  \beta(v):= \star v^\flat.
\end{aligned}
\end{equation}
The inverse operator is given by
\begin{equation}
    \beta^{-1}= (-1)^{n+1} \sharp \star.  
\end{equation}

\end{document}